\documentclass[reqno]{amsart}
\usepackage{amssymb,amsmath,amsfonts,latexsym}
\usepackage{enumerate}
\usepackage[all]{xy}
\numberwithin{equation}{section}
\IfFileExists{mathrsfs.sty}{
\usepackage{mathrsfs}}         
   {
     \IfFileExists{eucal.sty}{
        \usepackage[mathscr]{eucal}} 
   {
   }
}

\newcommand{\C}{{\mathbb{C}}}
\newcommand{\complex}{{\mathbb{C}}}
\newcommand{\R}{{\mathbb{R}}}

\newcommand{\N}{{\mathbb{N}}}

\newcommand{\frakB}{\mathfrak{B}}
\newcommand{\frakD}{\mathfrak{D}}

\newcommand{\fraks}{\mathfrak{s}}

\newcommand{\calA}{\mathcal{A}}
\newcommand{\calB}{\mathcal{B}}
\newcommand{\calC}{\mathcal{C}}
\newcommand{\calD}{\mathcal{D}}
\newcommand{\calF}{\mathcal{F}}

\newcommand{\calH}{\mathcal{H}}

\newcommand{\calR}{\mathcal{R}}

\newcommand{\calU}{\mathcal{U}}
\newcommand{\calW}{\mathcal{W}}

\newcommand{\tr}{\operatorname{tr}}
\newcommand{\GL}{\operatorname{GL}}
\newcommand{\op}{\mathrm{op}}
\newcommand{\Hom}{\operatorname{Hom}}
\newcommand{\Ext}{\operatorname{Ext}}
\newcommand{\Tor}{\operatorname{Tor}}
\newcommand{\Barcpl}{\operatorname{Bar}}

\newcommand{\lrBarcpl}{\operatorname{Bar}^\text{\tiny\rm l-red}}
\newcommand{\rrBarcpl}{\operatorname{Bar}^\text{\tiny\rm r-red}}
\newcommand{\Mod}{\operatorname{Mod}}
\newcommand{\im}{\operatorname{im}}

\newcommand{\Sp}{\operatorname{Sp}}

\newcommand{\id}{\operatorname{id}}
\newcommand{\eu}{\operatorname{eu}}

\newcommand{\unA}{{A^\text{\rm \tiny u}}}
\newcommand{\envA}{{A^\text{\rm \tiny e}}}
\newcommand{\hatotimes}{\hat{\otimes}}

\newcommand{\grp}{\mathsf{G}}
\newcommand{\hgrp}{\mathsf{H}}
\newcommand{\supp}{\operatorname{supp}}
\newcommand{\tot}{\operatorname{Tot}^\bullet}

\newcommand{\Stab}{\operatorname{Stab}}
\newcommand{\pr}{\operatorname{pr}}
\newcommand{\gr}{\operatorname{gr}}

\theoremstyle{plain}
        \newtheorem{theorem}{Theorem}[section]
        \newtheorem{lemma}[theorem]{Lemma}
        \newtheorem{proposition}[theorem]{Proposition}
        \newtheorem{corollary}[theorem]{Corollary}
        \theoremstyle{definition}
        \newtheorem{definition}[theorem]{Definition}
        \newtheorem{remark}[theorem]{Remark}
        \newtheorem{example}[theorem]{Example}

\title[Orbifold cup products]{Orbifold cup products and ring structures on Hochschild cohomologies}
\date{May 29th, 2007}
\author{M.J.~Pflaum, H.B.~Posthuma, X.~Tang, \textrm{and} H.-H.~Tseng}
\begin{document}
\begin{abstract}
In this paper we study the Hochschild cohomology ring of
convolution algebras associated to orbifolds, as well as their
deformation quantizations. In the first case the ring structure is given
in terms of a wedge product on twisted polyvectorfields on the inertia
orbifold. After deformation quantization, the ring structure defines a
product on the cohomology of the inertia orbifold. We study the relation
between this product and an $S^1$-equivariant version of the Chen--Ruan
product. In particular, we give a de Rham model for this equivariant
orbifold cohomology.
\end{abstract}
\address{\newline
   Markus J. Pflaum, {\tt pflaum@math.uni-frankfurt.de}\newline
   \indent {\rm Fachbereich Mathematik, Goethe-Universit\"at Frankfurt/Main,
           Germany } \newline
   Hessel Posthuma, {\tt posthuma@math.uu.nl}\newline
   \indent {\rm Mathematical Institute, Utrecht University, Utrecht, the Netherlands} \newline
   Xiang Tang, {\tt xtang@math.wustl.edu}   \newline
   \indent {\rm  Department of Mathematics, Washington University, St.~Louis,
           USA }\newline
   Hsian-Hua Tseng, {\tt hhtseng@math.ubc.ca}\newline
   \indent {\rm Department of Mathematics, University of British
Columbia, Vancouver, Canada}
}
\maketitle
\tableofcontents
\section{Introduction}
In this paper, we study orbifolds within the language of noncommutative
geometry. According to \cite{MoeOG},  an orbifold $X$ can be represented 
by a proper \'etale Lie groupoid $\grp$, and different representations 
of the same orbifold $X$ are Morita equivalent.  
A paradigm from noncommutative geometry tells that one should view the 
groupoid algebra $\calA\rtimes\grp$ of a proper \'etale groupoid $\grp$ 
representing the orbifold $X$ as the ``algebra of functions'' on $X$, 
where $\calA$ is the sheaf of functions on the unit space $\grp_0$ of $\grp$.
Though it is noncommutative, the algebra $\calA\rtimes \grp$
contains much important information of $X$.

We provide in this paper a complete description of the ring structures 
on the Hochschild cohomology of the groupoid algebra $\calA\rtimes
\grp$ and its deformation quantization $\calA^\hbar\rtimes \grp$ when 
$X$ is a symplectic orbifold.  We thus complete projects initiated 
in \cite{TanDQPPG} and \cite{NeuPflPosTanHFDPELG}.
By our results one obtains a cup product on the space of multivector
fields on the inertia orbifold $\widetilde{X}$ associated to $X$,  and
a Frobenius structure on the de Rham cohomology of the inertia
orbifold $\widetilde{X}$. This Frobenius structure is closely
related to the Chen-Ruan orbifold cohomology
\cite{cr-orbifoldCoh}, and inspires us to introduce a de Rham
model for some $S^1$-equivariant Chen-Ruan orbifold cohomology. We
prove that the algebra of the Hochschild cohomology of the
deformation quantization $\calA^\hbar\rtimes \grp$ is
isomorphic to the graded algebra of the Chen-Ruan orbifold
cohomology with respect to a natural filtration.

In this paper, we view the algebra $\calA\rtimes \grp$ as a
bornological algebra with the canonical bornology inherited from
the Frechet topology and  compute its Hochschild
cohomology respect to this bornology.  In
\cite{NeuPflPosTanHFDPELG}, we constructed a vector space isomorphism
\[
 H^\bullet(\calA\rtimes \grp, \calA\rtimes \grp)=
 \Gamma^\infty\Big( \wedge^{\bullet-\ell}T\mathsf{B}_0\otimes
 \wedge^{\ell}N\big)^{\grp},
\]
where $\mathsf{B}_0$ is the space of loops of $\grp$ defined as 
$\{g\in\grp \mid s(g)=t(g) \}$, $\ell$ is a locally constant function on
$\mathsf{B}_0$, namely the codimension of the germ of $\mathsf{B}_0$ inside 
$\grp_1$, and where $N$ is the normal bundle of $\mathsf{B}_0$ in $\grp_1$.  
In this article, we determine the cup product on the Hochschild cochain
$H^\bullet(\calA\rtimes \grp, \calA\rtimes \grp)$.
To do so, we need to understand the maps realizing the above isomorphisms
of vector spaces. In \cite{NeuPflPosTanHFDPELG}, the ring structures got lost
at the end of the final equality, since there we were dealing with a clumsy
chain of quasi-isomorphisms. The first goal of this work is to present a 
sequence of explicit quasi-isomorphisms of differential graded algebras 
preserving cup products.  Some parts of these quasi-isomorphisms 
have already appeared in \cite{NeuPflPosTanHFDPELG}  and \cite{HalTan},  
but in this work we succeeded to put all ingrediants together in the right way
and thus determine the cup products we were looking for. 
The new input consists in the following. Firstly, we introduce a 
complex of fine presheaves $\calH^\bullet$ on $X$ which has a natural
cup product and the global sections of which form a complex quasi-isomorphic
to the Hochschild cohomology complex. Secondly, we use \v{C}ech cohomology
methods  to localize the computation of the cohomology ring of 
$\calH^\bullet (X)$. Thirdly, we use the twisted cocycle construction and 
the local quasi-inverse map $T$ from \cite{HalTan} to compute the cup product
locally. By gluing together the local cup products to a global one we finally
arrive at a transparent computation of the cup product on
$H^\bullet(\calA\rtimes \grp, \calA\rtimes \grp)$. We would like to mention 
that in \cite{rino},  some similar but incomplete results in the local 
situation were obtained.

The above sequence of explicit quasi-isomorphisms opens a way to
compute the Hochschild cohomology of the deformation quantization
$\calA^\hbar\rtimes \grp$, which originally has been constructed in
\cite{TanDQPPG}. In the case of a global quotient, the Hochschild cohomology 
of this algebra has been computed by Dolgushev and Etingof \cite{dolget}
as a vector space using van den Berg duality.  Our method is completely 
different from \cite{dolget} and allows even to determine the ring structure 
on $H^\bullet (\calA^\hbar\rtimes \grp, \calA^\hbar\rtimes \grp)$ in full 
generality. The crucial step in our approach is that we generalize the 
above complex of presheves $\calH^\bullet$ on $X$ and the asssociated
\v{C}ech double complex to the deformed case. With the quasi-isomorphisms for
the undeformed algebra, one can check that there are natural morphisms of 
differential graded algebras from the Hochschild cochain complex of 
$\calA^\hbar\rtimes \grp$ to the presheaf complex 
$\calH_{\text{\tiny \rm loc}, \hbar}^\bullet$ and the associated \v{C}ech 
double complex. We prove these maps to be quasi-isomorphisms by looking at 
the $E_1$ terms of the spectral sequence associated to the $\hbar$-filtration, 
which agrees with the undeformed complexes.  
To perform the local computations, we generalize the Fedosov--Weinstein--Xu 
resolution in \cite{DolHCRCFT} for the computation of the Hochschild 
cohomology of a deformation quantization to the
$\grp$-twisted situation using ideas of Fedosov \cite{FedTIDQ}. We
use essentially an explicit map from the Koszul resolution of the
Weyl algebra to the corresponding Bar resolution by Pinczon
\cite{pinczon}. Our main theorem is that we have a natural
isomorphism of algebras over $\complex((\hbar))$
\[
  H^\bullet\left(\calA^\hbar\rtimes \grp, \calA^\hbar\rtimes
  \grp\right)\cong
  H^{\bullet-\ell}\left(\tilde{X},\C((\hbar))\right),
\]
where the product structure on the right hand side is defined
(cf.~Section \ref{Sec:defhoch}) by
\begin{equation}
 \label{eq:product} 
   [\alpha]\cup[\beta]= \int_{m_\ell} \pr_1^*\alpha \wedge \pr_2^* \beta.
\end{equation}
This generalizes Alvarez's result \cite{alvarez} on the Hochschild cohomology 
ring of the crossed product algebra of a finite group with the Weyl algebra.

The cup product (\ref{eq:product}) together with integration with
respect to the symplectic volume form defines a Frobenius
structure on the de Rham cohomology of the inertia orbifold
$\widetilde{X}$. One notices that there is similarity
between (\ref{eq:product}) and the de Rham model defined by Chen
and Hu \cite{chen-hu:de-rham}. However, Chen and Hu's model was
only defined for abelian orbifolds and works in a formal level. To
connect the Hochschild cohomology of $\calA^\hbar\rtimes \grp$ to
the Chen-Ruan orbifold cohomology, we extend the de Rham model to
an arbitrary almost complex orbifold using methods from equivariant
cohomology theory and \cite{jkk:stringy-k}'s result on obstruction bundles.
More precisely, we prove that the algebra $(H^\bullet(\calA^\hbar\rtimes
\grp,\calA^\hbar\rtimes \grp),\cup)$ is isomorphic to the graded
algebra of the $S^1$-equivariant Chen-Ruan orbifold cohomology
with respect to a natural filtration. In general, the Hochschild
cohomology and the Chen-Ruan orbifold cohomology are not
isomorphic as algebras. By construction, the Chen-Ruan orbifold
cohomology depends on the choice of an almost complex structure,
but the Hochschild cohomology is independent of the choice of
an almost complex or symplectic structure. Therefore, one naturally expects 
that information on the almost complex structure should be contained in the 
filtration on the de Rham model. It is a very interesting question whether one can
detect different almost complex structures through the filtration on the de
Rham model. Our de Rham model and the computation of the Hochschild
cohomology ring of the deformed convolution algebra give more insight to the 
Ginzburg-Kaledin conjecture \cite{GinKal:conj} for hyper-K\"ahler orbifolds. 
Our computations within the differential category suggest that it
is crucial to work in the holomorphic category of deformation
quantization, otherwise conjecture from \cite{GinKal:conj} that there is
an isomorphism between the Hochschild cohomology ring of a deformation quantization 
and the Chen-Ruan orbifold cohomology will in general not be true. 
Concerning the de Rham model for orbifold cohomology let us also mention that 
recently, a similar model has been obtained independantly by R.~Kaufmann 
\cite{KauNTASFO}.

Our paper is organized as follows. In Section \ref{Sec:outline},
we outline the strategy and the main results of this paper, in
Section \ref{Sec:clhoch}, we provide a detailed computation of the
Hochschild cohomology and its ring structure of the algebra
$\calA\rtimes \grp$. Next, in Section \ref{Sec:defhoch}, we
compute the Hochschild cohomology and its ring structure of the
deformed algebra $\calA^\hbar\rtimes \grp$. Then we switch in
Section 5 to orbifold cohomology theory. We introduce a de Rham
model for some $S^1$-equivariant Chen-Ruan orbifold cohomology and
connect this model to the ring structure of the Hochschild
cohomology of the deformed convolution algebra. In the Appendix, we provide 
a full introduction to bornological algebras, their modules and their Morita theory. 
We want to emphasize that the Appendix contains some original results
on Morita equivalence of bornological algebras, which to our knowledge has not
been covered in the literature before. We have chosen to keep these results in the
Appendix to avoid too technical arguments in the main part of our paper.
\vspace{3mm}

\noindent{\bf Acknowledgements:} H.P.~and X.T.~would like to thank
Goethe-Universit\"at Frankfurt/Main for hosting their visits.
M.P., X.T., and H.-H.T.~thank the organizers of the trimester 
``Groupoids and Stacks in Physics and Geometry" for 
hosting their visits of the Institut de Henri Poincar\'e, Paris.
X.T.~would like to thank G.~Pinczon and G.~Halbout for helpful
discussions, M.P. and H.P. thank M.~Crainic for fruitful discussions.
M.P.~acknowledges support by the DFG. H.P. is supported by NWO.
X.T.~acknowledges support by the NSF.


%
%
\section{Outline}
\label{Sec:outline}
As is mentioned above, the main goal of this article is to
determine the ring structure  of the Hochschild cohomology  of a
deformation quantization on a proper \'etale Lie groupoid. In this
section, we outline the strategy to achieve that goal and begin
with a brief overview over the basic notation and results needed
from the theory of groupoids. For further details on the latter we
refer the interested reader to the monograph \cite{MoeMrcIFLG} and
also to our previous article \cite{NeuPflPosTanHFDPELG}.

Recall that a \emph{groupoid} is a small category $\mathsf G$ with set
of objects denoted by $\mathsf G_0$ and set of morphisms by $\mathsf G_1$
such that all morphisms are invertible. The structure maps of a groupoid are
depicted in the diagram
\begin{displaymath}
   \grp_1 \times_{\mathsf {G}_0} \grp_1
   \overset{m}{\rightarrow} \grp_1 \overset{i}{\rightarrow}
   \grp_1 \overset{s}{\underset{t}{\rightrightarrows}}
   \grp_0 \overset{u}{\rightarrow} \grp_1 ,
\end{displaymath}
 where $s$ and $t$ are the source and target map, $m$ is the
 multiplication resp.~composition, $i$ denotes the inverse and finally
 $u$ is the  inclusion of objects by identity morphisms.
 In the most interesting cases, the groupoid carries an additional structure,
 like a topological or differentiable structure. If
 $\grp_1$ and $\grp_0$ are both topological spaces and if
 all structure maps are continuous, then $\grp$ is called a topological
 groupoid. Such a topological groupoid is called \emph{proper},
 if the map $(s,t): \grp_1 \rightarrow \grp_0\times \grp_0$ is a
 proper map, and \emph{\'etale}, if  $s$ and $t$ are both local homeomorphisms.
 In case $\grp_1$ and $\grp_0$ carry the structure of a
 $\calC^\infty$-manifold such that $s,t,m,i$ and $u$ are smooth and
 $s,t$ submersions, then $\grp$ is said to be a \emph{Lie groupoid}.
 \vspace{2mm}

 The situation studied in this article consists of an orbifold represented
 by a proper \'etale Lie groupoid  $\grp$. As a topological space,
 the orbifold coincides with the orbit space $X= \grp_0/\grp$.
 In the following we introduce several sheaves on $\grp$ and $X$.
 By $\calA$ we always denote the sheaf of smooth functions
 on $\grp_0$, and by $\calA \rtimes \grp$ the convolution algebra, i.e.~the
 space $\calC^\infty_\text{\tiny\rm cpt}  (\grp_1)$ together with the
 convolution product $*$ which is defined by the formula
 \begin{equation}
 \label{Eq:ConProd}
   a_1 * a_2 \, (g) = \sum_{g_1 \cdot g_2 =g}
   a_1 (g_1) \, a_2 (g_2) \quad \text{for all
   $a_1,a_2 \in \calC^\infty_\text{\tiny\rm cpt}  (\grp_1)$
   and $g\in \grp_1$}.
 \end{equation}
 Next, let $\omega$ be a $\grp$-invariant symplectic form on $\grp_0$
 and choose a $\grp$-invariant (local) star product $\star$ on $\grp_0$.
 The resulting sheaf of
 deformed algebras of smooth functions will be denoted by $\calA^\hbar$.
 The crossed product algebra  $\calA^\hbar \rtimes \grp$ has the underlying vector
 space $\calC^\infty_\text{\tiny\rm cpt} (\grp_1)[[\hbar]] =
 \Gamma^\infty_\text{\tiny\rm cpt} (\grp_1 , s^* \calA^\hbar ) $ and carries the
 product $\star_\text{\tiny\rm c}$ given by
 \begin{equation}
 \label{Eq:DefConProd}
   [ a_1 \star_\text{\tiny\rm c} a_2 ]_g = \sum_{g_1 \cdot g_2 =g}
   [a_1]_{g_1} g_2\star [a_2]_{g_2}\quad \text{for all
   $a_1,a_2 \in \calC^\infty_\text{\tiny\rm cpt}  (\grp_1)$
   and $g\in \grp_1$}.
 \end{equation}
 Hereby, $[a]_g$ denotes an element of the stalk
 $(s^* \calA^\hbar)_g\cong \calA^\hbar_{s(g)}$, and it has been used that
 $\grp$ acts from the right on the sheaf $\calA^\hbar$
 (see \cite[Sec.~2]{NeuPflPosTanHFDPELG} for details).

 For every  open subset $U\subset X$ define the space $\widetilde{\calA} (U)$
 by
 \begin{displaymath}
   \widetilde{\calA} (U) := (\pi s)_* s^*\calA (U) = \calC^\infty ((\pi s)^{-1}
   (U)) ,
 \end{displaymath}
 where $\pi : \grp_0 \rightarrow X$ is the canonical projection. Denote by
 $\widetilde{\calA}_\text{\tiny\rm fc} (U)$ the subspace
 \begin{displaymath}
   \big\{ f \in  \calC^\infty ((\pi s)^{-1} (U))\mid \supp f
   \cap (\pi s)^{-1} (K) \text{ is compact for all compact $K \subset U$}
   \big\}
 \end{displaymath}
 of all smooth functions on $(\pi s )^{-1}(U)$ with fiberwise compact
 support. Observe that the convolution product $*$ can be extended
 naturally  by Eq.~(\ref{Eq:ConProd}) to each of the spaces
 $\widetilde{\calA}_\text{\tiny\rm fc} (U)$.
 Indeed, since for $K_i := \supp a_i $ with
 $a_i \in  \widetilde{\calA}_\text{\tiny\rm fc} (U)$, $i=1,2$
 the set
 \begin{displaymath}
 \begin{split}
   & m  \big(( K_1 \times K_2 ) \cap (G_1 \times_{G_0} G_1 )
   \cap (\pi s)^{-1} (K) \big) \\
   & = m  \left( \big( ( K_1 \cap (\pi s)^{-1} (K) )
   \times ( K_2  \cap (\pi s)^{-1} (K)) \big)
   \cap (G_1 \times_{G_0} G_1) \right)
 \end{split}
 \end{displaymath}
 is compact by assumption on the $a_i$, the product $a_1 * a_2$ is
 well-defined and lies again in $\widetilde{\calA}_\text{\tiny\rm fc} (U)$.
 Hence, the spaces $\widetilde{\calA}_\text{\tiny\rm fc} (U)$ all carry the
 structure of an algebra and form the sectional spaces of a sheaf
 $\widetilde{\calA}_\text{\tiny\rm fc}$ on $X$. Likewise, one constructs the
 sheaf $\widetilde{\calA}_\text{\tiny\rm fc}^\hbar$.
 Finally note that the natural maps
 $\calA \rtimes\grp\hookrightarrow\widetilde{\calA}_\text{\tiny\rm fc} (X)$
 and
 $\calA^\hbar \rtimes\grp\hookrightarrow
 \widetilde{\calA}_\text{\tiny\rm fc}^\hbar (X)$
 are both algebra homomorphisms.

 From Appendix \ref{Sec:bcalgebras} one can derive the following result. \\[2mm]
 {\bf Theorem O}. \emph{The algebras $\calA \rtimes \grp$ and
 $\calA^\hbar \rtimes \grp$  carry in a natural way the structure of  a
 bornological algebra and are both quasi-unital. Likewise, the sheaves
 $\widetilde{\calA}_\text{\tiny\rm fc}$ and
 $\widetilde{\calA}_\text{\tiny\rm fc}^\hbar $
 are sheaves of quasi-unital bornological algebras.
 Moreover, the natural homomorphisms
 $\calA \rtimes \grp \hookrightarrow \widetilde{\calA}_\text{\tiny\rm fc} (X)$
 and
 $\calA^\hbar \rtimes \grp \hookrightarrow
 \widetilde{\calA}_\text{\tiny\rm fc}^\hbar (X)$ are bounded.}
 \begin{proof}
  Prop.~\ref{Prop:ConBorAlg} and Prop.~\ref{Prop:DefConBorAlg}
  in the appendix show that $\calA \rtimes \grp$ and $\calA^\hbar \rtimes \grp$
  are quasi-unital bornological algebras. By exactly the same methods
  as in there one shows that $\widetilde{\calA}_\text{\tiny\rm fc}$ and
  $\widetilde{\calA}_\text{\tiny\rm fc}^\hbar $
  are sheaves of quasi-unital bornological algebras.
  That the homomorphisms in Theorem O.~are bounded is straightforward.
 \end{proof}

 According to Appendix \ref{Sec:HHBR}, Theorem O implies in particular that
 each one of the algebras in the claim is H-unital and that the Bar
 complex provides a projective resolution.
 This will be the starting point for proving that several of the
 chain maps constructed in the following steps are indeed quasi-isomorphisms.
 \vspace{2mm}

 To formulate the next step, consider the Hochschild cochain complex
 (see \ref{Sec:HHBR})
 \begin{displaymath}
  C^\bullet (\calA \rtimes \grp , \calA \rtimes \grp ) :=
  \Hom_{(\calA \rtimes \grp)^\text{\tiny\rm e}} (\Barcpl_\bullet
  ( \calA \rtimes \grp ),\calA \rtimes \grp ),
 \end{displaymath}
 where $(\calA \rtimes \grp)^\text{\tiny\rm e}$
 is the enveloping algebra (see Sec.~\ref{Sec:HHBR}),
 and define for each open $U \subset X$ the bornological space
 $\calH_\grp^k (U)$ by
 \begin{displaymath}
   \calH_\grp^k (U)  :=
   \Hom \big( (\calA_{|U} \rtimes \grp_{|U})^{\hatotimes k} ,
   \widetilde{\calA}_\text{\tiny\rm fc} (U) \big),
 \end{displaymath}
 where $\grp_{|U}$ is the groupoid with object set
 ${\grp_{|U}}_0 = \pi^{-1} (U)$ and morphism set
 ${\grp_{|U}}_1 = (\pi s)^{-1} (U)$ and where $\calA_{|U}$ is the sheaf of
 smooth functions on $\pi^{-1} (U)$.
 Obviously, the spaces $\calH_\grp^k (U)$ form the sectional spaces of a
 presheaf $\calH_\grp^k$ on $X$ which we denote by  $\calH^k$ if no confusion
 can arise.
 The Hochschild coboundary map $\beta := b^*$ on
 $C^\bullet (\calA_{|U} \rtimes \grp_{|U} , \calA_{|U} \rtimes \grp_{|U} )$
 extends to a coboundary map $\beta$
 on $\calH^\bullet (U)$ by the following definition:
 \begin{displaymath}
 \begin{split}
   \beta F \, & (a_1 \otimes \ldots \otimes a_{k} ) :=
   a_1 \, F( a_2 \otimes \ldots \otimes a_k ) + \\
   & + \sum_{i=2}^{k-1} \, (-1)^{i+1} \,
   F (a_1 \otimes \ldots \otimes a_i \, a_{i+1}\otimes \ldots \otimes a_k ) \\
   & + F( a_1 \otimes \ldots \otimes a_{k-1} ) a_k ,
   \text{ for $F\in \calH^k (U)$ and
   $a_1, \ldots , a_k \in \calC^\infty_\text{\tiny\rm cpt} ((\pi s)^{-1}(U))$.}
 \end{split}
 \end{displaymath}
 Moreover, there is a product
 $\cup : \calH^\bullet (U) \times \calH^\bullet (U) \rightarrow
 \calH^\bullet (U) $, which is called the \emph{cup product} on
 $\calH^\bullet (U)$ and which is given as follows:
 \begin{displaymath}
 \begin{split}
   \cup : \: & \calH^k (U) \times \calH^l (U) \rightarrow \calH^{k+l}(U),\quad
   (F,G) \mapsto F\cup G , \\
   & F\cup G (a_1 \otimes \cdots \otimes a_{k+l} ) :=
   F(a_1 \otimes\cdots \otimes a_k)\, G(a_{k+1}\otimes\cdots\otimes a_{k+l})\\
   & \text{for $a_1, \cdots , a_{k+l} \in\calC^\infty_\text{\tiny\rm cpt}
   ((\pi s)^{-1} (U))$}.
 \end{split}
 \end{displaymath}
 It is straightforward to check that the cup product is associative and
 passes down to the cohomology of $\calH^\bullet (U)$.

 The compatibility between the Hochschild cohomology ring of
 $\calA \rtimes \grp$ and the ring structure on the cohomology of
 $\calH^\bullet$ is expressed by the following step and will be proved in
 Section \ref{Sec:locmeth}.
 \\[2mm]
 {\bf Theorem I}. \emph{The canonical embedding
 \begin{displaymath}
   \iota: \:  C^\bullet (\calA \rtimes \grp , \calA \rtimes \grp )
   \rightarrow \calH^\bullet (X)
 \end{displaymath}
 is a quasi-isomorphism which preserves cup products.}
 \vspace{2mm}

 Since $\calH^\bullet$ is a complex of presheaves on $X$, one can use
 localization techniques for the computation of its cohomology ring. Thus,
 methods from \v{C}ech cohomology theory come into play. To make these ideas
 precise let
 $\calU$ be an open cover of the orbit space $X$, and denote by
 $\check{\calH}^{\bullet,\bullet}_\calU :=
 \check{\calH}_{\grp,\calU}^{\bullet,\bullet} :=
 \check{C}^\bullet_\calU (\calH_\grp^\bullet)$
 the \v{C}ech double complex associated to the presheaf complex
 $\calH^\bullet_\grp$. This means that
 \begin{displaymath}
    \check{\calH}^{p,q}_\calU := \check{C}^q_\calU (\calH^p ) :=
    \prod_{(U_0,\cdots,U_q)\in\calU^{q+1}}\calH^p (U_0\cap \cdots \cap U_q ) .
 \end{displaymath}
 The coboundaries on $\check{\calH}^{\bullet,\bullet}_\calU$ are given  in
 $p$-direction by the Hochschild coboundary
 \begin{displaymath}
  \beta : \check{\calH}^{p,q}_\calU \rightarrow \check{\calH}^{p+1,q}_\calU
 \end{displaymath}
 and, in $q$-direction, by the \v{C}ech coboundary
 \begin{displaymath}
 \begin{split}
   \delta : \check{\calH}^{p,q}_\calU\rightarrow\check{\calH}^{p,q+1}_\calU ,\:
   & \big( H_{(U_0,\cdots,U_q)} \big)_{(U_0,\cdots,U_q) \in \calU^{q+1}}
   \mapsto \\
   & \Big( \sum_{0 \leq i \leq q+1} (-1)^i
   {H_{(U_0,\cdots,\widehat{U_i}, \cdots , U_{q+1})}}_{|U_0 \cap \cdots \cap
   U_{q+1}}
   \Big)_{(U_0,\cdots,U_{q+1}) \in \calU^{q+2}}.
 \end{split}
 \end{displaymath}
 The cohomology of the double complex $\check{\calH}^{\bullet,\bullet}_\calU$,
 i.e.~the cohomology of the total complex
 $\tot_\oplus (\check{\calH}^{\bullet,\bullet}_\calU)$,
 will be denoted by $\check{H}^\bullet_\calU (\calH^\bullet)$.
 The inductive limit
 \begin{displaymath}
   \check{H}^\bullet (\calH^\bullet) := \lim_{\underset{\calU}{\longleftarrow}}
   \check{H}^\bullet_\calU (\calH^\bullet) ,
 \end{displaymath}
 where $\calU$ runs through the set of open covers of $X$,
 then is the \emph{\v{C}ech cohomology} of  the presheaf complex
 $\calH^\bullet$.
 The crucial claim, which will be proved in Section \ref{Sec:locmeth} as well,
 now is the following.
 \\[2mm]
 {\bf Theorem II}.
 \emph{The presheaves $\calH^p$ are all fine, hence the \v{C}ech cohomology
 of  the presheaf complex $\calH^\bullet$ is concentrated in degree $q=0$,
 i.e.~$\check{H}^\bullet_\calU (\calH^\bullet)$ is canonically isomorphic
 to the cohomology of the cochain complex $\calH^{\bullet,0}_\calU$.
 Moreover, the \v{C}ech cohomology $\check{H}^\bullet (\calH^\bullet)$ is given
 by the global sections of a cohomology sheaf on $X$.
 Finally, for each sufficiently fine and locally finite open covering
 $\calU$ of $X$ the canonical chain map
 \begin{displaymath}
   \calH^p (X) \rightarrow \check{Z}^{p,0}_\calU (\calH^\bullet),
   \quad H \mapsto (H_{|U})_{U\in \calU}
 \end{displaymath}
 is a quasi-isomorphism,
 where $\check{Z}^{p,0}_\calU (\calH^\bullet) := \big\{
 H= (H_U)_{U\in \calU} \in \check{\calH}^{p,0}_\calU \mid \delta H =0 \big\}$.
 }
 \vspace{2mm}

By Theorem II one only needs to compute the cohomology of the complexes
$\calH^\bullet (U)$ for all elements $U$ of a sufficiently fine open covering
of $X$. This is the purpose of the following steps.

Let us consider now a weak equivalence of proper \'etale groupoids
$\varphi: \hgrp \hookrightarrow \grp$. Assume further that $\varphi$
is an open embedding and denote by $\calH_\hgrp^\bullet$ and
$\calH_\grp^\bullet$ the  complexes of presheaves as defined above.
Then $\varphi$  induces  a bounded linear map
$\varphi_* : \calC^\infty_\text{\tiny\rm cpt} (\hgrp_1)\rightarrow
\calC^\infty_\text{\tiny\rm cpt} (\grp_1)$ by putting
for $a \in  \calC^\infty_\text{\tiny\rm cpt} (\hgrp_1)$, $g \in \grp_1 $
\begin{displaymath}
\begin{split}
  \varphi_* (a) (g) =
  \begin{cases}
    a \circ \varphi^{-1} (g), & \text{if $g\in \im \varphi$}, \\
    0, & \text{else}.
  \end{cases}
\end{split}
\end{displaymath}
Moreover, one obtains bounded chain maps
\begin{displaymath}
\begin{split}
  \varphi^* : & \: C^\bullet (\calA \rtimes \grp ,\calA \rtimes \grp)
  \rightarrow  C^\bullet (\calA \rtimes \hgrp ,\calA \rtimes \hgrp),
  \quad F \mapsto \varphi^* (F) \quad \text{ and }\\
  \varphi^* : & \: \calH_\grp^\bullet \rightarrow \calH_\hgrp^\bullet ,
  \quad F \mapsto \varphi^* (F),
\end{split}
\end{displaymath}
where in both cases $\varphi^* (F)$ is defined by
\begin{displaymath}
  \varphi^* (F)(a_1 \otimes \cdots \otimes a_k) =
  F \big( \varphi_* a_1 \otimes \cdots\varphi_* a_k \big) \circ \varphi
  \quad \text{ for
  $a_1, \cdots , a_k \in \calC^\infty_\text{\tiny\rm cpt} (\hgrp_1)$}.
\end{displaymath}

By Theorem I, the chain map
$\iota : C^\bullet (\calA \rtimes \grp ,\calA \rtimes \grp)
\rightarrow \calH_\grp^\bullet (X) $ is a quasi-isomorphism.
Hence by  Theorems \ref{Thm:WeEqMor} and \ref{Thm:ConvAlgMorMor}
from the Appendix the following result holds true.
\\[2mm]
{\bf Theorem III}.
Under the assumptions on $\grp$, $\hgrp$ and $\varphi$ from above,
the convolution algebras $\calA \rtimes \grp$ and
$\calA \rtimes \hgrp $ are Morita equivalent as bornological algebras.
Moreover, there is a commutative  diagram consisting of quasi-isomorphisms
\begin{displaymath}
 \xymatrix{
  C^\bullet (\calA \rtimes \grp ,\calA \rtimes \grp)
  \ar[r]^{\varphi^*} \ar[d]_{\iota} &
  C^\bullet (\calA \rtimes \hgrp ,\calA \rtimes \hgrp)
  \ar[d]^{\iota} \\
  \calH_\grp^\bullet (X)
  \ar[r]^{\varphi^*} & \calH_\hgrp^\bullet (X)
 }
\end{displaymath}
such that the upper horizontal chain map coincides with the
natural isomorphism between the Hochschild cohomologies $H^\bullet
(\calA \rtimes \grp ,\calA \rtimes \grp ) \rightarrow
 H^\bullet (\calA \rtimes \hgrp ,\calA \rtimes \hgrp )$ induced by the
Morita context between $\calA \rtimes \grp$ and $\calA \rtimes
\hgrp $. \vspace{2mm}

As an application of this result, we consider a point $x\in
\grp_0$ in the object set of a proper \'etale Lie groupoid $\grp$,
and denote by $\grp_x$ the isotropy group of $x$ that means the
group of all arrows starting and ending at $x$. Choose for each $g
\in \grp_x$ an open connected neighborhood $W_g \subset \grp_1$
(which can be chosen to be sufficiently small) such that both
$s_{|W_g}: W_g \rightarrow \grp_0$ and $t_{|W_g}: W_g \rightarrow
\grp_0$ are diffeomorphisms onto their images. Let $M_x$ be the
connected component of $x$ in $\bigcap_{g \in \grp_x} s (W_g)
\cap t(W_g) $ and put $M_g := W_g \cap s^{-1} (M_x)$ for all $g\in
\grp_x$. Define an action of $\grp_x$ on $M_x$ by
\begin{displaymath}
  \grp_x \times M_x \rightarrow M_x, \quad (g,y) \mapsto
  t \big( s_{|W_g}^{-1} (y)\big).
\end{displaymath}
It is now straightforward to check that the canonical embedding
\begin{displaymath}
  \grp_x \ltimes M_x \hookrightarrow \grp_{|\pi (M_x)} ,  \quad
  (y,g) \mapsto s_{|W_g}^{-1} (y),
\end{displaymath}
is open and a weak equivalence of Lie groupoids. In this article, we will
call a manifold $M_x$ together with a $\grp_x$-action on $M_x$ and a
$\grp_x$-equivariant embedding $\iota_x : M_x \hookrightarrow \grp_0$
a \emph{slice} around $x$, if the induced embedding
$\grp_x \ltimes M_x \hookrightarrow \grp_{|\pi \iota (M_x)}$ is open and
a weak equivalence of Lie groupoids.
The argument above shows that for every point $x\in \grp_0$ there exists a
slice. As a corollary to the above one obtains
\\[2mm]
{\bf Theorem IIIb}.
\emph{ Let $x\in \grp_0$ be a point and
$\iota_x : M_x \hookrightarrow \grp_0$  a slice around $x$.
Let $\varphi_x: \grp_x \ltimes M_x \hookrightarrow \grp_{|U_x}$ with
$U_x := \pi\iota_x (M_x)$ be the corresponding weak equivalence.
Then the convolution algebras $\calC^\infty (M_x) \rtimes \grp_x$ and
$\calA_{|U_x} \rtimes \grp_{|U_x}$ are Morita equivalent.
Moreover, the canonical chain map
\begin{displaymath}
  \varphi_x^* : \calH^\bullet (U_x) \rightarrow C^\bullet
  \big(
  \calC^\infty (M_x) \rtimes \grp_x , \calC^\infty (M_x) \rtimes \grp_x \big)
\end{displaymath}
is a quasi-isomorphism which implements the quasi-isomorphism
induced in Hoch\-schild cohomology by the Morita context
between $\calC^\infty (M_x) \rtimes \grp_x$ and
$\calA_{|U_x} \rtimes \grp_{|U_x}$.}
\vspace{0mm}

Theorems II and III enable us to localize the computation of the
Hochschild cohomology rings.
Locally, we have the following result, also shown in Section \ref{Sec:clhoch}.
\\[2mm]
{\bf Theorem IV}.
\emph{Let $M$ be a smooth manifold, and $\Gamma$ a finite group acting on $M$.
Then the Hochschild cohomology ring
$H^\bullet ( \calC^\infty_\text{\tiny\rm cpt}(M)\rtimes \Gamma,
 \calC^\infty_\text{\tiny\rm cpt} (M)\rtimes \Gamma)$ is given as follows.
As a vector space, one has
\begin{displaymath}
 H^\bullet( \calC^\infty_\text{\tiny\rm cpt}(M)\rtimes \Gamma,
 \calC^\infty_\text{\tiny\rm cpt} (M)\rtimes \Gamma)
 =\bigoplus_{\gamma \in \Gamma}
 \Gamma^\infty \Big(
 \Lambda^{\bullet-\ell (\gamma)}TM^\gamma \otimes \Lambda^{\ell (\gamma)}
 N^\gamma \Big)^\Gamma,
\end{displaymath}
where $\ell (\gamma)$ is the
codimension of $M^\gamma$ in $M$, and $N^\gamma$ is the normal bundle to
$M^\gamma$ in $M$. For elements
\begin{displaymath}
  \xi= \big( \xi_\alpha \big)_{\alpha \in \Gamma}\, ,
  \:
  \eta = \big( \eta_\beta \big)_{\beta \in \Gamma}
  \in
  \bigoplus_{\gamma \in \Gamma }
  \Gamma^\infty \Big(
  \Lambda^{\bullet-\ell (\gamma)}TM^\gamma \otimes \Lambda^{\ell (\gamma)}
  N^\gamma \Big)^\Gamma
\end{displaymath}
their cup product is given by
\begin{displaymath}
  ( \xi \cup \eta )_\gamma = \sum_{
   \alpha \cdot \beta =\gamma, \atop
   \ell(\alpha)+\ell(\beta)=\ell(\gamma)}
  \xi_\alpha \wedge \eta_\beta .
\end{displaymath}}
\mbox{ } \\
By globalization of Theorem IV, one obtains the following result.
\\[2mm]
 {\bf Theorem V}.
 \emph{Let $\grp$ be a proper \'etale groupoid. Denote by
 \[ \mathsf{B}_0:= \{ g \in \grp_1\mid s(g)=t(g)\}\] its space of loops,
 and by $N \rightarrow \mathsf{B}_0$ the normal bundle to $T\mathsf{B}_0$ in $T_{|\mathsf{B}_0}\grp_1$.
 Then the Hochschild cohomology ring of the convolution algebra is given by
 \begin{displaymath}
  H^\bullet( \calA \rtimes \grp , \calA \rtimes \grp )
  = \Gamma^\infty \Big( \Lambda^{\bullet-\ell} T\mathsf{B}_0 \otimes \Lambda^{\ell} N
  \big)^\grp,
 \end{displaymath}
 where $\ell$ is the locally constant function on $\mathsf{B}_0$ the value of which 
 at $g\in \mathsf{B}_0$ coincides with the codimension of the germ of $\mathsf{B}_0$ at $g$
 within $\grp_1$. The cup-product is given by the formula
 \[
  \xi\cup\eta=\int_m pr_1^*\xi\wedge pr_2^*\eta,
 \]
 for ``multivectorfields'' 
 $\xi,\eta\in \Gamma^\infty 
 \Big( \Lambda^{\bullet-\ell} T\mathsf{B}_0 \otimes \Lambda^{\ell} 
 N \big)^\grp$. 
 In the formula above, the maps 
 $m,pr_1,pr_2:\mathsf{S}\rightarrow\mathsf{B}_0$ 
 are the multiplication and the projection onto the first and second 
 component, where 
 \[
   \mathsf{S}:=\{(g_1,g_2)\in\mathsf{B}_0\times\mathsf{B}_0 \mid 
   s(g_1)= t(g_2)\}.
 \]
 Finally, the integral over $m$ simply means summation over the discrete 
 fiber of $m$.}

\begin{remark}
We remark that one can also compute the Gerstenhaber bracket on
$H^\bullet(\calA\rtimes \grp, \calA\rtimes \grp)$ by tracing down
the quasi-isomorphisms constructed in Theorem I-III. Though, there
is no natural Gerstenhaber bracket defined on the complex
$\calH^\bullet$, the bracket is well defined on the subcomplex of
local cochains which take values in compactly supported functions.
And this subcomplex is quasi-isomorphic to the whole complex by
Teleman's localization as is explained in Section 3. Therefore,
the Gerstenhaber bracket is well defined on the Hochschild
cohomology. Using the presheaf $\calH^\bullet$ and the local
computation in \cite{HalTan}, one can generalize the computation of 
the Gerstenhaber bracket from \cite{HalTan} to general orbifolds.
\end{remark}

Let us now consider the deformed case. The strategy in computing
the Hochschild cohomology of the deformed algebra
$\calA^\hbar\rtimes \grp$ is basically the same as in the undeformed algebra.
We define the deformed analogue of the complex of presheaves
$\calH_{\grp}^\bullet$ in the obvious way and denote it by
$\calH_{\grp,\hbar}^\bullet$. The associated \v{C}ech complex is denoted by
$\check{\calH}^{\bullet, \bullet}_{\calU, \hbar}$.
With this, the deformed versions of Theorems I--III are straightforward 
to prove:
the maps in these theorems generalize trivially to the sheaf $\calA^\hbar$.
Using the $\hbar$-adic filtrations on the complexes, Theorems I--III imply that
in the zero'th order approximation these maps are quasi-isomorphisms. By an 
easy spectral sequence argument, cf.~Section \ref{Sec:defhoch} one then 
shows this must be quasi-isomorphisms in general.

Again, this enables us to localize the computation of the
Hochschild cohomology rings. For a global quotient orbifold, we have
the following result, proved in Section \ref{loc-comp}.
\\[2mm]
{\bf Theorem VI}. \emph{Let $M$ be a smooth symplectic manifold,
and $\Gamma$ a finite group acting on $M$ preserving the
symplectic structure. Then the Hochschild cohomology ring
$H^\bullet \Big(\calA^{((\hbar))}_\text{\tiny\rm cpt}(M)\rtimes\Gamma,
 \calA^{((\hbar))}_\text{\tiny\rm cpt}\rtimes\Gamma \Big)$
is given as follows. As a vector space, one has
\begin{displaymath}
  H^\bullet \Big(\calA^{((\hbar))}_\text{\tiny \rm cpt}(M)\rtimes\Gamma,
  \calA^{((\hbar))}_\text{\tiny\rm cpt}\rtimes\Gamma \Big)\cong
  \bigoplus_{(\gamma)\subseteq\Gamma}
  H^{\bullet-\ell(\gamma)}_{Z(\gamma)}\big( M^\gamma,\C((\hbar)) \big),
\end{displaymath}
where $Z(\gamma)$ is the centralizer of $\gamma$ in $\Gamma$, and
$(\gamma)$ stands for the conjugacy class of $\gamma$ inside
$\Gamma$. For elements
\begin{displaymath}
  \alpha= \big( \xi_\gamma \big)_{\gamma \in \Gamma}\, ,
  \:
  \beta = \big( \beta_\gamma \big)_{\gamma \in \Gamma}
  \in
  \bigoplus_{\gamma \in \Gamma }
   \Big(H^{\bullet-l(\gamma)}\big(M^\gamma, \complex((\hbar))\big)\Big)^\Gamma
\end{displaymath}
their cup product is given by
\begin{displaymath}
\label{eq:cup-loc} 
  \alpha\cup\beta=\sum_{\gamma_1\gamma_2=\gamma,\atop
  \ell(\gamma_1)+\ell(\gamma_2)=\ell(\gamma_1\gamma_2)}
  \iota_{\gamma_1}^*\alpha_{\gamma_1}\wedge
  \iota_{\gamma_2}^*\alpha_{\gamma_2}.
\end{displaymath}}
\mbox{ } \\
Given this result, one might hope for a quasi-isomorphism 
$\calH_{\grp,\hbar}^\bullet\rightarrow\Omega^{\bullet-\ell}_{\tilde{X}}$ to exist,
which implements the isomorphism of the theorem above. The situation however 
is more complicated than that, and this is where the deformed case notably 
differs from the undeformed case.

First of all, it turns out one has to consider a sub-complex of presheaves
$\calH_{\grp,\text{\tiny \rm loc},\hbar}^\bullet\subset
\calH_{\grp,\hbar}^\bullet$, of cochains that are
local in a sense explained in the beginning of Section \ref{Sec:defhoch}.
Second, instead of one quasi-isomorphism, there is a chain
\[
  \calH_{\grp,\text{\tiny \rm loc},\hbar}^\bullet\hookrightarrow
  \calC^{\bullet,\bullet}_{\tilde{X}}\hookleftarrow 
  \Omega^{\bullet-\ell}_{\tilde{X}},
\]
where the intermediate double complex of sheaves 
$\calC^{\bullet,\bullet}_{\tilde{X}}$
is a twisted version of the Fedosov--Weinstein--Xu resolution used in
\cite{DolHCRCFT}. With this, we finally obtain the following result:
\\[2mm]
 {\bf Theorem VII}.
 \emph{Let $\mathsf{G}$ be a proper \'etale groupoid with an invariant 
 symplectic structure, modeling a symplectic orbifold $X$. For any invariant 
 deformation quantization $\calA^\hbar$ of $\mathsf{G}$, we have a natural 
 isomorphism 
 \[ 
   H^\bullet\left(\calA^{((\hbar))}\rtimes \mathsf{G}, 
   \calA^{((\hbar))}\rtimes \mathsf{G}\right)\cong 
   H^{\bullet-\ell}\big(\tilde{X},\C((\hbar))\big).
 \]
 With this isomorphism, the cup product is given by
 \[
  \alpha\cup\beta=\int_{m_\ell}pr_1^*\alpha\wedge pr_2^*\beta,
 \]
 for $\alpha,\beta\in H^{\bullet-\ell}\big(\tilde{X},\C((\hbar))\big)$ and 
 where $m_\ell$ is the restriction of $m$, cf.~Theorem V, to those connected 
 components of $\mathsf{S}$ that satisfy 
 \[
  \ell(g_1g_2)=\ell(g_1)+\ell(g_2), \quad (g_1,g_2) \in\mathsf{S}.
 \]
 Moreover, this cup product and symplectic volume form together define a 
 Frobenius algebra structure on $H^{\bullet-\ell}(\tilde{X}, \complex((\hbar)))$.
 }
\vspace{2mm}

On the other hand, on $H^{\bullet-\ell}(\tilde{X}, \complex((t)))$, there is the
famous Chen-Ruan orbifold product \cite{cr-orbifoldCoh}.
In Section \ref{Sec:equiderham}, we study the connection between
the cup product defined in Theorem VI and the Chen-Ruan orbifold product.
We introduce a de Rham model for some particular $S^1$-equivariant
Chen-Ruan orbifold cohomology and relate this de Rham model to the
above computation of Hochschild cohomology of $\calA^\hbar\rtimes \grp$.  \\

Given an arbitrary almost complex orbifold $X$, we introduce a
trivial $S^1$-action on $X$, but a nontrivial $S^1$-action on the
bundle $TX\to X$ by rotating each fiber. This $S^1$-action is
compatible with all the orbifold structures on $X$ and the inertia
orbifold $\tilde{X}$. Therefore, we have the $S^1$-equivariant
Chen-Ruan orbifold cohomology $(H^\bullet_{CR}(X)((t)), \star_t)$
as introduced in Section 5.1 with $\star_t$ the equivariant
Chen-Ruan orbifold product.

The de Rham model $(HT^\bullet(X)((t)), \wedge_t)$ for the above
$S^1$-equivariant Chen-Ruan orbifold cohomology is defined as a
vector space equal to $H^\bullet(\tilde{X})((t))[\ell]$ with the
product defined by putting
\[
(\xi\wedge_t \eta)_{\gamma}:=\sum_{\gamma=\gamma_1\gamma_2}
\iota^*_{\gamma}\big(\iota_{\gamma_1*}(\xi_{\gamma_1})\wedge
{\iota_{\gamma_2*}}(\eta_{\gamma_2})\big),\quad \xi,\eta\in
H^\bullet(\tilde{X})((t)),
\]
where $\iota_{\gamma_i}$ is the embedding of $X^{\gamma_i}$ into
$X$. The following theorem is proved in Section \ref{Sec:equiderham}.
\\[2mm]
{\bf Theorem VIII}. \emph{The two algebras $(H^\bullet_{CR}(X)((t)), \star_t)$ 
and $(HT^\bullet(X)((t)),\wedge_t)$ are isomorphic. }
\vspace{2mm}

To connect  $(HT^\bullet(X)((t)), \wedge_t)$ to the above
Hochschild cohomology ring, we define a decreasing
filtration $\calF^*$ on $HT^\bullet(X)((t))$ by
\[
  \calF^\ast=\{ \alpha\in H^\bullet(X^\gamma)((t)) \mid 
  \deg(\alpha)- \ell (\gamma)\geq \ast\}.
\]
We prove in Section \ref{Sec:equiderham} the following result and thus finish
our article.\\[2mm]
{\bf Theorem IX} \emph{The graded algebra
$gr(HT^\bullet(X)((t)))$ associated to $(HT^\bullet(X)((t)), \wedge_t)$ with
respect to the filtration $\calF^\ast$ is isomorphic to the
Hochschild cohomology algebra 
$(H^\bullet(\calA^\hbar\rtimes G , \calA^\hbar\rtimes G), \cup)$ 
by identifying $t$ with $\hbar$. }


%
%
\section{Cup product on the Hochschild cohomology of the convolution algebra}
\label{Sec:clhoch}
\subsection{Localization methods}
\label{Sec:locmeth}
We start with the proof of Theorem I by using a localization method going back
to Teleman \cite{TelMHH}. Recall that the orbifold $X = \grp_0 /\grp$
represented by a proper \'etale Lie groupoid $\grp$ carries in a natural way a
sheaf $\calC^\infty_X$ of smooth functions. More precisely, for every open
$U\subset X$ the algebra $\calC^\infty (U)$ coincides naturally with the
algebra $\big( \calC^\infty (\pi^{-1} (U))\big)^\grp$ of smooth functions on
$\grp_0$ invariant under the action of $\grp$. Clearly,
$C^k (\calA \rtimes \grp,\calA \rtimes \grp )$ is a module over
$\calC^\infty (X)$, and $\calH^k$ is a module presheaf over the
$\calC^\infty_X$ for every $k \in \N$. Since $\calC^\infty_X$ is a fine sheaf,
this implies in particular that $\calH^k$ has to be a fine presheaf.

Next recall from \cite[Sec.~3, Step I]{NeuPflPosTanHFDPELG} that there is
a canonical isomorphism
\begin{equation}
\label{Eq:HochCohId}
  \hat{\hspace{2mm}} : C^k(\calA\rtimes \grp,\calA\rtimes \grp ) \rightarrow
  \Hom ( \calA\rtimes \grp^{\hatotimes k} , \calC^\infty (\grp_1) )
  =  C^k_\text{\tiny \rm red}
  (\calA\rtimes \grp, \calC^\infty (\grp_1) ), \:\: F \mapsto \hat{F} .
\end{equation}
Hereby, the map
$\hat{F} :\calA\rtimes \grp^{\hatotimes k}\rightarrow \calC^\infty (\grp_1)$
is uniquely determined by the requirement that for every compact
$K\subset \grp_1$ and all $a_1, \ldots ,a_k \in \calA\rtimes \grp$
the relation
\[
  \hat{F} (a_1 \otimes \ldots \otimes a_k)_{|K} =
  F (\varphi_K \delta_u \otimes a_1 \otimes \ldots \otimes
  a_k \otimes\varphi_K \delta_u )
\]
holds true, where $\varphi_K: \grp_0\rightarrow [0,1]$ is a smooth function
with compact support such that $\varphi_K (x)= 1$ for all $x$ in a
neighborhood of $s(K)\cup t(K)$, and where $\delta_u: \grp_1 \rightarrow \R$
is the locally constant function which coincides with $1$ on $\grp_0$ and
which vanishes elsewhere.

Now let us fix a smooth function $\varrho : \R \rightarrow [0,1]$ which has
support in $(-\infty , \frac 34]$ and which satisfies $\varrho (r) =1$ for
$r\leq \frac 12$. For $\varepsilon >0$ we denote by $\varrho_\varepsilon$
the rescaled function $r \mapsto \varrho (\frac s \varepsilon )$.
Next choose a $\grp$-invariant complete riemannian metric on $\grp_0$, and
denote by $d$ the corresponding geodesic distance on $\grp_0$ (where we put
$d(x,y) =\infty$, if $x$ and $y$ are not in the same connected component of
$\grp_0$). Then $d^2$ is smooth on the set of pairs of points of $\grp_0$
having finite distance. Put for every
$k\in \N \cup \{ -1 \}$, $i=1,\cdots, k +1 $ and $\varepsilon >0$:
\begin{displaymath}
  \Psi_{k,i,\varepsilon} (g_0,g_1,\cdots, g_k) =  \prod_{j=0}^{i-1}
  \varrho_\varepsilon \big( d^2( s( g_j ) , t (g_{j+1})) \big),
  \quad \text{where $g_j \in \grp_1$ and $g_{k+1}:= g_0$}.
\end{displaymath}
Moreover, put $\Psi_{k,\varepsilon} := \Psi_{k,k+1,\varepsilon}$.
Using the above identification (\ref{Eq:HochCohId}) we then define
for $F \in C^k := C^k (\calA \rtimes G , \calA \rtimes G)$ a
Hochschild cochain $\Psi^{k,\varepsilon} F$ as follows:
\begin{displaymath}
\begin{split}
 & \Psi^{k,\varepsilon} F (a_1 \otimes \cdots \otimes a_k ) \, (g_0) :=
  F \big( \Psi_{k,\varepsilon} (g_0^{-1}, -, \cdots,- ) \cdot
 (a_1 \otimes \cdots \otimes a_k) \big) \, (g_0),\\
 & \hspace{15em} \text{ for $g_0 \in \grp_1$ and
 $a_1, \cdots, a_k \in \mathcal C^\infty_{\text{\tiny \rm cpt}} ( \grp_1)$}.
\end{split}
\end{displaymath}
One immediately checks that $\Psi^{\bullet,\varepsilon}$
forms a chain map on the Hochschild cochain complex.
Likewise, one defines a chain map  $\Psi^{\bullet,\varepsilon}$ acting on the
sheaf of cochain complexes $\calH^\bullet$.
In \cite[Sec.~3, Step 2]{NeuPflPosTanHFDPELG} it has been shown that
there exist  homotopy operators $H^{k,\varepsilon} : C^k \rightarrow C^{k-1}$
such that
\begin{equation}
\label{Eq:HomLocDiag}
    (\beta H^{k,\varepsilon} + H^{k+1,\varepsilon} \beta ) F  =
    F - \Psi^{k ,\varepsilon} F
\end{equation}
for all $F\in C^{k}$. By a similar argument like in \cite{NeuPflPosTanHFDPELG}
one shows that this algebraic homotopy holds also for $F\in \calH^k (X)$.
By completeness of the metric $d$, the cochain
$\Psi^{\bullet ,\varepsilon} F $ is an element of
$C^k_\text{\tiny\rm red}(\calA \rtimes G , \calA \rtimes G)$ for
$F\in C^k$ or $F \in \calH^k (X)$. Hence $\Psi^{\bullet ,\varepsilon}$
is a quasi-inverse to the canonical embedding
$C^\bullet_\text{\tiny \rm red} (\calA \rtimes \grp , \calA \rtimes \grp )
 \hookrightarrow C^\bullet (\calA \rtimes \grp , \calA \rtimes \grp )$
resp.~to $\iota: \:  C^\bullet_\text{\tiny \rm red}
(\calA \rtimes \grp , \calA \rtimes \grp ) \rightarrow \calH^\bullet (X)$.
This proves Theorem I.

Next, we study the properties of the \v{C}ech double complex
$\check{\calH}^{\bullet,\bullet}_\calU$ associated to an open covering $\calU$
of $X$ and prove Theorem II. We already have shown above that each presheaf
$\calH^k$ is fine. Denote by $\hat{\calH}^k$ the sheaf associated to the
presheaf $\calH^k$. Then the \v{C}ech cohomology of $\calH^\bullet$
coincides with the  \v{C}ech cohomology of $\hat{\calH}^\bullet$, and the
latter is given by the global sections of the cohomology sheaf of
 $\hat{\calH}^\bullet$ (see for example \cite[Sec.~6.8]{SpaAT}).
To prove the last part of Theorem II choose a locally finite open covering
$\calU$ of $X$ such that each element $U\in \calU$ is relatively compact and
let $(\varphi_U )_{U\in \calU}$ be a subordinate partition of unity by
smooth functions on $X$. Then the \v{C}ech double complex
$\check{C}^\bullet_\calU(\calH^\bullet)$ collapses at the $E_1$ term, hence
its cohomology can be computed by the cohomology of
$\check{Z}^{\bullet,0}_\calU (\calH^\bullet)$.
By assumption on $\calU$ there exists for every $U\in \calU$ a
$\varepsilon_U >0$ such that for every $H_U\in \calH^p (U)$ the cochain
$\varphi_U \Psi^{p,\varepsilon_U} H_U\in \calH^p (U)$ can be extended by zero
to an element of $\calH^p (X)$ which we also denote by
$\varphi_U \Psi^{p,\varepsilon_U} H_U$.
Then it is easily checked that the restriction map
\[
    \calH^p (X)\rightarrow \check{Z}^{p,0}_\calU (\calH^\bullet )
    \subset \prod_{U\in \calU} \calH^p (U), \quad
    H\mapsto (H_{|_U})_{U\in \calU}
\]
is a quasi-isomorphism with quasi-inverse given by
\begin{equation}
  \check{Z}^{p,0}_\calU (\calH^\bullet ) \ni (H_U)_{U\in\calU}
  \mapsto
  \sum_{U\in \calU} \varphi_U \Psi^{p,\varepsilon_U} H_U \in \calH^p (X) .
\end{equation}
This finishes the proof of Theorem II.
\subsection{The global quotient case}
\label{Sec:clglobquot}
In this part, we provide a complete proof of Theorem IV.
Let $\Gamma$ be a finite group acting on a smooth orientable manifold $M$.
This defines a transformation groupoid
$\grp := (\Gamma \ltimes M \rightrightarrows M)$.
In this case, the groupoid algebra of $\grp$ is equal to the crossed
product algebra $\calC_{\text{\tiny \rm cpt}}^\infty(M)\rtimes \Gamma$.

In \cite[Thm.~3]{NeuPflPosTanHFDPELG}, we proved that as a vector
space the Hochschild cohomology of the algebra
$\calC^\infty_{\text{\tiny \rm cpt}}(M)\rtimes \Gamma$ is equal to
\[
 H^\bullet( \calC^\infty_\text{\rm\tiny cpt}(M)\rtimes \Gamma,
 \calC^\infty_\text{\rm\tiny cpt} (M)\rtimes \Gamma)
 =\Big(\bigoplus_{\gamma \in \Gamma}
 \Gamma^\infty (
 \Lambda^{\bullet-\ell (\gamma)}TM^\gamma \otimes \Lambda^{\ell (\gamma)}
 N^\gamma) \Big)^\Gamma,
\]
where $\ell (\gamma)$ is the codimension of $M^\gamma$ in $M$, and
$N^\gamma$ is the normal bundle to $M^\gamma$ in $M$. The main
goal of this section is to compute the cup product between
multivector fields on the inertia orbifold, that means between
elements of the right hand side of the  preceding equation, from
the cup product on the Hochschild cohomology of the left hand side
of that equation. We hereby restrict our considerations to the
particular case, where $M$ carries a $\Gamma$-invariant riemannian
metric such that $M$ is geodesically convex. This condition is in
particular satisfied for a linear $\Gamma$-representation space
carrying a $\Gamma$-invariant scaler product. Theorem IV can
therefore be immediately reduced to the case considered in the
following by Theorems II and III and the slice theorem.

In the first part of our construction, we outline how to determine the
Hochschild  cohomology as a vector space. To this end we construct
two  cochain maps $L$ and $T$ between the Hochschild cochain complex and the
space of sections of multivector fields on the inertia orbifold.
These two cochain maps are actually quasi-isomorphisms. The map
$L$ has already been constructed in \cite{NeuPflPosTanHFDPELG}, the map
$T$ in \cite{HalTan}.  In the second part of our construction,
we will use the cochain maps $T$ and $L$ to compute the cup product.

\subsubsection{The cochain map $L$}
Following \cite[Theorem 3.1]{NeuPflPosTanHFDPELG} we construct
\[
  L: C^\bullet (\calC_{\text{\rm \tiny cpt}}^\infty(M)\rtimes \Gamma,
  \calC^\infty_{\text{\rm \tiny cpt}}(M)\rtimes
  \Gamma)\longrightarrow \Big(\bigoplus_{\gamma \in \Gamma}
  \Gamma^\infty (
  \Lambda^{\bullet-\ell (\gamma)}TM^\gamma \otimes \Lambda^{\ell (\gamma)}
  N^\gamma) \Big)^\Gamma.
\]
The map $L$ is the composition of three cochain maps $L_1$, $L_2$ and $L_3$
defined in the following.

To define the first map $L_1$ recall that $\Gamma$ acts on
$F \in C^k (\calC^\infty_\text{\tiny \rm cpt}(M),
 \calC^\infty_\text{\tiny \rm cpt}(M)\rtimes \Gamma)$ by
\[
\gamma F : = \Big( {\calC^\infty_\text{\tiny \rm cpt}(M)}^{\hatotimes k}
\ni f_1 \otimes \cdots \otimes f_k  \mapsto \delta_{\gamma}\cdot
F \big(\gamma^{-1}(f_1) \otimes  \cdots \otimes \gamma^{-1}(f_k) \big)\cdot
\delta_{\gamma^{-1}} \Big).
\]
Given $f\in \calC^\infty_\text{\tiny \rm cpt}(M) $ and $\gamma \in \Gamma$
we hereby (and in the following) use the notation $f \delta_\gamma$
for the function in $\calC^\infty_\text{\tiny \rm cpt} (M) \rtimes \Gamma$
which maps $(\sigma,p) \in \Gamma \times M$ to  $f (\gamma p)$, if
$\sigma = \gamma$, and to $0$ else. Now we put
\[
\begin{split}
   L_1: \: & C^k(\calC_{\text{\rm \tiny cpt}}^\infty(M)\rtimes \Gamma,
   \calC_{\text{\rm \tiny cpt}}^\infty(M)\rtimes \Gamma) \longrightarrow
   C^k (\calC_{\text{\rm \tiny cpt}}^\infty(M),
   \calC_{\text{\rm \tiny cpt}}^\infty(M)\rtimes \Gamma)^\Gamma, \\
   & F \mapsto L_1 F :=
   \Big( {\calC^\infty_\text{\tiny \rm cpt}(M)}^{\hatotimes k}
   \ni f_1 \otimes \cdots \otimes f_k  \mapsto
   F (f_1 \delta_e \otimes \cdots \otimes f_k \delta_e) \Big).
\end{split}
\]
Next we explain how the map $L_2$ is constructed.  It has the following form:
\[
L_2: C^\bullet(\calC_{\text{\rm \tiny cpt}}^\infty(M),
     \calC_{\text{\rm \tiny cpt}}^\infty(M)\rtimes\Gamma)^\Gamma
     \longrightarrow \Big(\bigoplus_{\gamma\in \Gamma}
     \Gamma^\infty(\Lambda^\bullet T_{|M^\gamma}M)\Big)^\Gamma ,
\]
where $T_{|M^\gamma}M$ is the restriction of the vector bundle $TM$
to $M^\gamma$, and where the differential on the complex
$\bigoplus_{\gamma\in \Gamma}\Gamma^\infty(\Lambda^\bullet T_{|M^\gamma}M)$
is given by the $\wedge$-product with a nowhere vanishing vector field
$\kappa$ on $M$ which we define later.
Actually, we will define a $\Gamma$-equivariant chain map $L_2$ slight more
general than what is stated above, namely a map
\[
   L_2: C^\bullet(\calC_{\text{\rm \tiny cpt}}^\infty(M),
   \calC_{\text{\rm \tiny cpt}}^\infty(M)\rtimes\Gamma)
   \longrightarrow \bigoplus_{\gamma\in \Gamma} \Gamma^\infty
   (\Lambda^\bullet T_{|M^\gamma}M) .
\]
As a $\calC_{\text{\rm \tiny cpt}}^\infty(M)$-$\calC_{\text{\rm \tiny cpt}}^\infty(M)$ bimodule, $\calC_{\text{\rm \tiny cpt}}^\infty(M)\rtimes \Gamma$
has a natural splitting into a direct sum of submodules
$\bigoplus_{\gamma\in \Gamma}\calC_{\text{\rm \tiny cpt}}^\infty(M)_\gamma$.
Accordingly, the Hochschild cochain complex
$C^\bullet \big(\calC_{\text{\rm \tiny cpt}}^\infty(M),
\calC_{\text{\rm \tiny cpt}}^\infty(M)\rtimes \Gamma \big)$ naturally splits
as a direct sum
\[
  \bigoplus_{\gamma\in\Gamma}
  C^\bullet \big( \calC_{\text{\rm \tiny cpt}}^\infty(M),
  \calC_{\text{\rm \tiny cpt}}^\infty(M_\gamma) \big).
\]
Therefore, to define the map $L_2$, it is enough to consider each single map
\[
  L_2^\gamma: C^\bullet \big(\calC_{\text{\rm \tiny cpt}}^\infty(M),
  \calC_{\text{\rm \tiny cpt}}^\infty(M_\gamma ) \big)\longrightarrow
  \Gamma^\infty(\Lambda^\bullet T_{|M^\gamma}M ).
\]
In the following we use ideas from the paper \cite{ConNDG} to construct
$L_2^\gamma$. To this end let $\pr_2: M\times M\to M$ be the projection
onto the second factor of $M\times M$, and $\xi$ the vector field on
$M\times M$ which maps $(x_1, x_2)$ to $\exp_{x_2}^{-1}(x_1)$.
By our assumptions on the riemannian metric on $M$ the vector field $\xi$ is
well-defined and $\Gamma$-invariant.
According to \cite[Lemma 44]{ConNDG}, the complex
$K^\bullet=\big( \Gamma_\text{\tiny \rm cpt}^\infty (
\operatorname{pr}_2^*(\Lambda^\bullet T^*M)),\xi\llcorner \big)$
defines a projective resolution of
$\calC^\infty_\text{\tiny \rm cpt} (M)$. Essentially, it is a Koszul
resolution for $\calC^\infty_\text{\tiny \rm cpt} (M)$.
Following Appendix \ref{Sec:HHBR}, we use the resolution $K^\bullet$ to
determine the Hochschild cohomology
$H^\bullet \big( \calC_{\text{\rm \tiny cpt}}^\infty(M),
  \calC_{\text{\rm \tiny cpt}}^\infty( M_\gamma )\big) $ as the cohomology of
the cochain complex $\Hom_{\calC_\text{\rm \tiny cpt}^\infty (M^2)}
 \big( K^\bullet , \calC_{\text{\rm \tiny cpt}}^\infty(M_\gamma ) \big)$.
By \cite{ConNDG} the following chain map is a quasi-isomorphism between
the resolution $K^\bullet$ and the Bar resolution
$\Barcpl_\bullet (\calC^\infty_{\text{\rm \tiny cpt}} (M) \big)$:
\[
\begin{split}
  \Phi :\: & K^k \rightarrow
  \Barcpl_k ( \calC^\infty_{\text{\rm \tiny cpt}} (M) )
  = \calC^\infty_\text{\tiny \rm cpt} (M^{k+2} ),\\
  &\omega \mapsto \Big( M^{k+2} \ni ( a,b, x_1, \cdots, x_k) \mapsto
  \langle \xi (x_1, b)\wedge\cdots \wedge \xi (x_k, b), \omega(a,b) \rangle
  \Big).
\end{split}
\]
Hence the dual of the chain map $\Psi$ defines a quasi-isomorphism
\[
 \Phi^*: \: \Hom_{\calC^\infty_{\text{\rm \tiny cpt}} (M^2)}
 (\calC^\infty_{\text{\rm \tiny cpt}} (M^{k+2}),
 \calC^\infty_{\text{\rm \tiny cpt}}) \rightarrow
 \Hom_{\calC^\infty_{\text{\rm \tiny cpt}} (M^2)}
 (K^k, \calC^\infty_{\text{\rm \tiny cpt}} (M)).
\]
Now consider the embedding $\Delta_\gamma: M\to M\times M$ given by
$\Delta_\gamma(x)=(\gamma(x), x)$.  According to
\cite[Sec.~3, Step 4]{NeuPflPosTanHFDPELG}, the map
\[
  \eta : \: \Gamma^\infty (\Lambda^k TM) \rightarrow
  \Hom_{\calC^\infty_{\text{\rm \tiny cpt}} (M^2)}
  (K^k, \calC^\infty_{\text{\rm \tiny cpt}} (M)), \quad
  \tau \mapsto \eta (\tau),
\]
defined by $\eta (\tau) (\omega) = \langle \Delta^*_\gamma \omega , \tau \rangle$
for $\omega \in \Gamma^\infty_\text{\tiny\rm cpt}\pr_2^*(\Lambda^\bullet T^*M))$
is an isomorphism.
So finally we can define for $F \in C^k(\calC_{\text{\rm \tiny cpt}}^\infty(M),
\calC_{\text{\rm \tiny cpt}}^\infty(M)_\gamma)$ an element
$L^\gamma_2(F)\in \Gamma^\infty(\Lambda^k TM)$ by
\[
  L^\gamma_2(F )=\eta^{-1} \Phi^* \Psi^{k,\varepsilon} (F),
\]
where we have used the cut-off cochain map $\Psi^{k,\varepsilon}$ defined above.
Thus we obtain an isomorphism of complexes
\[
  L^\gamma_2 : \: \big( C^k (\calC^\infty_\text{\tiny \rm cpt} (M),
  \calC^\infty_\text{\tiny \rm cpt} (M) ), b\big) \rightarrow
  (\Lambda^\bullet TM, \kappa_\gamma\wedge ),
\]
where $\kappa_\gamma$ is the restriction of the vector field $\xi$ on
$M\times M$ to the $\gamma$-diagonal $\Delta_\gamma$.

In the case, where $M$ is a (finite dimensional) vector space $V$
with a linear $\Gamma$-action, we can write down $L_2^\gamma$
explicitly. Choose coordinates $x^i$, $i=1,\ldots , \dim V$ on
$V$. Then the vector field $\xi$ on $V\times V$ can be written as
\begin{equation}
\label{eulervf}
 \xi (x_1, x_2)=\sum_{i}(x_1-x_2)^{i}\frac{\partial}{\partial x_2^i}.
\end{equation}
Moreover, $L^\gamma_2(F)$ is given as follows:
\begin{equation}
\label{eq:L-2}
\begin{split}
 L^\gamma_2 (F)&\,(x) = \sum\limits_{i_1, \cdots, i_k}\!\!
 F\big( \Psi_{k,\varepsilon} ( -, x_1, \cdots ,x_k) \cdot \\
 & \cdot \langle \xi (x_1,x)\wedge\cdots \wedge \xi (x_k, x), \text{pr}_2^*(dx^{i_1}
 \wedge\cdots \wedge dx^{i_k})\rangle \big)(x)
 \frac{\partial}{\partial x^{i_1}}\wedge \cdots \wedge
 \frac{\partial}{\partial x^{i_k}}\\
 = \: & \sum\limits_{i_1, \cdots i_k}\Big(\sum_{\sigma\in S_k}
 (-1)^{\sigma} F \big((x_{\sigma(1)}-x)^{i_1}\cdots
 (x_{\sigma(k)}-x)^{i_k}\big)\Big) (x) \frac{\partial}{\partial
 x^{i_1}}\wedge \cdots \wedge \frac{\partial}{\partial x^{i_k}},
\end{split}
\end{equation}
where $x\in V$, and where we have identified $F$ with a bounded linear map
from $\calC^\infty_\text{\tiny \rm cpt} (V^k)$ to
$\calC^\infty (V)$.

To define $L_3$ we construct for each $\gamma \in \Gamma$ a localization map
to the fixed point sub manifold $M^\gamma$.
Recall that we have chosen a $\Gamma$-invariant complete riemannian metric on $M$,
and consider the normal bundle $N^\gamma$ to the embedding
$\iota_\gamma:M^\gamma\hookrightarrow M$. The riemannian metric allows us to
regard $N^\gamma$ as a subbundle of the restricted tangent bundle
$T_{|M^\gamma}M$.  Now we denote by $\pr^\gamma$ the orthogonal projection from
$\Lambda^\bullet T_{|M^\gamma} M$ to
$\Lambda^{\bullet-\ell(\gamma)}TM^\gamma\otimes \Lambda^{\ell(\gamma)}N^\gamma$.
The chain map
\[
 L_3: \Big(\bigoplus_{\gamma\in \Gamma} \Gamma^\infty(\Lambda^\bullet TM), 
 \kappa\wedge - \Big)\longrightarrow \bigoplus_{\gamma\in \Gamma}
 \big(\Gamma^\infty(\Lambda^{\bullet-\ell (\gamma)} 
 TM^\gamma \otimes \Lambda^{\ell (\gamma)}N^\gamma), 0\ \big).
\]
is then constructed as the sum of the maps $L_3^\gamma$ defined by
\[
  L_3^\gamma (X) =\text{pr}^\gamma(X|_{M^\gamma}) \quad
  \text{for $X\in \Gamma^\infty (\Lambda^\bullet TM)$}.
\]

In \cite[Sec.~3]{NeuPflPosTanHFDPELG} we proved that
$L=L_3\circ L_2\circ L_1$ is a quasi-isomorphism of cochain complexes.
\subsubsection{The chain map $T$}
Under the assumption that $M$ is a linear $\Gamma$-representation space $V$
we construct in this section a quasi-inverse
\[
  T:  \Big(\bigoplus_{\gamma\in \Gamma}
  \Gamma^\infty(\Lambda^{\bullet-\ell(\gamma)}
  TM^\gamma\otimes \Lambda^{\ell(\gamma)}N^\gamma)\Big)^{\Gamma}\longrightarrow
  C^\bullet(\calC_{\text{\rm \tiny cpt}}^\infty(M)\rtimes \Gamma,
  \calC_{\text{\rm \tiny cpt}}^\infty(M)\rtimes \Gamma).
\]
to the above cochain map $L$.
To this end we first recall the construction of the
\textit{normal twisted cocycle}
\[
  \Omega_\gamma \in C^{l(\gamma)}(\calC_{\text{\rm \tiny cpt}}^\infty(V),
  \calC_{\text{\rm \tiny cpt}}^\infty(V)_\gamma)
\]
from \cite{HalTan}. Since $\Gamma$ acts linearly on $V$, $V^\gamma$ is a
linear subspace of $U$ and has a normal space $V^\perp$.
Let $x_i$, $i=1, \cdots, n-\ell(\gamma)$, be coordinates on
$V^\gamma$, and $y_j$, $j=1, \cdots, \ell(\gamma)$, coordinates on $V^\perp$.
We write $\tilde{y}=\gamma y$, and for every $\sigma \in S_{\ell(\gamma)}$ we
introduce the following vectors in $V^\perp$:
\[
\begin{array}{rclrcl}
 z^0 \!\!&=&\!\!(y_1, \cdots, y_{\ell(\gamma)}) , &
 z^1 \!\!&=&\!\!(y_1, \cdots, \tilde{y}_{\sigma(1)}, \cdots, y_{\ell(\gamma)}) , \\
 z^2 \!\!&=&\!\!(y_1, \cdots, \tilde{y}_{\sigma(1)}, \cdots, \tilde{y}_{\sigma(2)},
 \cdots, y_{\ell(\gamma)}), & \cdots \\
 z^{\ell(\gamma)-1}\!\!&=&\!\!(\tilde{y}_1, \cdots, y_{\sigma(\ell(\gamma))}, \cdots ,
 \tilde{y}_{\ell(\gamma)}), &
 z^{\ell (\gamma)}\!\!&=&\!\!(\tilde{y}_1, \cdots, \tilde{y}_{\ell(\gamma)}).
\end{array}
\]
Then we define a cochain $\Omega_{\gamma} \in C^{\ell
(\gamma)}(\calC^\infty_\text{\tiny\rm cpt}(V),
\calC^\infty_\text{\tiny\rm cpt}(V)_\gamma)$ as follows:
\[
\begin{split}
  \Omega_{\gamma} & \, (f_1, \cdots, f_{\ell(\gamma)})(x,y):=
  \frac{1}{\ell(\gamma)!} \cdot \\
  &\!\!\!\cdot \! \!\sum\limits_{\sigma\in S_{\ell(\gamma)}}\!\!\!
  \frac{(f_1(x,z^0)-f_1(x, z^1))(f_2(x, z^1)-f_2(x,
  z^2))\cdots (f_{\ell(\gamma)}(x,z^{n-1})-f_{\ell(\gamma)}(x,
  z^n))}{(y_1-\tilde{y}_1)\cdot\ldots\cdot
  (y_{\ell(\gamma)}-\tilde{y}_{\ell(\gamma)})},
\end{split}
\]
where $f_1, \cdots, f_{\ell(\gamma)} \in \calC^\infty_\text{\tiny\rm cpt} (V)$,
$x\in V^\gamma$ and $y\in V^\perp$.
It is straightforward to check that $\Omega_\gamma$ is a cocycle in
$C^{\ell(\gamma)}(\calC_{\text{\rm \tiny cpt}}^\infty(V),
\calC_{\text{\rm \tiny cpt}}^\infty(V)_\gamma)$ indeed.
Now define the cochain map
\[
  T_1: \Big(\bigoplus_{\gamma \in \Gamma}
  \Gamma^\infty(\Lambda^{\bullet-\ell(\gamma)}TV^\gamma\otimes
  \Lambda^{\ell(\gamma)}N^\gamma)\Big)^{\Gamma}\longrightarrow
  C^\bullet(\calC_{\text{\rm \tiny cpt}}^\infty(V),
  \calC_{\text{\rm \tiny cpt}}^\infty(V)\rtimes \Gamma)^\Gamma
\]
as the sum of maps
\[
  T_1^\gamma:
  \Gamma^\infty(\Lambda^{k-\ell(\gamma)} TV^\gamma\otimes
  \Lambda^{\ell(\gamma)}N^\gamma)\longrightarrow
  C^k (\calC_{\text{\rm \tiny cpt}}^\infty(V),
  \calC_{\text{\rm \tiny cpt}}^\infty(V)_\gamma),
\]
defined by
\[
  T_1^\gamma(X\otimes Y_\gamma)=
  Y_\gamma(y_1, \cdots, y_{\ell(\gamma)}) \,
  X\sharp \Omega_\gamma,
\]
where $ X\in \Gamma^\infty(\Lambda^{k-\ell(\gamma)}TV^\gamma)$,
$Y_\gamma \in  \Gamma^\infty(\Lambda^{\ell(\gamma)}N^\gamma)$,
and where $X\sharp \Omega_\gamma(f_1, \cdots, f_k)$ is equal to
\[
  X (f_1, \cdots, f_{k-\ell(\gamma)}) \,
  \Omega_\gamma(f_{k-\ell(\gamma )+1}, \cdots, f_{\ell(\gamma)}).
\]
Observe hereby that $X$ to act on $f_1, \cdots,
f_{k-\ell(\gamma)}$, we need to use a $\Gamma$-invariant
connection, i.e.~the Levi-Civita connection of the invariant
metric, on the normal bundle of $V^\gamma$ in $V$ to lift $X$ to a
vector field on $V$.

The map
\[
  T:\Big(\bigoplus_{\gamma\in \Gamma}\Gamma^\infty
  (\Lambda^{\bullet-\ell(\gamma)}TV^\gamma\otimes
  \Lambda^{\ell(\gamma)}N^\gamma)\Big)^{\Gamma}\longrightarrow
  C^\bullet(\calC_{\text{\rm \tiny cpt}}^\infty(V)\rtimes \Gamma,
  \calC_{\text{\rm \tiny cpt}}^\infty(V)\rtimes \Gamma)
\]
is now written as the composition of $T_1$ and $T_2$, where $T_2$ is the
standard cochain map from the Eilenberg-Zilber theorem:
\[
  T_2:C^\bullet (\calC_{\text{\rm \tiny cpt}}^\infty(V),
  \calC_{\text{\rm \tiny cpt}}^\infty(V)\rtimes \Gamma)^\Gamma
  \longrightarrow C^\bullet (\calC_{\text{\rm \tiny cpt}}^\infty(V)\rtimes
  \Gamma, \calC_{\text{\rm \tiny cpt}}^\infty(V)\rtimes \Gamma).
\]
More precisely, for $F\in C^k(\calC_{\text{\rm \tiny cpt}}^\infty(V),
\calC_{\text{\rm \tiny cpt}}^\infty(V)\rtimes \Gamma)$ one has
\[
  T_2(F)( f_1 \delta_{\gamma_1}, \cdots, f_k \delta_{\gamma_k})=
  F (f_1 , \gamma_1(f_2), \cdots, \gamma_1\cdots\gamma_{k-1}(f_k))
  \delta_{\gamma_1\cdots \gamma_k}.
\]
The following result then holds for the composition $T=T_2\circ T_1$.
Its proof is performed by a straightforward check
(cf.~\cite[Sec.~2]{HalTan} for some more details).
\begin{theorem}
\label{thm:quasi-inverse}
  Let $V$ be a finite dimensional real linear $\Gamma$-representation
  space. Then the cochain maps $L$ and $T$ defined above satisfy
  $L\circ T=\id$. In particular, $T$ is a quasi-inverse to $L$.
\end{theorem}

By the above considerations one concludes that there is an isomorphism of
vector spaces between the Hochschild cohomology of
$\calC_{\text{\rm \tiny cpt}}^\infty(M)\rtimes \Gamma$ and the space of
smooth sections of alternating multi-vector fields on the corresponding
inertia orbifold.
\subsubsection{The cup product}
\label{Sec:LocClCupProd}
In this part, we use the above constructed maps to compute the cup
product on the Hochschild cohomology of the algebra
$\calC_{\text{\rm \tiny cpt}}^\infty(M)\rtimes \Gamma$. By proving
the following proposition, we will complete proof of Theorem IV.

\begin{proposition}
\label{prop:cup} For every smooth $\Gamma$-manifold $M$ the cup product on
the Hochschild cohomology
$H^\bullet (\calC_{\text{\rm \tiny cpt}}^\infty(M)\rtimes \Gamma ,
\calC_{\text{\rm \tiny cpt}}^\infty(M)\rtimes \Gamma ) \cong
\Big(\bigoplus_{\gamma\in\Gamma}\Gamma^\infty(\Lambda^{\bullet-\ell(\gamma)}
TM^\gamma\otimes \Lambda^{\ell(\gamma)}N^\gamma)\Big)^{\Gamma}$
is given for two cochains
\[
 \xi=(\xi_\alpha)_{\alpha\in \Gamma}, \:
 \eta=(\eta_\beta)_{\beta\in \Gamma}\in
 \Big(\bigoplus_{\gamma\in\Gamma}\Gamma^\infty(\Lambda^{\bullet-\ell(\gamma)}
 TM^\gamma\otimes \Lambda^{\ell(\gamma)}N^\gamma)\Big)^{\Gamma}
\]
as the cochain $\xi \cup \eta$ with components
\[
 (\xi \cup \eta)_\gamma=\sum_{\alpha\beta=\gamma, \ell(\alpha)+\ell(\beta)=\ell(\gamma)}\xi_\alpha\wedge \eta_\gamma.
\]
\end{proposition}

\begin{proof}
It suffices to prove the claim under the assumption that $M$ is a
linear $\Gamma$-representation space $V$. Then we have the above
defined quasi-inverse $T$ to the cochain map $L$ at our disposal.
To compute $\xi\cup \eta$ we thus have to determine the
multivector field $L(T(\xi)\cup T(\eta))$. Since $L$ is the
composition of $L_1$, $L_2$, and $L_3$, we compute $L_1(T(\xi)\cup
T(\eta))$ first. Recall that the cochain $L_1(T(\xi)\cup
T(\eta))\in C^{p+q}(\calC_{\text{\rm \tiny cpt}}^\infty(V),
\calC_{\text{\rm \tiny cpt}}^\infty(V)_{\gamma})$ is defined by
\begin{equation}
\label{eq:cup}
  L_1(T(\xi)\cup T(\eta))(f_1, \cdots, f_{p+q})=
  \sum_{\alpha\beta=\gamma}T_1^\alpha(\xi_\alpha)(f_1, \cdots,f_p)
  \, \alpha(T_1^\beta(\eta_\beta)(f_{p+1},\cdots, f_{p+q})),
\end{equation}
where $f_1, \cdots , f_{p+q}\in\calC_{\text{\rm \tiny cpt}}^\infty(V)$,
Recall also that the cochain map
\[
  L_2: C^k (\calC_{\text{\rm \tiny cpt}}^\infty(V),
  \calC_{\text{\rm \tiny cpt}}^\infty(V)\rtimes \Gamma)\rightarrow
  \bigoplus_{\gamma\in \Gamma} \Gamma^\infty \big(\Lambda^k T_{V^\gamma}V \big)
\]
essentially is the anti-symmetrization of the linear terms of a cochain.
Hence  $L_2(L_1(T(\xi)\cup T(\eta)))$ is equal to
\[
  \sum_{\alpha\beta=\gamma}L_2^\alpha(T_1^\alpha(\xi_\alpha))\wedge
  \alpha(L_2^\beta(T_2^\beta(\eta_\beta))).
\]
To compute $L_2(L_1(T(\xi)\cup T(\eta)))$, it thus suffices to determine
\[
  L_2^\alpha(T_1^\alpha(\xi_\alpha))\wedge
  \alpha(L_2^\beta(T_2^\beta(\eta_\beta))),
\]
which defines a $(p+q)$-multivector field $Z$ supported in a
neighborhood of $V^\alpha\cap V^\beta$ in $V$.  By Equation
(\ref{eq:L-2}), one observes that when the restrictions of the
normal bundles $N^\alpha$ and $N^\beta$ to $V^\alpha\cap V^\beta$
have a nontrivial intersection, for instance along a coordinate
$x^0$,  then in Equation (\ref{eq:cup}), the derivative
$\frac{\partial}{\partial x^0}$ shows up in both
$L_2^\alpha(T_1^\alpha(\xi_\alpha))$ and
$\alpha(L_2^\beta(T_x^\beta(\eta_\beta)))$. Therefore  their wedge
product then has to vanish.  This argument shows that the
nontrivial contribution of the cup product $\xi\cup \eta$ comes
from those components, where $N^\alpha$ and $N^\beta$ do not have
a nontrivial intersection. By the following Lemma
(\ref{lem:comp-inter}) this implies that at a point $x\in
V^\alpha\cap V^\beta$ with $N^\alpha_x \cap N^\beta_x = \{ 0\}$
one has $T_xV^\alpha+T_xV^\beta=T_xV$ and therefore
$V^{\alpha\beta}=V^\alpha\cap V^\beta$.  This last condition by
Lemma \ref{lem:codim} is equivalent to
$\ell(\alpha)+\ell(\beta)=\ell(\gamma)$. When $V^\alpha\cap
V^\beta=V^{\alpha\beta}$ and $N^\alpha \cap N^\beta=\{0\}$, one
computes $L_3(Z)$ using the definition of $L_3$ and obtains
\[
\begin{split}
  L(T(\xi)\circ T(\eta))&=\sum_{\alpha\beta=\gamma \atop
  \ell(\alpha)+\ell(\beta)=\ell(\beta)}
  L_3\Big(L_2^\alpha(T_1^\alpha(\xi_\alpha))\wedge
  \alpha(L_2^\beta(T_2^\beta(\eta_\beta)))\Big)\\
  &=\sum_{\alpha\beta=\gamma \atop \ell(\alpha)+\ell(\beta)=\ell(\beta)}
  \xi_\alpha\cup \alpha(\eta_\beta).
\end{split}
\]
Note that on $V^\alpha\cap V^\beta=V^{\alpha\beta}$ one has
$\alpha(\eta_\beta)=\eta_\beta$. This finishes the proof of the claim.
\end{proof}
To end this section we finally show a lemma which already has been
used in the proof of the preceding result.
\begin{lemma}
\label{lem:comp-inter}
  Let $\alpha, \beta$ be two linear automorphisms on the real vector space
  $V$. Let $\langle - , -\rangle$ be a scalar product preserved by $\alpha$
  and $\beta$ and let $V^\alpha$, $V^\beta$ be the corresponding fixed point
  subspaces. If $V^\alpha+V^\beta=V$, then
  $V^\alpha\cap V^\beta=V^{\alpha\beta}$.
\end{lemma}

\begin{proof}
 Obviously, $V^\alpha\cap V^\beta\subset V^{\alpha\beta}$. It is enough to
 show that if $v\in V^{\alpha\beta}$, then $v\in V^\alpha\cap V^\beta$.

 Since $v\in V^{\alpha\beta}$, one has $\alpha\beta(v)=v$, hence
 $\beta(v)=\alpha^{-1}(v)$. Define $w=\beta(v)-v=\alpha^{-1}(v)-v$. We prove
 that $w$ is orthogonal to both $V^\alpha$ and $V^\beta$. For every
 $u\in V^\alpha$ one has
 \[
 \begin{split}
  \langle w,u \rangle &=\langle\alpha^{-1}(v)-v,u\rangle=
  \langle\alpha^{-1}(v), u\rangle-\langle v,u\rangle\\
  &=\langle v, \alpha(u)\rangle-\langle v,u\rangle=
  \langle v,u\rangle-\langle v,u\rangle=0,
 \end{split}
 \]
where in the first equality of the second line we have used the fact that
$\alpha$ preserves the metric $\langle - , - \rangle$, and in the second
equality of the second line we have used that $u$ is $\alpha$-invariant.
Therefore one concludes that $w$ is orthogonal to $V^\alpha$. Likewise
one shows that $w$ is orthogonal to $V^\beta$.
Therefore, $w$ is orthogonal to $V^\alpha +V^\beta=V$, hence $w$ has to be
$0$. This implies that $v$ is invariant under both $\alpha$ and $\beta$.
\end{proof}


%
%


%
%
%
%
\section{Cup product on the Hochschild cohomology of the deformed
         convolution algebra}
\label{Sec:defhoch}
In this section we compute the Hochschild cohomology  together with the cup
product of a formal deformation of the convolution algebra of a proper
\'etale groupoid $\grp$. For this we assume that the orbifold $X$ is
symplectic or in other words that $\grp_0$ carries a $\grp$-invariant 
symplectic form $\omega$, i.e., satisfying $s^*\omega=t^*\omega$. 
We let $\calA^\hbar$ be a $\grp$-invariant formal deformation quantization of
$\calA=\calC^{\infty}_{\grp_0}$, where the deformation parameter is
denoted by $\hbar$. This means that $\calA^\hbar$ is a $\grp$-sheaf over
$\grp_0$ and the associated crossed product $\calA^\hbar\rtimes\grp$
is a formal deformation of the convolution algebra, cf.~\cite{TanDQPPG}.

As a formal deformation the algebra $\calA^\hbar\rtimes\grp$
is filtered by powers of $\hbar$, i.e., 
$F_k(\calA^\hbar\rtimes\grp):=\hbar^k(\calA^\hbar\rtimes\grp)$ and we have
\begin{equation}
\label{def-quotient}
  \left. F_k\left(\calA^\hbar\rtimes\grp\right)\right\slash
  F_{k-1}\left(\calA^\hbar\rtimes\grp\right)\cong \calA\rtimes\grp.
\end{equation}
As usual, the Hochschild cochain complex is defined by
\[
  C^\bullet\big(\calA^\hbar\rtimes\grp,\calA^\hbar\rtimes\grp\big)
  :=\Hom_{\C[[t]]}\big( (\calA^\hbar\rtimes\grp )^{\hatotimes\bullet},
  \calA^\hbar\rtimes\grp\big),
\]
with differential $\beta$ defined with respect to the deformed convolution 
algebra.
The justification for this definition comes from Proposition 
\ref{Prop:DefConBorAlg}, which also shows that the cup-product, defined by 
\eqref{Eq:DefCup}, extends this complex to a differential graded algebra 
(DGA). The $\hbar$-adic filtration of $\calA^\hbar\rtimes\grp$ above induces a
complete and exhaustive filtration of the Hochschild complex. Since the 
product in $\calA^\hbar\rtimes\grp$ is a formal deformation of the 
convolution product, cf.~equation \eqref{def-quotient},
the associated spectral sequence has $E_0$-term just the undeformed Hochschild
complex of the convolution algebra.

This has the following useful consequence that we will use several times in 
the course of the argument: Suppose that $A_1^\hbar$ and $A_2^\hbar$ are 
formal deformations of the algebras $A_1$ and $A_2$, and
\[
  f: C^\bullet \big(A_1^\hbar,A_1^\hbar\big)\rightarrow 
  C^\bullet\big(A_2^\hbar,A_2^\hbar\big)
\]
is a morphism of filtered complexes. Then $f$ is a quasi-isomorphism, 
if it induces an isomorphism at level $E_1$.
The proof of this statement is a direct application of the 
Eilenberg--Moore spectral sequence comparison theorem, 
cf.~\cite[Thm. 5.5.11]{WeiIHA}.

Let us apply this to the following situation: consider the following 
subspace of the space of Hochschild cochains on 
$\calA^\hbar\rtimes\grp$:
\begin{displaymath}
\begin{split}
  C^k_\text{\tiny\rm loc}&\,\big(\calA^\hbar\rtimes\grp,
  \calA^\hbar\rtimes\grp\big):= \\
  & \Big\{ \Psi\in C^k (\calA^\hbar\rtimes\grp,\calA^\hbar\rtimes\grp) \mid
  \pi s \big( \supp \Psi(a_1,\ldots,a_k ) \big) \subset \bigcap_{i=1}^k
  \pi s ( \supp a_i) \Big\}.
\end{split}
\end{displaymath}
Here ${\rm supp}(a)$ denotes the support of a function. These are
the \textit{local} cochains with respect to the underlying
orbifold $X$. Notice that because of the convolution nature of the
algebra $\calA^\hbar\rtimes\grp$, which involves the action
of $\grp$, it is unreasonable to require locality with
respect to $\grp_0$ or $\grp_1$. The important point
now is:
\begin{proposition}
 The complex of local Hochschild cochains 
 $C^\bullet_\text{\tiny\rm loc}\big(\calA^\hbar\rtimes\grp,
 \calA^\hbar\rtimes\grp\big)$ is a subcomplex of
 $C^\bullet\left(\calA^\hbar\rtimes\grp,
 \calA^\hbar\rtimes\grp\right)$,
 and the canonical inclusion map is a quasi-isomorphism preserving
 cup-products.
\end{proposition}
\begin{proof}
The orbifold $X$ can be identified with the quotient space 
$\grp_0\slash\grp_1$. The deformed convolution product on 
$\calA^\hbar\rtimes\grp$ involves the local product on the 
$\grp$-sheaf $\calA^\hbar$ on $\grp_0$ and the groupoid action, and 
the product in turn defines the Hochschild complex as well as the 
cup-product. With this, it is easy to check that the
locality condition on $X$ is compatible with both the differential and 
the product.

To show that the canonical inclusion is a quasi-isomorphism, first observe 
that the map clearly respects the $\hbar$-adic filtration. It follows from 
Theorem IV that for the undeformed convolution algebra the local Hochschild 
cochain complex computes the same cohomology, since the vector fields clearly 
satisfy the locality condition. Therefore the inclusion map is a 
quasi-isomorphism at the $E_0$-level, and by the above, a quasi-isomorphism 
in general.
\end{proof}

\begin{remark}
In the following we will often consider the ring extension
\[
  \calA^{((\hbar))} :=
  \calA^\hbar \hatotimes_{\C[[\hbar]]} \C((\hbar),
\]
where $\C((\hbar))$ denotes the field of formal Laurent series in $\hbar$,
and will then regard $\calA^{((\hbar))}$ as an algebra over the ground field 
$\C((\hbar))$.
By standard results from Hochschild (co)homology theory one knows that
\begin{equation}
  H^\bullet\big(\calA^{((\hbar))}\rtimes\grp,
  \calA^{((\hbar))}\rtimes\grp\big) = 
  H^\bullet\big(\calA^{\hbar}\rtimes\grp,\calA^{\hbar}\rtimes\grp\big)
  \hatotimes_{\C[[\hbar]]} \C((\hbar) .
\end{equation}
In the remainder of this article we will tacitly make use of this fact.
\end{remark}
\subsection{Reduction to the \v{C}ech complex}
As in the undeformed case, the idea is to use a \v{C}ech complex
to compute the cohomology. For $U\subset X$, introduce
\[
  \mathcal{H}^k_{\grp,\hbar}(U):= \Hom_{\C[[\hbar]]}
  \big( \Gamma^\infty_\text{\tiny\rm cpt} 
  (U,\tilde{\calA}^\hbar_\text{\tiny\rm fc})^{\hatotimes k},
  \tilde{\calA}^\hbar_\text{\tiny \rm fc}(U)\big). 
\]
This is clearly a deformation of the sheaf 
$\calH^\bullet_\grp$. The sheaf $\calH^k_{\grp,\text{\tiny\rm loc},\hbar}$ 
is similarly defined. As in the undeformed case, we now have an obvious map
\[
  I^\hbar_\text{\tiny\rm loc}:
  C^\bullet_\text{\tiny\rm loc}\big( \calA^\hbar\rtimes\grp,
  \calA^\hbar\rtimes\grp\big)\rightarrow
  \calH_{\grp,\text{\tiny\rm loc},\hbar}^\bullet(X).
\]
\begin{proposition}
The map $I^\hbar_\text{\tiny\rm loc}$ is a quasi-isomorphism of DGA's.
\end{proposition}
\begin{proof}
By the Eilenberg-Moore spectral sequence, this follows from Theorem I.
\end{proof}

\subsection{Twisted cocycles on the formal Weyl algebra}
Our aim is to reduce the computation of Hochschild to sheaf
cohomology. The present section can be viewed as a stalkwise
computation. Let $V=\R^{2n}$ equipped with the standard symplectic
form $\omega$, and suppose that $\Gamma\subset Sp(V,\omega)$ is a
finite group acting on $V$ by linear symplectic transformations.
The action of an element $\gamma\in\Gamma$ induces a decomposition
$V=V^\gamma\oplus V^\perp$ into symplectic subspaces. Put
\[
  \ell(\gamma):=\dim(V^\perp)=\dim(V)-\dim(V^\gamma).
\] 
Let $\mathbb{W}_{2n}$ be the formal Weyl algebra, i.e.,
$\mathbb{W}_{2n}=\C[[y_1,\ldots,y_n]][[\hbar]]$ equipped with the
Moyal product 
\[
 f\star g=
 \sum_{k=0}^\infty\sum_{1\leq i_1,\ldots,i_k \leq n \atop 
 1 \leq j_1,\dots,j_k \leq n}
 \Pi^{i_1j_1}\cdots\Pi^{i_kj_k}\frac{\hbar^k}{k!}
 \frac{\partial^k f}{\partial y_1\ldots \partial y_k}
 \frac{\partial^k g}{\partial y_1\ldots \partial y_k},
\] 
where $\Pi:=\omega^{-1}$ is the Poisson tensor associated to $\omega$. 
With this product, the formal Weyl algebra $\mathbb{W}_{2n}$ is a unital 
algebra over $\C[[\hbar]]$. It is a formal deformation of the commutative 
algebra $\C[[y_1,\ldots,y_{2n}]]$. With an automorphism $\gamma\in\Gamma$, 
we can consider the Weyl algebra bimodule $\mathbb{W}_{2n,\gamma}$ which 
equals $\mathbb{W}_{2n}$ except for the fact that the right action of 
$\mathbb{W}_{2n}$ is twisted by $\gamma$. With this we have:
\begin{proposition}[cf.~\cite{pinczon}]
\label{inf-comp} 
 The twisted Hochschild cohomology is given by
 \[
   H^k\left(\mathbb{W}_{2n},\mathbb{W}_{2n,\gamma}\right)=
   \begin{cases}
     \C[[\hbar]],& \text{for $k=\ell(\gamma)$,}\\
     0,&\text{else.}
   \end{cases}
 \] 
 There exists a generator $\Psi_\gamma$ in the reduced Hochschild complex 
 satisfying 
 \[
   \left.\Psi_\gamma\right|_{\Lambda V^*}=
   \left(\Pi^\perp_\gamma\right)^{\ell(\gamma)/2}.
 \]
 In fact, $\Psi$ is, up to a coboundary, uniquely determined by this property.
\end{proposition}
\begin{proof}
 The first part of the Proposition is essentially well-known, 
 cf.~\cite{AleFarLamSolHIAWAGF,alvarez}. It is conveniently proved by the 
 Koszul resolution of the Weyl algebra
 \[
   0\longleftarrow\mathbb{W}_{2n}\stackrel{\star}{\longleftarrow}
   \mathbb{W}_{2n}\otimes \mathbb{W}_{2n}^\text{\tiny \rm op}
   \stackrel{\partial}{\longleftarrow} K_1
   \stackrel{\partial}{\longleftarrow} K_2
   \stackrel{\partial}{\longleftarrow}\ldots
 \]
 where 
 \[
  K_p:=\mathbb{W}_{2n}\otimes \Lambda^pV^*\otimes 
  \mathbb{W}^\text{\tiny \rm op}_{2n},
 \] 
 and where the differential $\partial:K_p\rightarrow K_{p-1}$ is defined by
\begin{displaymath}
\begin{split}
  \partial & (a_1\otimes a_2\otimes dy_{i_1}\wedge\ldots\wedge dy_{i_p}):=\\
  &\sum_{j=1}^p(-1)^j 
  \left((y_{i_j}\star a_1)\otimes a_2-a_1\otimes (a_2\star y_{i_j})\right) 
  dy_{i_1}\wedge\ldots\wedge\widehat{dy}_{i_j}\wedge\ldots\wedge dy_{i_p},
\end{split}
\end{displaymath}
with respect to a Darboux basis of $V$, i.e., 
$\omega(y_i,y_{i+n})=1$ and zero otherwise.

To compute the Hochschild cohomology, we take 
$\Hom_{\mathbb{W}_{2n}}\left(-,\mathbb{W}_{2n,\gamma}\right)$ to obtain 
the complex 
\begin{equation}
\label{koszul}
  K^p_\gamma:=\Lambda^p V\otimes\mathbb{W}_{2n},
\end{equation}
with differential $d_\gamma:K^p\rightarrow K^{p+1}$ given by 
\[
  d_\gamma\left (a\otimes y_{i_1}\wedge\ldots\wedge y_{i_{p+1}}\right)=
  \sum_{j=1}^{2n}(-1)^j\left(y_j\star a-a\star_\gamma y_j\right)y_j\wedge 
  y_{i_1}\wedge\ldots\wedge y_{i_p}.
\]
The cohomology of this complex can easily be computed using the spectral 
sequence of the $\hbar$-adic filtration. In degree zero one finds the 
ordinary, i.e.~commutative Koszul complex and therefore we find 
\[
  E_1^{p,q}=\Lambda^{p+q}V^\perp.
\] 
The differential $d_1:E_1^{p,q}\rightarrow E_1^{p+1,q}$ is given by the 
Poisson cohomology differential, which has trivial cohomology except in 
maximal degree and therefore 
\[ 
  E_2^{p,q}=
  \begin{cases} 
    \Lambda^{\ell(\gamma)}V^\perp,& \text{for $p+q=\ell(\gamma)$,}\\ 
    0,&\text{else}.
  \end{cases}
\]
The spectral sequence degenerates at this point and the first part of the 
Proposition is proved. The second part is as in \cite{pinczon}: the Koszul 
complex is naturally a subcomplex of the reduced Bar complex 
$(K_\bullet,\partial)\subset (B_\bullet^\text{\tiny \rm red},b)$, where
\[
  B_k^\text{\tiny \rm red}=
  \mathbb{W}_{2n}\otimes\left( \mathbb{W}_{2n}
  \slash\C[[\hbar]]\right)^{\otimes k}\otimes\mathbb{W}_{2n},
\]
and the embedding is induced by the natural inclusion 
$V^*\hookrightarrow\mathbb{W}_{2n}$ as degree one homogeneous polynomials. 
This leads to a natural projection 
\[
  R_V:\left(C^\bullet_\text{\tiny \rm red}
  \left(\mathbb{W}_{2n},\mathbb{W}_{2n}\right),
  \beta_\gamma\right)\rightarrow \left(K^\bullet,d_\gamma\right)
\] 
given by restricting cochains to $\Lambda V^*$. It is easily checked that 
$\left(\Pi^\perp_\gamma\right)^{\ell(\gamma)/2}$ defines a cocycle of 
degree $\ell(\gamma)$ in the complex $\left(K^\bullet,d_\gamma\right)$, 
and the statement follows.
\end{proof}
Let $\mu_\gamma:K^\bullet_\gamma\rightarrow\C[[\hbar]]$ be the morphism 
defined by
\[
  \mu_\gamma(a\otimes v_{i_1}\wedge\ldots\wedge v_{i_k}):=
  a(0)\left(\omega^\perp_\gamma\right)^{\ell(\gamma/2)}
  \left(v_{i_1},\ldots,v_{i_k}\right).
\]
Clearly, this map is only nontrivial in degree $\ell(\gamma)$ and maps the 
differential $d_\gamma$ on $K^\bullet_\gamma$ to zero. Define
\[
  P_\gamma:C^\bullet_\text{\tiny \rm red}
  \left(\mathbb{W},\mathbb{W}_\gamma\right)\rightarrow \C[[\hbar]]
\] 
to be $P_\gamma:=\mu\circ R_V$. On the other hand, choosing 
$\Psi_\gamma$ as in the proposition defines a morphism 
$I_{\Psi_\gamma}:\C[[\hbar]][\ell(\gamma)]\rightarrow 
C^\bullet_\text{\tiny \rm red}\left(\mathbb{W},\mathbb{W}_\gamma\right)$. 
The argument in the proof of the proposition then shows:
\begin{corollary}
  The inclusion $I_{\Psi_\gamma}$ and the projection $P_\gamma$ are 
  quasi-isomorphisms satisfying $P_\gamma\circ I_{\Psi_\gamma}= \id$.
\end{corollary}
Finally, we come to the full crossed product 
$\mathbb{W}_{2n}\rtimes\Gamma$. As usual, we have
\[
  H^\bullet\left(\mathbb{W}_{2n}\rtimes\Gamma,
  \mathbb{W}_{2n}\rtimes\Gamma\right)\cong
  \left(\bigoplus_{\gamma\in\Gamma}
  H^\bullet(\mathbb{W}_{2n},\mathbb{W}_{2n,\gamma})\right)^\Gamma,
\]
where the $\Gamma$-action is as explained in Section \ref{Sec:clglobquot}.
\begin{corollary}
\label{q-i-crs}
  The generators $\Psi_\gamma,~\gamma\in\Gamma$ satisfy
\[
  \gamma_1\cdot\Psi_{\gamma_2}-\Psi_{\gamma_1\gamma_2\gamma_1^{-1}}=
  \text{\rm exact}.
\]
Therefore they define a canonical isomorphism
\[
  H^\bullet\left(\mathbb{W}_{2n}\rtimes\Gamma,
  \mathbb{W}_{2n}\rtimes\Gamma\right)=
  \bigoplus_{\left<\gamma\right>\subset\Gamma}
  \C[[\hbar]][\ell(\left<\gamma\right>)].
\]
\end{corollary}
\begin{proof}
Restricting to $\Lambda^\bullet V^*$, we find
\begin{displaymath}
\begin{split}
 \left(\gamma_1\cdot\Psi_{\gamma_2}- \Psi_{\gamma_1\gamma_2\gamma_1^{-1}}
 \right)_{|\Lambda^\bullet V^*}
 &=\gamma_1\cdot\left(\Pi^\perp_{\gamma_2}\right)^{\ell(\gamma_2)/2}-
 \left(\Pi^\perp_{\gamma_1\gamma_2\gamma_1^{-1}}
 \right)^{\ell(\gamma_1\gamma_2\gamma_1^{-1})/2}\\
 &=0
\end{split}
\end{displaymath}
in $K^{\ell(\gamma_2)}_{\gamma_1\gamma_2\gamma_1^{-1}}$. By the
argument above, the cocycles $\gamma_1\cdot\Psi_{\gamma_2}$ and
$\Psi_{\gamma_1\gamma_2\gamma_1^{-1}}$ therefore differ by a
coboundary, and the result follows.
\end{proof}

\subsection{The Fedosov--Weinstein--Xu resolution over $\mathsf{B}_0$} 
Let $\mathsf{B}_0$ be the space of loops in $\grp$: 
\[
  \mathsf{B}_0:=\{g\in\grp_1 \mid s(g)=t(g)\}.
\] 
We recall from \cite{NeuPflPosTanHFDPELG} that the canonical inclusion 
$\iota:\mathsf{B}_0\hookrightarrow\grp_1$ gives $\mathsf{B}_0$ a 
symplectic form by pull-back. Denote by 
$\calA^\hbar_{\mathsf{B}_0}:=\iota^{-1}\calA^\hbar$ the pull-back of the 
deformation quantization of $\grp$; this is not quite a deformation 
quantization of $(\mathsf{B}_0,\iota^*\omega)$, because it involves the 
germ of $\mathsf{B}_0$ inside $\grp_1$. Recall from \cite{ppt} that the 
sheaf $\calA^\hbar_{\mathsf{B}_0}$ has a canonical local automorphism, 
denoted $\theta$, coming from the fact that $\mathsf{B}_0$ has a cyclic 
structure \cite{CraCCEGGC}. This enables us to define the following 
complex of sheaves on $\mathsf{B}_0$:
\[
  \calC^0\stackrel{\beta_{\theta}}{\longrightarrow}\calC^1 
  \stackrel{\beta_{\theta}}{\longrightarrow}\ldots
\]
where
\[
  \calC^k:=
  \underline{\Hom}_{\C[[\hbar]]}\left(\left(\calA^\hbar_{\mathsf{B}_0}
  \right)^{\hatotimes k},\calA^\hbar_{\mathsf{B}_0,\theta}\right)
\]
is the sheaf of Hochschild $k$-cochains, and 
$\beta_\theta:\calC^k\rightarrow\calC^{k+1}$ is the twisted Hochschild 
coboundary. 
We will now write down a resolution of this complex of sheaves. For this, 
let $\calW_\grp$ be the bundle of Weyl algebras over $\grp_0$.
This is just the bundle $F_{\Sp,\grp_0}\times_{\Sp}\mathbb{W}$ associated to 
the symplectic frame bundle over $\grp_0$ with typical fiber $\mathbb{W}$. This 
construction shows that $\calW_\grp$ carries a canonical action of the groupoid 
$\grp$. In the following we will denote its sheaf of sections by the same 
symbol $\calW_\grp$.
\begin{proposition}[cf.\ \cite{FedTIDQ}]
\label{fedosov-resolution}
  On $\mathsf{B}_0$ there exists a resolution
\begin{equation}
\label{resolution}
  0\longrightarrow \calA^\hbar_{\mathsf{B}_0}\longrightarrow
  \Omega^0_{\mathsf{B}_0}\otimes\calW_\grp\stackrel{D}{\longrightarrow}
  \Omega^1_{\mathsf{B}_0}\otimes\calW_\grp\stackrel{D}{\longrightarrow}\ldots,
\end{equation}
where $D$ is a Fedosov connection on $\calW_\grp$.
\end{proposition}
\begin{proof}
Let us first construct the Fedosov differential $D$. The sequence of maps 
\[
  T\mathsf{B}_0\stackrel{\tilde{\omega}_0}{\longrightarrow}T^*\mathsf{B}_0
  \rightarrow T^*\grp\rightarrow\calW_\grp
\]
determines an element $A_0\in\Omega^1(\mathsf{B}_0,\calW_\grp)$. 
On easily verifies that 
\[
  [A_0,A_0]=\omega_0\in\Omega^2_{\mathsf{B}_0}\otimes\C [[\hbar]],
\] 
which is central in $\calW_\grp$. Therefore, $\delta={\rm ad}(A_0)$ defines a 
differential on $\Omega^\bullet (\mathsf{B}_0,\calW_\grp)$, and we have 
\[
  H^k\left(\Omega^\bullet _{\mathsf{B}_0}\otimes \calW_\grp,{\rm ad}(A_0)
  \right)=
 \begin{cases}
   \Omega^0_{\mathsf{B}_0}\otimes\calW_N,& \text{for $k=0$},\\
   0, & \text{for $k\neq 0$},
 \end{cases}
\]
because ${\rm ad}(A_0)$ is simply a Koszul differential in the tangential 
directions along $\mathsf{B}_0$.

Choose a symplectic connection $\nabla_{\mathsf{B}_0}$ on $\mathsf{B}_0$ and a 
symplectic connection $\nabla_N$ on the normal bundle 
$N\rightarrow\mathsf{B}_0$. We will consider connections on $\calW_\grp$ of 
the form 
\[
  D=\delta+\nabla+{\rm ad}(A_{\bullet}),
\] 
where $\nabla=(\nabla_{\mathsf{B}_0}\otimes 1+1\otimes \nabla_N)$ and 
$A\in\Omega^1(\mathsf{B}_0,\calW_\grp)$ has $\deg(A)\geq 2$. Notice that 
$\deg(\delta)=0$ and $\deg(\nabla)=1$. Such a connection has Weyl curvature 
given by 
\[
  \Omega=\omega_0+\tilde{R}+\nabla A+\frac{1}{2}[A,A].
\]
By the usual Fedosov method, we can find $A$ such that $\Omega$ is central, 
i.e., $\Omega\in\Omega^2_{\mathsf{B}_0}\otimes\C[[\hbar]]$. Since the 
differential $D$ is a deformation of the Koszul complex above, acyclicity 
of the sequence \eqref{resolution} follows. This shows the existence of the 
resolution 
\[
  0\longrightarrow\Omega^0_{\mathsf{B}_0}\otimes\calW_N\longrightarrow 
  \Omega^0_{\mathsf{B}_0}\otimes\calW_\grp\stackrel{D}{\longrightarrow}
  \Omega^1_{\mathsf{B}_0}\otimes\calW_\grp\stackrel{D}{\longrightarrow}\ldots,
\]
so it remains to construct an isomorphism 
$\Omega^0_{\mathsf{B}_0}\otimes\calW_N\cong\iota^{-1}\calA^\hbar_\grp$. 
This is done by minimal coupling in \cite{FedTIDQ}.
\end{proof}
Consider now the following double complex 
$(\calC^{\bullet,\bullet},\beta_\theta,D)$ of sheaves:
\begin{displaymath}
\xymatrix{
 \vdots & \vdots & \vdots \\
 \Omega^0_{\mathsf{B}_0}\otimes\calC^1\ar[u]^{\beta_\theta}\ar[r]^{D} &
 \Omega^1_{\mathsf{B}_0}\otimes\calC^1\ar[u]^{\beta_\theta}\ar[r]^{D} &
 \Omega^2_{\mathsf{B}_0}\otimes\calC^1\ar[u]^{\beta_\theta}\ar[r]^{D} &
 \ldots \\
 \Omega^0_{\mathsf{B}_0}\otimes\calC^0\ar[u]^{\beta_\theta}\ar[r]^{D} &
 \Omega^1_{\mathsf{B}_0}\otimes\calC^0\ar[u]^{\beta_\theta}\ar[r]^{D} &
 \Omega^2_{\mathsf{B}_0}\otimes\calC^0\ar[u]^{\beta_\theta}\ar[r]^{D} &
 \ldots
}
\end{displaymath}
where $\calC^\bullet$ is the sheaf of formal power series of polydifferential 
operators on $\calW_{\grp}$. This complex is a twisted version of 
the Fedosov--Weinstein--Xu resolution considered in \cite{DolHCRCFT}.
\begin{proposition}
\label{coh-fwx}
 The total cohomology of this complex is given by 
 \[
   H^k({\rm Tot}\left(\calC^{\bullet,\bullet}\right),\beta_\theta+D)=
  \mathbb{H}^k(\calC^\bullet,\beta_\theta).
\]
\end{proposition}
\begin{proof}
Consider the spectral sequence by filtering the total complex by rows. This 
yields 
\[ 
  E_1^{p,q}=H_v^p(\calC^{\bullet,q})=
  \begin{cases} 
    \Omega^0_{\mathsf{B}_0}\otimes\calC^q, & \text{for $p=0$},\\
    0, & \text{for $p\neq 0$},
  \end{cases}
\] 
since the resolution \eqref{resolution} is acyclic. At the second stage 
\[
  E_2^{0,q}=\mathbb{H}^q(\calC^\bullet,\beta_\theta)\Rightarrow H^q
  \left({\rm Tot}\left(\calC^{\bullet,\bullet}\right)\right).
\] 
Since the spectral sequence collapses at this point, this proves the statement.
\end{proof}
Of course the cohomology sheaf of the vertical complex is computed
by Proposition \ref{inf-comp}. If we replace the vertical complex
by its reduced counterpart, we get a natural projection
\[
  R:\Gamma\left(\mathsf{B}_0,\calC^\ell\right)\rightarrow
  \Gamma\left(\mathsf{B}_0,\Omega^\ell_{\grp}\right),
\]
where $\ell$ is the locally constant function on $\mathsf{B}_0$ defined in 
Section \ref{Sec:outline}. The transversal Poisson structure induced by 
$\omega\in\Omega^2(\grp_0)$ induces a section
\[
  \left(\Pi_\theta^\perp\right)^{\ell/2}\in
  \Gamma\left(\mathsf{B}_0,\Omega^\ell_{\grp}\right),
\]
and we choose $\Psi\in \Gamma\left(\mathsf{B}_0,\calC^\ell\right)$ which generates 
the fiberwise vertical cohomology and projects onto the section above.
\begin{proposition}
\label{embd-n}
The map $I_\Psi$ extends to a morphism 
\[ 
 I_\Psi:\left(\Omega^\bullet_{\mathsf{B}_0}\otimes\C [[\hbar]][\ell],d\right)
 \rightarrow \left({\rm Tot}^\bullet\left(\calC^{\bullet,\bullet}\right),
 D+\beta_\theta\right)
\]  
of cochain complexes of sheaves by the formula 
\[
  I_\Psi(\alpha):=\alpha\otimes \Psi.
\] 
In fact, this is a quasi-isomorphism.
\end{proposition}
\begin{proof}
  Let us first prove that $I_\Psi$ is a map of cochain complexes of sheaves. Consider 
  the projection map 
  $P:\Omega^\bullet_{\mathsf{B}_0}\otimes\calC^\bullet_{red}\rightarrow
  \Omega^\bullet_{\mathsf{B}_0}\otimes\C((t))$. It satisfies $P\circ I_\Psi=\id$. 
  We therefore have to show that 
  \[
   P\circ D\circ I_\Psi=d.
  \] 
  Since the statement of the proposition is a local statement we can write the 
  Fedosov connection as 
  \[
   D=d+{\rm ad}(A),
  \] 
  with $A\in\Omega^1_{\mathsf{B}_0}\otimes \calW_\grp=
  \Omega^1_{\mathsf{B}_0}\otimes\calC^0$. It then follows easily that 
  $D:\Omega^k_{\mathsf{B}_0}\otimes\calC^\bullet\rightarrow 
  \Omega^{k+1}_{\mathsf{B}_0}\otimes\calC^\bullet$ is given by 
  \[
   D=d+[\beta A,~],
  \] 
  where $[~,~]$ is the Gerstenhaber bracket on the Hochschild cochain complex. 
  We now compute
 \begin{displaymath}
 \begin{split}
   PDI_\Psi(\alpha)-d\alpha&= \alpha\otimes P\left([\beta A,\Psi]\right)\\
   &=\alpha\otimes P\left(\beta\left([A,\Psi]\right)-[A,\beta \Psi]\right)\\&=0,
 \end{split}
 \end{displaymath}
 where we have used that $\Psi$ is a cocycle, i.e., $\beta \Psi=0$, and the fact 
 that $P$ maps coboundaries to zero. This proves that $I_\Psi$ is a map of cochain 
 complexes of sheaves.
\end{proof}
Using Proposition \ref{coh-fwx} we now find:
\begin{corollary}
\label{fxf}
  There is a natural isomorphism
  \[
   \mathbb{H}^\bullet(\calC_{\mathsf{B}_0}^\bullet,\beta_\theta)\cong 
   H^{\bullet-\ell}\left(\mathsf{B}_0,\C[[\hbar]]\right).
  \]
\end{corollary}
Finally, notice that $\calC^{\bullet,\bullet}_{\mathsf{B}_0}$ is in a natural way a 
double complex of $\Lambda(\grp)$-sheaves on
$\mathsf{B}_0$, because the Fedosov resolution of 
Proposition \ref{fedosov-resolution} carries a natural $\grp$-action.
Therefore, we can define the following sheaf on $\tilde{X}$:
\[
 \calC^{\bullet,\bullet}_{\tilde{X}}(\tilde{U}):=
 \calC^{\bullet,\bullet}_{\mathsf{B}_0}
 (\tilde{\pi}^{-1}(\tilde{U}))^{\Lambda(\grp)_{|\tilde{U}}}.
\]
Because $\grp$, and therefore also $\Lambda(\grp)$, is proper, it follows from 
Corollary \ref{fxf} that its hypercohomology is given by
\begin{equation}
\label{hyp-fwx}
  \mathbb{H}^\bullet(\calC_{\tilde{X}}^{\bullet,\bullet},\beta_\theta)\cong 
  H^{\bullet-\ell}\left(\tilde{X},\C[[\hbar]]\right).
\end{equation}
\subsection{Local computations}
\label{loc-comp}
 In this section we will perform several explicit computations in some open 
 orbifold  charts. This suffices to prove the result in the case of a global 
 quotient orbifold. The general case is treated in the next section.

 Let $U\subset\R^{2n}$ be an open orbifold chart with a finite group $\Gamma_U$ 
 acting  by linear symplectic transformations, so that we have 
 $U\slash\Gamma_U\subset X$.
\begin{proposition}
\label{gl-q}
  There exist a natural quasi-isomorphism 
  \[
   \calH^\bullet_{M\rtimes\Gamma,\text{\tiny\rm loc},\hbar}(M/\Gamma)\rightarrow
   \calC^{\bullet,\bullet}_{\widetilde{M\slash\Gamma}}\left(
   \widetilde{M\slash\Gamma} \right).
  \]
\end{proposition}
\begin{proof}
We first use the natural map
\[
  \calH^\bullet_{\grp,\text{\tiny\rm loc},\hbar}(U\slash\Gamma_U)\rightarrow 
  C^\bullet_\text{\tiny\rm loc}\left(\calA^\hbar(U)\rtimes\Gamma_U,
  \calA^\hbar(U)\rtimes\Gamma_U\right)
\]
of the ``deformed version'' of Theorem IIIb. As in the undeformed case, there 
is a quasi-isomorphism 
\[
  L_1^\hbar:C^\bullet_\text{\tiny\rm loc}\left(\calA^\hbar(U)\rtimes\Gamma_U,  
  \calA^\hbar(U)\rtimes\Gamma_U\right)\rightarrow 
  C^\bullet_\text{\tiny\rm loc}\left(\calA^\hbar_\text{\tiny\rm cpt}(U), 
  \calA^\hbar_\text{\tiny\rm cpt}(U)\rtimes\Gamma_U\right)^{\Gamma_U},
\]
given by the same formula as for $L_1$. The right hand side is the space of
$\Gamma_U$-invariants of a complex which decomposes into
\[
  \bigoplus_{\gamma\in\Gamma_U}C^\bullet_\text{\tiny\rm loc}\left(\calA^\hbar(U),
  \calA^\hbar(U)_\gamma\right).
\]
There is a natural morphism
\[
  C^\bullet_\text{\tiny\rm loc}\left(\calA^\hbar(U),\calA^\hbar(U)_\gamma\right)
  \rightarrow
  C^\bullet_\text{\tiny\rm loc}\left(\iota^{-1}\calA^\hbar(U^\gamma),
  \iota^{-1}\calA^\hbar(U^\gamma)_\gamma\right),
\] 
where
$\iota^{-1}\calA^\hbar(U^\gamma):=\Gamma\left(U^\gamma,\iota^{-1}\calA^\hbar
\right)$
is by definition the algebra given by the jets along the embedding
$\iota:U^\gamma\hookrightarrow U$. Indeed the locality condition
for cochains $\Psi\in C^k_\text{\tiny\rm loc}$ states in this case that the
value $\Psi(f_1,\ldots,f_k)(x)$ can only depend on the germs of
$f_1,\ldots,f_k$ at the points of the $\gamma$-orbit of $x\in U$.
Restricted to $U^\gamma\subset U$, such cochains therefore
preserve the subalgebra
$\iota^{-1}\calA^\hbar(U^\gamma)\subset\calA^\hbar(U)$, and the
restriction map above is well-defined. Even stronger, it is a
quasi-isomorphism because on $E_0$-level of the spectral sequence
associated to the filtration by powers of $\hbar$ this map is
simply given by localization, which we already know to be a
quasi-isomorphism, cf.~Section \ref{Sec:clglobquot}.

As remarked above, local cochains, in the sense defined above, are
truly local with respect to $U^\gamma$, because points in
$U^\gamma$ by definition have a trivial $\gamma$-orbit. From this
we see that there is a canonical isomorphism
\[
  C^\bullet_\text{\tiny\rm loc}\left(\iota^{-1}
  \calA^\hbar_\text{\tiny\rm cpt}(U^\gamma), 
  \iota^{-1}\calA^\hbar_\text{\tiny\rm cpt}(U^\gamma)_\gamma\right)
  \cong\calC^\bullet(U^\gamma),
\]
compatible with differentials. Taking the sum over all $\gamma\in\Gamma_U$,
and taking $\Gamma_U$-invariants, defines the map of the proposition. As the 
argument shows, it is a quasi-isomorphism.
\end{proof}
Using the Fedosov--Weinstein--Xu resolution, this result suffices to compute the 
Hochschild cohomology for a global quotient symplectic orbifold:
\begin{corollary}
\label{gl-q-vs}
  For a global quotient $X=M\slash\Gamma$ of a finite group $\Gamma$ acting on a 
  symplectic manifold $M$, there is a natural isomorphism
  \[
    H^\bullet(\calA^{((\hbar))}_\text{\tiny\rm cpt}(M)\rtimes\Gamma,
    \calA^{((\hbar))}_\text{\tiny\rm cpt}\rtimes\Gamma)\cong
    \bigoplus_{(\gamma)\subset\Gamma}
    H^{\bullet-\ell_\gamma}_{Z(\gamma)}(M^\gamma,\C((\hbar))).
\]
\end{corollary}
\begin{proof}
This follows from the isomorphism \eqref{hyp-fwx}.
\end{proof}
Next, we consider the cup-product. An easy computation shows that the map 
$L_1^\hbar$ induces the following product on the complex
$C^\bullet_\text{\tiny\rm loc}\left(\calA^\hbar_\text{\tiny\rm cpt}(U),
\calA^\hbar_\text{\tiny\rm cpt}(U)\rtimes\Gamma_U\right)^{\Gamma_U}$:
\[
  (\psi\bullet\phi)_\gamma:=
  \sum_{\gamma_1\gamma_2=\gamma}\psi_{\gamma_1}\cup_\text{\tiny\rm tw}
  \phi_{\gamma_2}
\]
where the map
\[
  \cup_\text{\tiny\rm tw}: C^k(\calA^\hbar(U),\calA^\hbar(U)_{\gamma_1})
  \times C^l(\calA^\hbar(U),\calA^\hbar(U)_{\gamma_2})\rightarrow 
  C^{k+l}(\calA^\hbar(U),\calA^\hbar(U)_{\gamma_1\gamma_2})
\]
is defined as
\[
  \left(\psi_{\gamma_1}\cup_\text{\tiny\rm tw}\phi_{\gamma_2}\right)
  (a_1,\ldots,a_{k+l}):=
  \psi_{\gamma_1}(a_1,\ldots,a_k)\gamma_1\phi_{\gamma_2}(a_{k+1},\ldots,a_{k+l}).
\]
Indeed one easily checks that
\[
  \beta_{\gamma_1\gamma_2}\left(\psi_{\gamma_1}\cup_\text{\tiny\rm tw}
  \phi_{\gamma_2}\right)=
  (\beta_{\gamma_1}\psi_{\gamma_1})\cup_\text{\tiny\rm tw}\phi_{\gamma_2}
  +(-1)^{deg(\psi)}\psi_{\gamma_1}\cup_\text{\tiny\rm tw}
  (\beta_{\gamma_2}\phi_{\gamma_2}).
\]
Restricting to $\tilde{U}=\bigsqcup_{\gamma}U^\gamma$, this induces the 
following product on the Fedosov--Weinstein--Xu resolution 
$\calC^{\bullet,\bullet}_{\tilde{U}\slash\Gamma_U}({\tilde{U}\slash\Gamma_U})$:
\[
  ((\alpha\otimes\psi)\bullet(\beta\otimes\phi))_\gamma:=
  \sum_{\gamma_1\gamma_2=\gamma}\left(\iota^*_{\gamma}(\alpha_{\gamma_1})\wedge
  \iota^*_\gamma(\beta_{\gamma_2})\right)\otimes
  \left(\iota_\gamma^*(\psi_{\gamma_1})\cup_\text{\tiny\rm tw}
  (\iota^*_\gamma(\phi_{\gamma_2})\right),
\]
where an element of 
$\calC^{\bullet,\bullet}_{\tilde{U}\slash\Gamma_U}({\tilde{U}\slash\Gamma_U})$ 
is written as
$\alpha\otimes\psi=\sum_\gamma\alpha_\gamma\otimes\psi_\gamma$, with 
$\alpha_\gamma\in\Omega^\bullet(U^\gamma)$ and $\psi_\gamma$ is a local section 
the sheaf of Hochschild cocycle on $\calW_{\grp}$ over $U^\gamma\subset U$. 
We therefore have:
\begin{proposition}
 The map of Proposition \ref{gl-q} is compatible with products that means 
 defines a quasi-isomorphism of sheaves of DGA's on $U\slash\Gamma_U$.
\end{proposition}
For a global quotient orbifold, this leads immediately to:
\begin{corollary}
\label{cpp-lc}
Under the isomorphism of Corollary \ref{gl-q-vs}, the cup-product is given by
\[
  \alpha\bullet\beta=\sum_{{\tiny \begin{array}{c}\gamma_1\gamma_2=\gamma\\ 
  \ell(\gamma_1)+\ell(\gamma_2)=\ell(\gamma_1\gamma_2)\end{array}}}
  \iota_{\gamma}^*\alpha_{\gamma_1}\wedge\iota_{\gamma_2}^*\alpha_{\gamma_2}.
\]
\end{corollary}
\begin{proof}
The isomorphism of Corollary \ref{gl-q-vs} is induced by the quasi-isomorphism
\[
  I_\Psi:\left(\Omega^\bullet_{\tilde{X}}\otimes\C[[\hbar]],d\right)\rightarrow
  \left(\calC_{\tilde{X}}^{\bullet,\bullet},D+\beta_\text{\tiny\rm tw}\right)
\]
of Proposition \ref{embd-n}. We therefore find
\[
  I_\Psi(\alpha)\bullet I_\Psi(\beta)=
  \sum_{\gamma_1\gamma_2=\gamma}\left(\iota^*_{\gamma}(\alpha_{\gamma_1})\wedge
  \iota^*_\gamma(\beta_{\gamma_2})\right)\otimes\left(
  \iota_\gamma^*(\Psi_{\gamma_1})\cup_\text{\tiny\rm tw}
  (\iota^*_\gamma(\Psi_{\gamma_2})\right).
\]
We will now consider the second component 
$\iota_\gamma^*(\Psi_{\gamma_1})\cup_\text{\tiny\rm tw}
(\iota^*_\gamma(\Psi_{\gamma_2})$ for a moment.
An easy calculation shows that in the Koszul complex \eqref{koszul}, the product
\[
  \cup_\text{\tiny\rm tw}:K^p_{\gamma_1}\times K^q_\gamma\rightarrow 
  K^{p+q}_{\gamma_1\gamma_2}
\] 
is given by
\[
  \left( a_1\otimes y_{I_q}\right)\cup_\text{\tiny\rm tw}\left(
  a_2\otimes y_{I_q}\right)=
  (a_1\gamma_1 a_2)\otimes y_{I_p}\wedge\gamma_1 y_{I_q},
\]
where $I_p$ and $I_q$ are multi-indices of length $p$ resp.~$q$.
Therefore,
\[
  \left(\Pi^\perp_{\gamma_1}\right)^{\ell(\gamma_1)/2}\cup_\text{\tiny\rm tw}
  \left(\Pi^\perp_{\gamma_2}\right)^{\ell(\gamma_2)/2}
  =
  \begin{cases} 
    \left(\Pi_{\gamma_1\gamma_2}^\perp\right)^{\ell(\gamma_1\gamma_2)/2},&
    \text{if $\ell(\gamma_1)+\ell(\gamma_2)=\ell(\gamma_1\gamma_2)$},\\
    0,&\text{else.}
  \end{cases}
\]
In the reduced Hochschild complex, this gives
\[
  \Psi_{\gamma_1}\cup_\text{\tiny\rm tw}\Psi_{\gamma_2}=
  \begin{cases}
    \Psi_{\gamma_1\gamma_2}+\text{exact},&
    \text{if $\ell(\gamma_1)+\ell(\gamma_2)=\ell(\gamma_1\gamma_2)$},\\ 
    \text{exact},&\text{else.}
  \end{cases}
\]
Taking cohomology we find the product as stated above.
\end{proof}
\subsection{The general case}\label{defhoch}
Recall that we have spaces and morphisms as in the following
diagram:
\[
  \xymatrix{\mathsf{B}_0\ar[r]\ar[d]^{\tilde{\pi}}&\mathsf{G_0}\ar[d]^\pi\\
  \tilde{X}\ar[r]^\psi&X}
\]
As in \cite{chen-hu:de-rham}, define the space 
\[
  \mathsf{S}^1_{\grp}:=\{(g_1,g_2)\in 
  \grp_1\times\grp_1 \mid s(g_1)=t(g_1)=s(g_2)=t(g_2)\}.
\] 
It comes equipped with three maps 
$\pr_1,m,\pr_2:\mathsf{S}^1_{\grp}\rightarrow\mathsf{B}_0$, where 
$\pr_1(g_1,g_2)=g_1$, $m(g_1,g_2)=g_1g_2$ and $\pr_2(g_1,g_2)=g_2$. For 
differential forms $\alpha,\beta\in\Omega^\bullet(\mathsf{B}_0)$, define the 
following product:
\begin{equation}
\label{orb-prod}
  \alpha\bullet\beta=\int_{m_\ell}\pr_1^*\alpha\wedge \pr_2^*\beta,
\end{equation}
where
$m_\ell:\mathsf{S}^1_{\grp,\ell}\rightarrow\mathsf{B}_0$ is
the restriction of the multiplication map and
\[
  \mathsf{S}^1_{\grp,\ell}:=
  \{(g_1,g_2)\in\mathsf{S}^1_\grp \mid \ell(g_1)+\ell(g_2)=\ell(g_1g_2)\}.
\]
Both $\mathsf{S}^1_{\grp,\ell}$ and $\mathsf{B}_0$ carry a natural action of 
$\grp$ by conjugating loops and the quotient
$\mathsf{B}_0\slash\grp=\tilde{X}$. We have 
\[
  \Omega^\bullet(\tilde{X})=\Omega^\bullet(\mathsf{B}_0)^\grp,
\] 
and the product \eqref{orb-prod} defines an associative graded product on  
$\Omega^\bullet(\tilde{X})$. Together with the de Rham differential, it turns
$\Omega^\bullet(\tilde{X})$ into a differential graded algebra. Of course, 
$\Omega^\bullet(\tilde{X})$ is the global sections of a sheaf on $\tilde{X}$,
but it is important to notice that the product \eqref{orb-prod} is not local. 
However, if we consider $\psi_*\Omega^\bullet_{\tilde{X}}$, the
push-forward to $X$, we have 
$\Gamma(X,\psi_*\Omega^\bullet_{\tilde{X}})=\Omega^\bullet(\tilde{X})$ and now 
the product is local, i.e.,
$(\psi_*\Omega^\bullet_{\tilde{X}},d,\bullet)$ does define a sheaf of DGA's on 
$X$.

The same can be done for the FWX-resolution on $\tilde{X}$. This time we 
introduce the product
\begin{equation}
\label{orb-prod-fwx}
  (\alpha\otimes\psi)\bullet(\beta\otimes\phi):=\int_{m_\ell}
  \left(pr^*_1\alpha\wedge \pr_2^*\beta\right)\otimes\left( 
  \pr^*_1\psi\cup_\text{\tiny\rm tw} \pr_2^*\phi\right).
\end{equation}
Again, because of the integration over the fiber, this defines a local product 
on $\psi^*\calC^{\bullet,\bullet}_{\tilde{X}}$, so that the total sheaf complex 
is a sheaf of DGA's.

\begin{lemma}
\label{qi1}
The embedding of Proposition \ref{embd-n} defines a quasi-isomorphism
\[
  I_\Psi:\psi_*\Omega^{\bullet-\ell}_{\tilde{X}}\otimes\C[[\hbar]]
  \rightarrow\psi_*\calC^{\bullet,\bullet}_{\tilde{X}}
\]
compatible with products up to homotopy.
\end{lemma}
\begin{proof}
By Corollary \ref{q-i-crs}, if we choose
$\Psi\in\Gamma\left(\mathsf{B}_0,\calC^\ell\right)$ to
be $\grp$-invariant, it descends to a morphism
\[
 I_\Psi:\psi_*\Omega^{\bullet-\ell}_{\tilde{X}}\otimes\C[[\hbar]]\rightarrow
 \psi_*\calC^{\bullet,\bullet}_{\tilde{X}}
\]
By assumption, the groupoid $\grp$ is proper, so we have
\[
  \mathbb{H}\left(\tilde{X},(\Omega^\bullet_{\tilde{X}},d)\right)\cong 
  \mathbb{H}\left(\mathsf{B}_0,(\Omega^\bullet_{\mathsf{B}_0},d)\right)^\grp,
\]
and similarly for $\calC^{\bullet,\bullet}_{\tilde{X}}$. Therefore
the morphism $I_\Psi$ is a quasi-isomorphism because it is a
quasi-isomorphism on $\mathsf{B}_0$. The fact that it preserves
products up to a coboundary, is a simple calculation as in the
proof of Corollary \ref{cpp-lc}.
\end{proof}
\begin{proposition}
\label{qi2}
There is a quasi-isomorphism 
\[
  \calH^\bullet_{\grp,\text{\tiny\rm loc},\hbar}\rightarrow
  \psi_*\calC^{\bullet,\bullet}_{\tilde{X}}
\] 
which maps the cup-product to the product \eqref{orb-prod-fwx}
\end{proposition}
\begin{proof}
For any $x\in X$, choose a local slice to obtain a morphism
\[
 \calH^\bullet_{\grp,\text{\tiny\rm loc},\hbar}(U_x)\rightarrow 
 C^\bullet\left(\calA^\hbar(M_x)\rtimes\grp_x,\calA^\hbar(M_x)\rtimes\grp_x
 \right),
\]
as in Theorem IIIb. By this very same Theorem IIIb, one knows that
on $E_0$ of the spectral sequences associated to the
$\hbar$-filtration, the above chain morphism is a
quasi-isomorphism, and therefore it is a quasi-isomorphism on the
original complexes. We now compose with the morphism of
Proposition \ref{gl-q} to get a map
\[
  \calH^\bullet_{\grp,\text{\tiny\rm loc},\hbar}\rightarrow
  \psi_*\calC^{\bullet,\bullet}_{\tilde{M_x\slash\grp_x}}(U_x)\cong
  \psi_*\calC^{\bullet,\bullet}_{\tilde{X}}(U_x).
\]
Because the sheaves are fine, this in fact defines a global quasi-isomorphism 
over the orbifold $X$.
\end{proof}

Finally, combining Lemma \ref{qi1} with Proposition \ref{qi2}, we have arrived 
at the main conclusion:
\begin{theorem}
\label{hc-product-gc} 
Let $\grp$ be a proper \'etale groupoid with an invariant symplectic structure, 
modeling a symplectic orbifold $X$. For any invariant deformation quantization 
$\calA^\hbar$ of $\grp$, we have a natural isomorphism 
\[
  H^\bullet\left(\calA^{((\hbar))}\rtimes \grp, \calA^{((\hbar))}\rtimes 
  \grp\right)\cong H^{\bullet-\ell}\left(\tilde{X},\C((t))\right).
\] 
With this isomorphism, the cup product is given by \eqref{orb-prod}.
\end{theorem}
\begin{remark}
We explain the product (\ref{orb-prod}) using the orbifold language.  Let $X$ be 
represented by the groupoid $\grp$ such that $X=\grp_0/\grp$, and $\tilde{X}$ be 
the corresponding inertia orbifold represented by $\mathsf{B}_0/\grp$.
Locally, an open chart of $X$  is like $U/\Gamma$ with $\Gamma$ a finite group
acting linearly on an open subset $U$ of $\mathbb{R}^n$.
Accordingly $\tilde{X}$ is locally represented by 
$\big(\coprod_{\gamma \in \Gamma}U^\gamma\big)/\Gamma \cong 
\coprod_{(\gamma)\subset \Gamma} U^\gamma/Z(\gamma)$, where $(\gamma)$ is the 
conjugacy class of $\gamma$ in $\Gamma$. We usually use $X^\gamma$ to stand for 
the component of $\tilde{X}$ containing $U^\gamma/Z(\gamma)$. 
With these notations, 
the above $S^1_{\grp}$ and $S^1_{\grp, \ell}$ are locally represented as
\[
 S^1_{\grp}=\coprod_{\gamma_1, \gamma_2}U^{\gamma_1}\cap U^{\gamma_2},
 \quad S^1_{\grp,\ell}=\coprod_{\tiny \begin{array}{c}\gamma_1, \gamma_2,\\
 \ell(\gamma_1)+\ell(\gamma_2)=\ell(\gamma_1\gamma_2)\end{array}}U^{\gamma_1}\cap
 U^{\gamma_2}.
\]

As we are considering $\grp$-invariant differential forms on $B_0$,
their pull-backs through projections $\pr_1, \pr_2$ to $S^1_{\grp}$
and $S^1_{\grp, \ell}$  are invariant under the following $\grp$-action on 
$S^1_{\grp}$ and $S^1_{\grp, \ell}$, which is defined as
\[
  (g_1, g_2)g=(g^{-1}g_1g,g^{-1}g_2g), \quad (g_1, g_2)\in S^1_{\grp,
  \ell},\ s(g)=t(g_2)=t(g_2).
\]
Locally this action can be written as a $\Gamma$-action on
$\coprod_{\gamma_1, \gamma_2}U^{\gamma_1}\cap U^{\gamma_2}$,
\[
(x, \gamma_1, \gamma_2)\gamma=(\gamma^{-1}(x), \gamma^{-1} \gamma_1\gamma,
\gamma^{-1}\gamma_2\gamma), \quad \gamma_1(x)=\gamma_2(x)=x, \gamma\in
\Gamma.
\]
The corresponding quotient space $S^1_{\grp}/\grp$ is usually denoted by
\begin{equation}
\label{tri-cyclic}
  X_3=\{ (x, (g_1, g_2, g_3))\mid g_1, g_2\in \Stab (x), \,
  g_1g_2g_3=\id, \, x\in X\}.
\end{equation}
One can see that locally $X_3=X^{g_1}\cap X^{g_2}$.  The pullbacks of
$\grp$-invariant differential forms on $\mathsf{B}_0$  are differential
forms on $X_3$.
Therefore, the formula (\ref{orb-prod}) can be interpreted as follows. For
$\alpha_1, \alpha_2\in \Omega^\bullet(\tilde{X})$,
\[
\alpha_1\bullet \alpha_2|_{\gamma}=\sum_{\tiny \begin{array}{c}\gamma_1,
\gamma_2,\\
\ell(\gamma_1)+\ell(\gamma_2)=\ell(\gamma_1\gamma_2)\end{array}}
\iota^{*}_{\gamma_1}(\alpha_1|_{\gamma_1})\wedge
\iota^*_{\gamma_2}(\alpha_2|_{\gamma_2}),
\]
where $\iota^*_{\gamma_i}$ is the embedding of $X_3$ in $X^{\gamma_i}$,
$i=1,2$.
\end{remark}

\subsection{Frobenius algebras from Hochschild cohomology} 
The product structure of Theorem \ref{hc-product-gc} is part of a natural graded 
Frobenius algebra associated to $\calA^{((\hbar))}\rtimes\grp$. Recall that a Frobenius
algebra is a commutative unital algebra equipped with an invariant trace. 
The construction of this Frobenius algebra on the Hochschild cohomology uses 
one additional piece of data, namely the trace on the algebra 
$\calA^\hbar\rtimes\grp$ constructed in \cite{ppt}.

Let $A$ be a unital algebra over a field $\Bbbk$ equipped with a trace 
$\tr:A\rightarrow \Bbbk$. As we have seen, the cup-product \eqref{Eq:DefCup} 
gives the Hochschild cohomology $H^\bullet(A,A)$ the structure of a graded 
algebra. The Hochschild homology $HH_\bullet(A)$ is a natural module over this 
algebra if we let a cochain $\psi\in C^k(A,A)$ act as 
$\iota_\psi:C_p(A)\rightarrow C_{p-k}(A)$ given by
\[
  \iota_\psi(a_0\otimes\ldots\otimes a_p)=(-1)^{\text{\tiny \rm deg}(\psi)}
  a_0\psi(a_1,\ldots,a_k)\otimes a_{k+1}\otimes\ldots\otimes a_p.
\]
With this module structure, the trace induces a pairing
\[
  \langle~,~\rangle:H^\bullet(A,A)\times HH_\bullet(A)\rightarrow \Bbbk
\]
which is given by
\[
 \langle \psi,a_0\otimes\ldots\otimes a_k\rangle=
 \tr\left(\iota_\psi(a_0\otimes\ldots\otimes a_k)\right).
\]
Let us assume, as in our case, that the Hochschild cohomology and homology are 
finite dimensional and that this pairing is perfect.
\begin{proposition}
Under these assumptions, the ring structure on $H^\bullet(A,A)$ is part of a 
natural  graded Frobenius algebra structure.
\end{proposition}
\begin{proof}
The pairing gives us a canonical isomorphism 
$H^\bullet(A,A)\cong H_\bullet(A)^*$. The trace defines a canonical element 
$[\tr]\in H_0(A)^*$ which defines the unit. 
Furthermore, the unit $1\in H_0(A)\cong H^0(A,A)^*$ defines an invariant trace.
\end{proof}
In the case at hand, the deformation of the convolution algebra on a symplectic 
orbifold, the Hochschild homology was computed in \cite{NeuPflPosTanHFDPELG} to 
be
\[
 H_\bullet\left(\calA^{((\hbar))}\rtimes\grp\right)\cong 
 H^{2n-\bullet}_\text{\tiny\rm cpt}\left(\tilde{X},\C((\hbar))\right).
\]
With the trace of \cite{ppt}, one checks that the pairing between Hochschild 
homology and cohomology above is nothing but Poincar\'e duality on $\tilde{X}$.


%
%
\section{Chen-Ruan orbifold cohomology}
\label{Sec:equiderham}
In this section, we study $S^1$-equivariant Chen-Ruan orbifold
cohomologies on an almost complex orbifold. In a special case, we
apply the idea from \cite{chen-hu:de-rham} to introduce a de Rham
model (topological Hochschild cohomology) to compute this
equivariant cohomology. In the last subsection, we compare this de
Rham model with the previous computation of Hochschild cohomology
of the quantized groupoid algebra. The Hochschild cohomology of
the quantized groupoid algebra is identified as a graded algebra
of the de Rham model with respect to some filtration.
\subsection{$S^1$-Equivariant Chen-Ruan orbifold
cohomology}\label{sec:equi-chen-ruan}In this subsection, we
briefly introduce the idea of $S^1$-equivariant Chen-Ruan orbifold
cohomology.
Let $X$ be an orbifold with an $S^1$ action. This means that there
is a morphism $f:S^1\times X\to X$ of orbifolds. One can think of
a morphism between two orbifolds as a collection of morphisms
between charts and group homomorphisms between local groups such
that the morphisms between charts are equivariant with respect to
local group actions, and are compatible with overlaps of charts.
(See \cite{cr-orbifoldCoh} and \cite{alr} for more details.) The
action is assumed to be associative, which is a somewhat delicate
property since the category of orbifolds is not a category but a
{\em 2-category}. This means that the standard associativity
diagram for a group action on an orbifold is only required to be
commutative up to 2-morphisms. Generalities on group actions on
categories can be found in \cite{romagny}. The $S^1$-equivariant
cohomology is defined via the standard Borel construction:
$$H^\bullet_{S^1}(X):=H^\bullet(X\times_{S^1}ES^1).$$
We identify $H^\bullet_{S^1}(\text{pt})=\complex[t]$.
With this, $H^\bullet_{S^1}(X)$ is considered as a $\complex[t]$-module. As
usual, we consider the fraction field $\complex((t))$ and put
\[
H^\bullet_{S^1}(X)((t)):=H^\bullet_{S^1}(X)\otimes_{\complex[t]}\complex((t)).
\]
In the following, we shall use the Cartan model for equivariant cohomology to
represent cohomology classes by equivariant differential forms.

As before, $\tilde{X}$ is the inertia orbifold, and $p: \tilde{X}\to
X$ the natural projection. It is easy to check that the
$S^1$-action lifts to $\tilde{X}$. Indeed, for any $s\in S^1$ the
action morphism $f_s: s\times X\to X$ defines for each $x\in X$ a
group homomorphism $\rho_s: \Stab(x)\to \Stab(f_s(x))$. This induces
the $S^1$-action on $\tilde{X}$, whose points are pairs $(x, (\gamma)),
x\in X, \gamma\in \Stab(x)$. More precisely, the action is given by
$$S^1\times \tilde{X}\to \tilde{X},\quad (s, (x,(\gamma)))\mapsto (f_s(x), (\rho_s(\gamma))).$$

As a $\complex[t]$-module, the $S^1$-Chen-Ruan orbifold cohomology
can be defined exactly in the same fashion as its non-equivariant
version \cite{cr-orbifoldCoh}, that is
$$H^\bullet_{S^1}(\tilde{X}):=H^\bullet(\tilde{X}\times_{S^1}ES^1).$$
There is a natural involution $I:\tilde{X}\to \tilde{X}$ which
maps a point $(x,(\gamma))$ to $(x,(\gamma^{-1}))$. The orbifold Poincar\'e
pairing $\left<\,\,\,,\,\,\,\right>$, which is defined by
$$\left<a,b\right>:=\int_{\tilde{X}} a\wedge I^*b,$$ naturally extends to a non-degenerate pairing on $H^\bullet_{S^1}(\tilde{X})$.

The additional structures one defines on Chen-Ruan orbifold
cohomology require a choice of an almost complex structure on the
tangent bundle $TX$, which we now make. We also assume that the
$S^1$-action on $TX$ is compatible with this almost complex
structure.

We will assume an $S^1$-action on the tangent bundle $TX$ which
commutes with the $S^1$-action on the base $X$. It should be noted
that we {\em do not} necessarily work with the canonical action on
$TX$ induced from that on $X$. This will be important in what
follows. Therefore the pull-back bundle $p^*TX$ admits an
$S^1$-action covering that on $\tilde{X}$. Let $X^\gamma$ be a
component of $\tilde{X}$. The bundle $p^*TX|_{X^\gamma}$ splits into a
direct sum of $\gamma$-eigenbundles. This allows one to define the age
function, denoted by $\iota(\gamma)$ (c.f. \cite{CheRuaOGWT}). This is
a locally constant function on $\tilde{X}$. We consider the
shifted $S^1$-equivariant cohomology of $\tilde{X}$,
\[
  H^\bullet_{S^1}(\tilde{X})((t))[-2\iota(\gamma)].
\] 
Here $t$ is assigned degree $2$.

The $S^1$-action on $p^*TX|_{X^\gamma}$ restricts to an $S^1$-action on each 
eigenbundle. Now consider the tri-cyclic sector \eqref{tri-cyclic}, i.e.,
the quotient $\mathsf{S}_{\mathsf{G}}\slash\mathsf{G}$.
There are three evaluation maps $e_i: X_3\to \tilde{X}$, 
$e_i((x,(\gamma_1,\gamma_2,\gamma_3))=(x,(\gamma_i))$. The $S^1$-action also lifts to 
tri-cyclic sector:
\[
  S^1\times X_3\to X_3,\quad (s,(x,(\gamma_1,\gamma_2,\gamma_3)))\mapsto 
  (f_s(x),(\rho_s(\gamma_1),\rho_s(\gamma_2),\rho_s(\gamma_3))).
\]
The evaluation maps are clearly $S^1$-equivariant. It follows from the 
above discussion that the obstruction bundle $\Theta$ over the tri-cyclic 
sector $X_3$ is an $S^1$-equivariant orbifold bundle on $\tilde{X}$. 
Therefore, we can define $S^1$-equivariant orbifold cup product $\star_t$ by
\[
  \left<\alpha_1\star_t \alpha_2, \alpha_3\right>=\int_{X_3}
  e_1^*(\alpha_1)\wedge e_2^*(\alpha_2)\wedge
  e_3^*(I^*(\alpha))\wedge \eu_{S^1}(\Theta),
\]
where $\eu_{S^1}(\Theta)$ is the equivariant Euler class of the
obstruction bundle. Many properties of the Chen-Ruan orbifold
cohomology algebra holds for the algebra
\[
(H^\bullet_{S^1}(\tilde{X})((t))[-2\iota(\gamma)],\star_t),
\]
with the same proofs. For example, the associativity of $\star_t$ is reduced to the rational equivalence between two points in the moduli space $\overline{M}_{0,4}$ of genus zero stable curves with four marked points. (Note that $\overline{M}_{0,4}\simeq \mathbb{CP}^1$.) See \cite{cr-orbifoldCoh} for more details.

\subsection{Equivariant de Rham model}
In this subsection, we define an equivariant de Rham model for a
special case of the above introduced $S^1$-equivariant Chen-Ruan
orbifold cohomology. We introduce our definition of topological Hochschild
cohomology algebra with the following steps.\\

\noindent{\bf Step I:} We start with an arbitrary almost complex
orbifold $X$ locally like $M/\Gamma$, and introduce a trivial
$S^1$ action on $X$ and therefore also on $\tilde{X}$ which is
locally like $(\coprod_{\gamma\in \Gamma} M^\gamma)/\Gamma$. Accordingly,
the $S^1$-equivariant cohomology of $\tilde{X}$ is equal to
$H^\bullet(\tilde{X})((t))$. \\

\noindent{\bf Step II:} We introduce a ``nontrivial'' $S^1$-action
on the tangent bundle $TX$ of $X$ which commutes with the trivial
$S^1$-action on $X$. Since $TX$ is an almost complex bundle,
$S^1$ identified with $U(1)$ acts on $TX$ as the center of the
principal group $\GL(\text{dim}_\complex(TX),\complex)$.
Geometrically, this action is simply rotation by an angle. We
remark that since $S^1$ is identified as the center of the
principal group, the above $S^1$-action commutes with all the
orbifold structure. And we have made
$TX$ into an $S^1$-equivariant orbifold bundle on $X$, and the 
same is for $p^*TX$ on $\tilde{X}$.\\

\noindent{\bf Step III:} We consider the normal bundle $N^\gamma$ of
the embedding of $X^\gamma$ into $X$. Since the $S^1$-action on $TX$
commutes with the $\gamma$-action, $N^\gamma$ inherits an $S^1$-action, and
becomes an $S^1$-equivariant vector bundle on $X^\gamma$. We decompose
$N^\gamma$ into a direct sum of $S^1$-equivariant line bundles
$\oplus_i N^\gamma _i$, with respect to the eigenvalue of $\gamma$-action,
i.e. $\exp(2\pi i\theta_i)$ and $0\leq \theta_i< 1$. Let $t_i$ be
the equivariant Thom form for $N^\gamma _i$, and the equivariant Thom
class $T_\gamma$ of $N^\gamma$ be defined by
\[
T_\gamma:=\prod_i t_i.
\]
For the following, it is important to remark that $T_\gamma$ is invertible in $\Omega^\bullet_{S^1}(N_\gamma)$.

\begin{definition}\label{def:top-hoch}
Define the topological Hochschild cohomology $HT^\bullet(X)((t))$
of an orbifold $X$ to be
\[
\bigoplus H^{\bullet}(X^\gamma)((t))[-\ell(\gamma)],
\]
where $\ell(\gamma)$ is, as before,  the codimension of $X^\gamma\in X$.

On $HT^\bullet(X)((t))$, we define a cup product $\wedge_t$ as
follows. First of all, the cup product is $\complex((t))$ linear. For
$\alpha_i\in \Omega^{\bullet-\ell(\gamma_i)}(X^{\gamma_i})((t))$, $i=1,2$,
$\alpha_1\wedge \alpha_2$ is defined by the following  integral,
\[
\left<\alpha_1\wedge_t \alpha_2, \alpha_3\right>=\int_{X^{\gamma_1\gamma_2}} \frac{\iota^*(\alpha_1\wedge T_{\gamma_1}\wedge \alpha_2\wedge T_{\gamma_2})}{\iota^*(T_{\gamma_1\gamma_2})}\wedge I^*(\alpha_3),
\]
for any $\alpha_3\in \Omega^{\bullet-\ell(\gamma_1\gamma_2)}(X^{\gamma_1\gamma_2})((t))$,
\end{definition}
\begin{remark}
More explicitly, if $\iota_*$ is the pushforward of
$\Omega^*(\tilde{X})$ into $\Omega^*(X)$ we have that
$\alpha_1\wedge_t\alpha_2=\iota^*(\iota_*(\alpha_1)\wedge
\iota_*(\alpha_2))$. A more global way to write the product, in
the style of Section \ref{defhoch}, is as follows:
\[
 \alpha_1\wedge_t\alpha_2=\int_{m}
 \frac{\pr_1^*(\alpha_1\wedge T)\wedge \pr_2^*(\alpha_2\wedge T)}{m^*T},
\]
where, as before $m:\mathsf{S}\rightarrow\mathsf{B}_0$ is the multiplication 
and $\int_m$ means integration over the discrete fiber.
\end{remark}
We remark that the associativity of $\wedge_t$ is an easy
corollary of the associativity of the wedge product on
differential forms of $X$. The following are a few simple
observations of $HT^\bullet(X)((t))$, which we state without
proofs. (They are corollaries of Theorem
\ref{thm:de-rham}.)\begin{enumerate}
\item the product $\wedge_t$ is $\complex((t))$-linear;
\item $(HT^\bullet(X)((t)),\wedge_t)$ is a graded algebra.
\end{enumerate}

In summary, with the $S^1$-action introduced in Step I and II, we
can introduce two cohomology algebra structures on
$\Omega^\bullet(\tilde{X})((t))$.
\begin{enumerate}
\item $S^1$-equivariant Chen-Ruan orbifold cohomology algebra as in Section \ref{sec:equi-chen-ruan};
\item topological Hochschild cohomology algebra as in Definition \ref{def:top-hoch}.
\end{enumerate}

In the rest of this subsection, we relate the topological
Hochschild cohomology $(HT^\bullet(\tilde{X})((t))[-l], \wedge_t)$
with the $S^1$-equivariant Chen-Ruan orbifold cohomology
$(H_{CR}^\bullet(\tilde{X})((t))[-2\iota], \star_t)$. The key
ingredient connecting these two algebra structures is a certain
equivariant Euler class naturally associated to the orbifold.

We consider a stringy K-group class \cite{jkk:stringy-k} $\fraks^\gamma$ 
associated to the normal bundle $N^\gamma$ of an orbifold $X$, i.e.,
\begin{equation}\label{eq:fraks}
\fraks^\gamma:=\bigoplus_{i}\theta_i N^\gamma_i.
\end{equation}
The equivariant Euler class $t_\gamma$ of $\fraks^\gamma$ is defined to be
\[
  t_\gamma:=
  \iota^*(\prod_i t_i^{\theta_i})\in H^{2\iota(\gamma)}(X^\gamma)((t)),
\]
where $t_i$ is the $S^1$-equivariant Thom class of $N_i^\gamma$, $\iota^*$ is the pullback of the Thom form to $X^\gamma$ which is embedded as the zero section. We remark that
$t_i^{\theta_i}$ and $t_\gamma$ are well defined in the $S^1$-equivariant cohomology $H^\bullet(X^\gamma)((t))$ by using the Taylor expansion of $t_i^{\theta}$ for the $\theta$-power.

We define the following isomorphism $J_\gamma: H^\bullet(X^\gamma)((t))[-2\iota(\gamma)]\to H^{\bullet-l(\gamma)}(X^\gamma)((t))$ of vector spaces,
\[
J_\gamma(\alpha)=\alpha/t_{\gamma^{-1}},\ \ \ \text{for all\ } \alpha\in
H^{\bullet-2\iota(\gamma)}(X^\gamma)((t)).
\]

\begin{remark}
We observe that $t_\gamma$ is invertible in $H^\bullet(X^\gamma)((t))$,
because it has a nonzero constant term. Accordingly, $J_\gamma$ is a
linear isomorphism of the vector spaces. The collection of all
$J_\gamma$ defines an isomorphism $J=\oplus_\gamma J_\gamma$ on
$H^\bullet(\tilde{X})$.

The map $J$ preserves grading. The degree of $t_{\gamma^{-1}}$ is equal
to $\ell(\gamma)-2\iota(\gamma)$. If $\alpha$ is an element in
$H^\bullet(X^\gamma)((t))[-2\iota(\gamma)]=H^{\bullet-2\iota(\gamma)}(X^\gamma)((t))$,
$J(\alpha)$ is of degree
\[
\bullet-2\iota(\gamma)-(\ell(\gamma)-2\iota(\gamma))=\bullet-\ell(\gamma).
\]
\end{remark}

The following theorem is a generalization of the result in 
\cite{chen-hu:de-rham}.\\

\noindent{\bf Theorem VI.}\emph{\label{thm:de-rham}
The map $J$ is an isomorphism of $\C((t))$-algebras from 
the algebra $(H^\bullet_{CR}(X)((t)),\star_t)$ to 
$(HT^\bullet(X)((t)), \wedge_t)$.}

\begin{proof}
As we have remarked,  $J$ is an isomorphism of vector spaces
preserving the degrees. It is sufficient to show that $J$ is
compatible with the algebra structures.

For $\alpha^i\in H^\bullet(X^{\gamma_i})((t))[-2\iota(\gamma_i)]$, $i=1,2$, and $\alpha_3\in H^\bullet(X^{\gamma_1\gamma_2})((t))[-2\iota(\gamma_1\gamma_2)]$ we have
\[
\begin{split}
&\left<J(\alpha_1)\wedge J(\alpha_2), \alpha_3\right> \\
=&\int_{X^{\gamma_1\gamma_2}}\frac{\iota^*(\frac{\alpha_1}{t_{{\gamma_1}^{-1}}}\wedge T_{\gamma_1}\wedge \frac{\alpha_2}{t_{{\gamma_2}^{-1}}}\wedge T_{\gamma_2})}{\iota^*(T_{\gamma_1\gamma_2})}\wedge I^*(\alpha_3)\\
=&\int_{X^{\gamma_1, \gamma_2}}
\frac{\iota^*(\frac{\alpha_1}{t_{{\gamma_1}^{-1}}}\wedge T_{\gamma_1}\wedge \frac{\alpha_2}{t_{{\gamma_2}^{-1}}}\wedge T_{\gamma_2})}{\iota^*(T_{\gamma_1\gamma_2})}\wedge I^*(\alpha_3)
\wedge \frac{\iota^*(T_{\gamma_1\gamma_2})}{\iota^*(T_{\gamma_1,\gamma_2})}\\
=&\int_{X^{\gamma_1, \gamma_2}}\iota^*(\alpha_1)\wedge \iota^*(\alpha_2)\wedge \iota^*(I^*(\alpha_3))\wedge \frac{\iota^*(T_{\gamma_1})\wedge\iota^*(T_{\gamma_2})}{\iota^*(t_{{\gamma_1}^{-1}})\wedge \iota^*(t_{{\gamma_2}^{-1}})\wedge\iota^*(T_{\gamma_1, \gamma_2})},
\end{split}
\]
where  $X^{\gamma_1,\gamma_2}:=X^{\gamma_1}\bigcap X^{\gamma_2}$, and $T_{\gamma_1, \gamma_2}$ is the equivariant Thom form for the normal bundle of $X^{\gamma_1, \gamma_2}$ in $X$, and $\iota^*$ is the pullback of the forms to $X^{\gamma_1, \gamma_2}$. And we can summarize the above computation in the following equation, for $\gamma_3=(\gamma_1\gamma_2)^{-1}$,
\begin{equation}\label{eq:step1}
\left<J(\alpha_1)\wedge J(\alpha_2), \alpha_3\right>
=\int_{X^{\gamma_1, \gamma_2}} \iota^*(\alpha_1)\wedge \iota^*(\alpha_2)\wedge \iota^*(I^*(\alpha_3))\wedge \calR_{\gamma_1, \gamma_2, \gamma_3},
\end{equation}
with
\[
\calR_{\gamma_1, \gamma_2, \gamma_3}=\frac{\iota^*(T_{\gamma_1})\wedge\iota^*(T_{\gamma_2})}{\iota^*(t_{{\gamma_1}^{-1}})\wedge \iota^*(t_{{\gamma_2}^{-1}})\wedge\iota^*(T_{\gamma_1, \gamma_2})}
\]

We now apply the result of \cite{jkk:stringy-k} to better understand the term $\calR_{\gamma_1, \gamma_2, \gamma_3}$. By \cite{jkk:stringy-k}[Thm. 1.2], for $\gamma_1\gamma_2\gamma_3=id$, when restricted to $X^{\gamma_1,
\gamma_2}:=X^{\gamma_1}\bigcap X^{\gamma_2}$, the obstruction bundle
$\Theta_{\gamma_1, \gamma_2}$ as a stringy $K$-group class has a natural splitting
\begin{equation}\label{eq:obstruction}
\Theta_{\gamma_1, \gamma_2}=T(X^{\gamma_1, \gamma_2})\ominus TX|_{X^{\gamma_1, \gamma_2}}\oplus
\fraks^{\gamma_1}|_{X^{\gamma_1, \gamma_2}}\oplus \fraks^{\gamma_2}|_{X^{\gamma_1,
\gamma_2}}\oplus \fraks^{\gamma_3}|_{X^{\gamma_1, \gamma_2}},
\end{equation}
where we remind that $\fraks^{\gamma_i}$ is  an element in the stringy
$K$-group \cite{jkk:stringy-k} as defined in Eq. (\ref{eq:fraks}).

We remark that the above isomorphism for $\Theta_{\gamma_1, \gamma_2}$ again holds as $S^1-$equivariant bundles because the $S^1$ actions on the respective bundles are defined by the almost complex structures and the equation (\ref{eq:obstruction}) preserves almost complex structures. Now taking the equivariant Euler classes of the bundles in Eq. (\ref{eq:obstruction}) on $X^{\gamma_1, \gamma_2}$, we have that on $X^{\gamma_1, \gamma_2}$
\[
\begin{split}
 \eu_{S^1}(\Theta_{\gamma_1, \gamma_2})&=
 \frac{\iota^*(t_{\gamma_1})\wedge \iota^*(t_{\gamma_2})\wedge 
 \iota^*(t_{\gamma_3})}{\iota^*(T_{\gamma_1,\gamma_2})}\\
 &=\frac{\iota^*(T_{\gamma_1})\wedge\iota^*(T_{\gamma_2})\wedge 
 \iota^*(T_{\gamma_3})}{\iota^*(t_{{\gamma_1}^{-1}})\wedge 
 \iota^*(t_{{\gamma_2}^{-1}})\wedge \iota^*(t_{{\gamma_3}^{-1}})\wedge
 \iota^*(T_{\gamma_1, \gamma_2})},
\end{split}
\]
where in the second equality, we have used the fact that on 
$X^{\gamma_1, \gamma_2}$,
\[
  \iota^*(t_{\gamma_i})=
  \frac{\iota^*(T_{\gamma_i})}{\iota^*(t_{\gamma_i^{-1}})} \quad 
  \text{for $i=1,2,3$}.
\]
We use the above expression for $\eu_{S^1}(\Theta_{\gamma_1, \gamma_2})$ 
to compute $\left<J(\alpha_1\star_t \alpha_2), \alpha_3\right>$.
\[
\begin{split}
\langle J(&\alpha_1\star_t\alpha_2), \alpha_3\rangle \\
=&\int_{X^{\gamma_1\gamma_2}}\frac{\alpha_1\star_t \alpha_2}{t_{(\gamma_1\gamma_2)^{-1}}}\wedge I^*(\alpha_3)\\
=&\int_{X^{\gamma_1,\gamma_2}}\alpha_1\wedge \alpha_2\wedge \frac{\iota^*(I^*(\alpha_3))}{\iota^*(t_{(\gamma_1\gamma_2)^{-1}})}\wedge \eu_{S^1}(\Theta_{\gamma_1, \gamma_2})\\
=&\int_{X^{\gamma_1,\gamma_2}}\!\!\!
 \alpha_1\wedge \alpha_2\wedge \frac{\iota^*(I^*(\alpha_3))}{\iota^*(t_{(\gamma_1\gamma_2)^{-1}})}\wedge \frac{\iota^*(T_{\gamma_1})\wedge\iota^*(T_{\gamma_2})\wedge \iota^*(T_{\gamma_3})}{\iota^*(t_{{\gamma_1}^{-1}})\wedge \iota^*(t_{{\gamma_2}^{-1}})\wedge \iota^*(t_{{\gamma_3}^{-1}})\wedge\iota^*(T_{\gamma_1, \gamma_2})}.
\end{split}
\]
Using the equality
\[
\iota^*(t_{(\gamma_1\gamma_2)^{-1}})=\iota^*(t_{\gamma_3})=\frac{\iota^*(T_{\gamma_3})}{\iota^*(t_{\gamma_3 ^{-1}})},
\]
we conclude that
\[
\begin{split}
\langle J(&\alpha_1\star_t\alpha_2), \alpha_3\rangle\\
=&\int_{X^{\gamma_1, \gamma_2}}\iota^*(\alpha_1)\wedge \iota^*(\alpha_2)\wedge \iota^*(I^*(\alpha_3))\wedge \frac{\iota^*(T_{\gamma_1})\wedge\iota^*(T_{\gamma_2})}{\iota^*(t_{{\gamma_1}^{-1}})\wedge \iota^*(t_{{\gamma_2}^{-1}})\wedge\iota^*(T_{\gamma_1, \gamma_2})}\\
=&\left<J(\alpha_1)\wedge_t J(\alpha_2), \alpha_3\right>.
\end{split}
\]
The last equation, combining with Poincar\'e duality, implies that
\[
  J^{-1}(J(\alpha_1)\wedge J(\alpha_2))=\alpha_1\star_t\alpha_2.
\]
This completes the proof.
\end{proof}

\begin{remark}
Note that when $t$ is equal to $0$, the map $J$ is not  invertible
generally. However, one can solve this problem by working in the formal
framework as in \cite{chen-hu:de-rham}. In this case our model extends 
Chen-Hu's model to an arbitrary almost complex orbifold.
\end{remark}

\subsection{Topological and algebraic Hochschild cohomology}
In the case of a symplectic orbifold $(X, \omega)$, we have two
cohomology algebra structures from different approaches. One is
the Hochschild cohomology algebra of the quantized groupoid
algebra computed in Theorem \ref{hc-product-gc}, the other is the
topological Hochschild cohomology $HT^\bullet(X)((t))$ defined in
Definition \ref{def:top-hoch} using essentially a unique (up to
homotopy) compatible almost complex structure to the symplectic
structure on $X$. We observe that the algebra structure on the
Hochschild cohomology of the quantized groupoid algebra is
completely topological, which does not depend on the symplectic
structures or the almost complex structures at all. On the other
hand, the topological Hochschild cohomology $HT^\bullet(X)((t))$
does depend on the choices of almost complex structures.
Therefore, it is natural to expect that these two algebras are not
isomorphic. In this subsection, we would like to study the
connections between these two algebra structures. We show in the
following that the graded algebra of the topological Hochschild
cohomology algebra is isomorphic to the Hochschild cohomology of
the corresponding quantized groupoid algebra.

We introduce a decreasing filtration on the topological Hochschild
cohomology $HT^\bullet(X)((t))$ as follows
\[
  \calF^\ast=\{ \alpha\in HT^\bullet(X^\gamma)((t))\mid
  \deg(\alpha)-\ell(\gamma)\geq \ast\}.
\]

\begin{lemma}\label{lem:filtration}
$(HT^\bullet(X)((t)), \wedge_t, \calF^\ast)$ is a filtered
algebra.
\end{lemma}
\begin{proof}
One needs to prove that $\calF^k\wedge_t \calF^l\subset
\calF^{k+l}$. To this end let $\alpha_1\in \calF^k$ and $\alpha_2\in \calF^l$
and consider $\alpha_1\wedge_t \alpha_2$. Without loss of generality, let us 
assume that $\alpha_1\in HT^{k_1}(X^{\gamma_1})((t))$ and 
$\alpha_2\in HT^{k_2}(X^{\gamma_2})((t))$ with 
$k_1\geq k+\ell(\gamma_1)$ and $k_2\geq l+\ell(\gamma_2)$. 
Since $(HT^\bullet(X)((t)), \wedge_t)$ is a graded
algebra respect to $\bullet$, we have that $\deg(\alpha_1\wedge
\alpha_2)=k_1+k_2\geq k+\ell(\gamma_1)+l+\ell(\gamma_2)$. Since
$\ell(\gamma_1)+\ell(\gamma_2)\geq \ell(\gamma_1\gamma_2)$, we have that
\[
 \deg(\alpha_1\wedge_t\alpha_2)-\ell(\gamma_1\gamma_2)\geq
 k+l+\ell(\gamma_1)+\ell(\gamma_2)-\ell(\gamma_1\gamma_2)\geq k+l.
\]
Therefore, $\alpha_1\wedge_t\alpha_2$ belongs to $\calF^{k+l}$.
\end{proof}

\begin{lemma}
\label{lem:codim} Let $\Gamma$ be a finite group acting a vector space
$V$. Then for every $\gamma_1, \gamma_2\in \Gamma$ one has
$\ell(\gamma_1)+\ell(\gamma_2)=\ell(\gamma_1\gamma_2)$ if and
only if $V^{\gamma_1}+V^{\gamma_2}=V$ and 
$V^{\gamma_1\gamma_2}=V^{\gamma_1}\cap V^{\gamma_2}$.
\end{lemma}
\begin{proof}
By linear algebra one knows that
\[
  \dim(V^{\gamma_1})+\dim(V^{\gamma_2})=
  \dim(V^{\gamma_1}+V^{\gamma_2})+\dim(V^{\gamma_1}\cap V^{\gamma_2}).
\]
Moreover, one has
\[
\begin{split}
 \ell(\gamma_1)+\ell(\gamma_2) & =2\dim(V)-(\dim(V^{\gamma_1})+\dim(V^{\gamma_2}))\\
&=2\dim(V)-\dim(V^{\gamma_1}+V^{\gamma_2})-\dim(V^{\gamma_1}\cap V^{\gamma_2})\\
&=\dim(V)-\dim(V^{\gamma_1}+V^{\gamma_2})+\dim(V)-\dim(V^{\gamma_1}\cap V^{\gamma_2}).
\end{split}
\]
Since $V^{\gamma_1}+V^{\gamma_2}\subset V$ and 
$V^{\gamma_1}\cap V^{\gamma_2}\subset V^{\gamma_1\gamma_2}$, we have
\[
\dim(V)-\dim(V^{\gamma_1}+V^{\gamma_2})\geq0,\ \ \ \
\dim(V)-\dim(V^{\gamma_1}\cap V^{\gamma_2})\geq \dim(V)-\dim(V^{\gamma_1\gamma_2}).
\]
Therefore
\[
\ell(\gamma_1)+\ell(\gamma_2)\geq \ell(\gamma_1\gamma_2),
\]
and equality holds, if and only if $\dim(V)=\dim(V^{\gamma_1}+V^{\gamma_2})$
and $\dim(V^{\gamma_1}\cap V^{\gamma_2})=\dim(V^{\gamma_1\gamma_2})$.
\end{proof}

\noindent{\bf Theorem VII.}\label{thm:hoch-de-rham}
\emph{The graded algebra $\gr(HT^\bullet(X)((t)))$ of
$(HT^\bullet(X)((t)), \wedge_t)$ with respect to the filtration
$\calF^\ast$ is isomorphic to the Hochschild cohomology algebra
$(H^\bullet(\calA^{((\hbar))}\rtimes G; \calA^{((\hbar))}\rtimes G), \cup)$
by identifying $t$ with $\hbar$.}

\begin{proof} Obviously, the two vector spaces over
$\C((t))$ are isomorphic. It is sufficient to prove that the
two product structures agree.

According to the proof of Lemma \ref{lem:filtration}, we have that
for $\alpha_1\in \calF^k$ and $\alpha_2\in \calF^l$, the graded
product $gr(\alpha_1\wedge_t\alpha_2)$ is not equal to zero only
when $\ell(\gamma_1)+\ell(\gamma_2)=\ell(\gamma_1\gamma_2)$.

In the case of $\ell(\gamma_1)+\ell(\gamma_2)=\ell(\gamma_1\gamma_2)$, 
by Lemma \ref{lem:codim}, we have that $V^{\gamma_1}+V^{\gamma_2}=V$ and
$V^{\gamma_1\gamma_2}=V^{\gamma_1}\cap V^{\gamma_2}$. This implies that $N^{\gamma_1}\oplus
N^{\gamma_2}=N^{\gamma_1\gamma_2}$ on $X^{\gamma_1\gamma_2}$. Therefore, 
the following identity of equivariant Thom classes holds true:
\[
\iota^*(T_{\gamma_1}\wedge T_{\gamma_2})=\iota^*(T_{\gamma_1\gamma_2}).
\]
Hence, by Definition \ref{def:top-hoch}, one obtains
\[
 \left<\alpha_1\wedge_t\alpha_2, \alpha \right>=
 \int_{X^{\gamma_1,\gamma_2}}\iota^*_{\gamma_1}(\alpha_1|_{\gamma_1})\wedge
 \iota_{\gamma_2}^*(\alpha_2|_{\gamma_2})\wedge 
 I^*(\alpha)|_{\gamma_1\gamma_2}, 
\]
 where the $\wedge$ on the right hand side is the wedge product on
 differential forms.
 One concludes that $\gr(\alpha_1\wedge_t\alpha_2)$ agrees with the
 cup product on the Hochschild cohomology algebra.
\end{proof}

From Theorem VII, we can view that the
topological Hochschild cohomology $(HT^\bullet(X)((t)), \wedge_t)$
as a deformation of the algebraic Hochschild cohomology
\[
(H^\bullet(\calA^{((\hbar))}\rtimes G; \calA^{((\hbar))}\rtimes G), \cup).
\]
It is very interesting to study this deformation using the
Hochschild cohomology method again, which will illustrate the role
of the almost complex structure chosen to define $\wedge_t$. We
leave this topic for future research.

\appendix
\section{Homological algebra of bornological algebras and modules}
\subsection{Bornologies on vector spaces}
In this appendix we recollect the basic definitions and constructions
in the theory of bornological vector spaces. For further details on
this see \cite{BouTVS} and \cite{HogBFA}.

Let $\Bbbk$ be the ground field $\R$ or $\C$, and $V$ be a
vector space over $\Bbbk$. A set $\frakB$  of subsets of $V$ is
called a
(\emph{convex linear}) \emph{bornology} on $V$ and $(V,\frakB)$ a
(\emph{convex linear}) \emph{bornological vector space}, if the
following axioms hold true:
\begin{enumerate}[(BOR1)]
\item
  Every subset of an element of $\frakB$ belongs to $\frakB$.
\item
  Every finite union of elements of $\frakB$ belongs to $\frakB$.
\item
  The set $\frakB$ is \emph{covering} for $V$ that means every
  element of $V$ is contained in some set belonging to $\frakB$.
\item
  For every $B \in \frakB$, the \emph{absolutely convex hull}
  $$B^\Diamond := \{ \lambda_1 v_1 + \lambda_2 v_2 \mid
  v_1,v_2 \in V, \: \lambda_1,\lambda_2 \in \Bbbk, \:
  |\lambda_1|+|\lambda_2|\leq 1 \}$$  is again  $\frakB$.
\end{enumerate}
The elements of a bornology $\frakB$ are called its
\emph{bounded sets} or sometimes its \emph{small sets}.

Given an absolutely convex set $S \subset V$, we denote its
linear span by $V_S$ and by $\|\cdot \|_B$ the seminorm on
$V_S$ with unit ball
$\overline{S} := \bigcap_{\lambda > 1} \lambda S$. If $\|\cdot\|_S$
is a norm on $V_S$, then $S$ is said to be \emph{norming}, and
\emph{completant}, if  $(V_S, \|\cdot\|_S)$ is even a Banach space.
A bornological vector space $(V,\frakB)$ is called \emph{separated}
(resp.~\emph{complete}), if  every bounded absolutely convex set
$B \subset V$ is norming resp.~completant.

\begin{proposition}
  Let $V$ be a bornological vector space. Then there exists a
  complete bornological vector space $\hat{V}$ together with a bounded
  linear map $\iota : V \rightarrow \hat{V}$ such that the following
  universal property is fulfilled:
  \begin{itemize}
  \item
    For every complete bornological vector space $W$ and
    every bounded linear map $f: V \rightarrow W$ there exists
    a unique bounded linear map $\hat{f} : \hat{V} \rightarrow W$
    such that the diagram
    \begin{equation}
    \label{UnivBorComp}
    \xymatrix{
      V\ar[rr]^{f}\ar[d]^{\iota} && W \\
      \hat{V} \ar[urr]_{\hat{f}}
    }
    \end{equation}
    commutes.
  \end{itemize}
\end{proposition}
\begin{proof}
  For the proof of this see \cite{MeyACC}.
\end{proof}

For $(V,\frakB)$ and $(W,\mathscr{D})$ two bornological vector spaces,
a linear map $f:V\rightarrow W$ is called {\it bounded}, if
for every $S\in \frakB$ the image $f(S)$ is in $\mathscr{D}$.
The space of bounded linear maps $V\rightarrow W$ will be denoted by
$\Hom (V,W)$. It carries itself a canonical bornology, namely the
bornology of \emph{equibounded} sets of linear maps, i.e.~of subsets
$E \subset \Hom (V,W)$ such that for each $S\in \frakB$ the set
$E(S)$ is bounded in $(W,\frakD )$.  Obviously, the bornological
vector spaces together with the bounded linear maps then form a
category. Since the direct sum $V\oplus W$ of two bornological
vector spaces obviously inherits a canonical bornological structure
from its components, the category of bornological vector spaces is
even an additive category.
Moreover, it carries the structure of a tensor category, since the algebraic
tensor product $V\otimes W$ of two bornological vector spaces $(V,\frakB)$
and $(W,\frakD)$ carries a natural bornology which is generated by the sets
$S \otimes T$, where $S \in \frakB$, $T \in \frakD$.

In case $V$ and $W$ are both complete bornological vector spaces, the
direct sum
$V\oplus W$ is obviously a complete bornological vector space as well.
For the tensor product $V\otimes W$, though, with its canonical
bornological structure, completeness need not necessarily hold. Therefore,
one introduces the completed tensor product
$V \hatotimes W := (V \otimes W)\hat{\hspace{0.3em}}$ for any pair of
bornological vector spaces $V,W$.
Note that the category of complete bornological vector spaces
with $\oplus$ and $\hatotimes$ as direct sum resp.~tensor functor also
satisfies the axioms of an additive tensor category. We denote the category
of complete bornological vector spaces and bounded linear maps by
$\mathsf{Bor}$.

\begin{example}
\label{Ex:BorEx}
 Let $V$ be a locally convex topological vector space. Then
\begin{displaymath}
\begin{split}
   \mathfrak{Bnd} (V) & \, :=
   \{ S \subset V\mid p (S) < \infty
   \text{ for every seminorm $p$ on $V$}\}
   \quad  \text{and} \\
   \mathfrak{Cpt} (V)& \, :=
   \{ S \subset V\mid S \text{ is precompact in $V$} \}
\end{split}
\end{displaymath}
 are two, in general different, bornologies on $V$, which one calls,
 respectively, the {\it von Neumann} and the {\it precompact} bornology.
\end{example}

\subsection{Bornological algebras and modules}
By a \emph{bornological algebra} one understands a $\Bbbk$-algebra $A$
together with a complete convex bornology $\frakB$ such that the
product map $m: A \otimes A \rightarrow A$ is bounded.
By the universal property of the completed bornological tensor product
one knows that for such an $A$ the multiplication $m$ lifts uniquely to a
bounded map $A \hatotimes A \rightarrow A$.

For any (real or complex) algebra $A$ we denote by $A^+$ the
unital algebra $A\oplus \Bbbk$, and by $\unA$ the smallest unital
algebra containing $A$, which means that $\unA$ coincides with
$A$, if $A$ is unital, and with $A^+$ otherwise. Obviously, $A^+$
and $\unA$ are again bornological algebras, if that is the case
already for $A$. For every bornological algebra $A$ we denote by
$\envA$ its {\it enveloping algebra} which is defined as the
bornological tensor product algebra $\unA \hatotimes (\unA)^\op $.

By a ({\it left}) $A$-{\it module} over a bornological algebra $A$ one
understands a complete bornological vector space $M$ together with a bounded
linear map $\unA \hatotimes M \rightarrow M$ such that the following axioms
are satisfied:
\begin{enumerate}[(MOD1)]
\item
  One has $(a_1 \cdot a_2 ) \cdot m = a_1 \cdot (a_2\cdot m)$ for all
  $a_1,a_2 \in A$ and $m\in M$.
\item
  The relation $1 \cdot m = m$ holds for all $m\in M$.
\end{enumerate}
\begin{example}
 For every complete bornological vector space $V$ the tensor product
 $\unA \hatotimes V$ carries in a natural way the structure of a
 left $A$-module. Modules of this form are called \emph{free}
 left $A$-modules; likewise one defines free right $A$-modules.
\end{example}

Given left $A$-modules $M$ and $N$ we write $\Hom_A (M,N)$ for the space of
bounded $A$-module homomorphisms with the equibounded bornology. Obviously,
the left $A$-modules together with these morphisms form a category, which
we will denote by $\Mod (A)$. Note that every morphism $f:M\rightarrow N$
in $\Mod (A)$ has a kernel and a cokernel. The kernel simply coincides with
the vector space kernel equipped with the subspace bornology, where the
cokernel is the quotient $N / f(M)^{\hat{\hspace{0.3em}}}$ together with
the quotient bornology. Similarly, one defines {\it right $A$-modules} over
a bornological algebra $A$ and writes $\Mod (A^\op)$
(resp.~$\Hom_{A^\op} (M,N)$) for the category of right $A$-modules
(resp.~the set of right $A$-module morphisms from $M$ to $N$). Finally,
an object in the category $\Mod (\envA)$ will be called an
$A$-{\it bimodule}.

For any right $A$-module $M$  and any left $A$-module $N$ we denote by
$M \hatotimes_A N$ the \emph{$A$-balanced tensor product} that means
the cokernel of the bounded linear map
\begin{displaymath}
  M \hatotimes A \hatotimes N \rightarrow M \hatotimes N, \quad
  m \otimes a \otimes n \mapsto m \cdot a \otimes n - m \otimes a \cdot n.
\end{displaymath}

A bornological algebra $A$ is said to have an \emph{approximate identity},
if for every bounded subset $S \subset A$ there is a bounded sequence
$(u_{S,k})_{k\in \N}$ and an absolutely convex bounded $T_S \subset A$
such that the following properties hold true:
\begin{enumerate}[({AID}1)]
\item
  For every $a \in A_S$ one has $u_{S,k} \cdot a \in A_{T_S}$ and
  $a \cdot u_{S,k} \in A_{T_S}$.
\item
  For all $a\in S$, the sequences $u_{S,k} \cdot a$ and  $ a \cdot u_{S,k}$
  converge to $a$ in the Banach space $A_{T_S}$, and the convergence is
  uniform in $a$.
\item
  For bounded subsets $S_1,S_2 \subset A$ such that $S_1\subset S_2$ one has
  \[
    \mbox{ } \hspace{10mm}
    \| u_{S_2 , k} \cdot a - a \|_{T_{S_2}}  \leq
    \| u_{S_1 , k} \cdot a - a \|_{T_{S_2}} \quad \text{for all
    $a \in A_{T_{S_1}}$ and $k\in \N$.}
  \]
\end{enumerate}
In other words, an approximate identity
$\big( u_{S,k} \big)_{S\in \calB,k\in \N}$
is essentially a net in $A$ such that each of the nets
$\big( u_{S,k} \, a \big)$ and $\big( a \, u_{S,k} \big)$
converges to $a$.

A bornological algebra $A$ which possesses an approximate identity
and which, additionally, is projective both as a left and a right
$A$-module, is called \emph{quasi-unital}. Note that under the
assumption that $A$ has an approximate identity, projectivity of
$A$ is equivalent to the existence of a bounded left $A$-module
map $l: A \rightarrow \unA \hatotimes A$ and a bounded right
$A$-module map $r: A \rightarrow A \hatotimes \unA$ which are both
sections of the multiplication map (cf.~\cite{MeyEDCBM}).

Given a quasi-unital bornological algebra $A$, a  left $A$-module $M$
(resp.~a right $A$-module $N$) is called \emph{essential}, if the
canonical map
$A \hatotimes_A M \rightarrow M$ (resp.~$N \hatotimes_A A \rightarrow N$)
is an isomorphism. If the left $A$-module $M$ (resp.~the right $A$-module
$N$) has the property that the canonical map
$M \rightarrow \Hom_A(A,M)$ (resp.~$N \rightarrow \Hom_{A^\op}(A,N)$)
is an isomorphism, one calls $M$ (resp.~$N$) a {\it rough} module.
The category of essential left $A$-modules
(resp.~right $A$-modules) will be denoted by $\mathsf{Mod_e} (A)$
(resp.~by $\mathsf{Mod_e} (A^\op)$). Since $A$ is assumed to be
quasi-unital, one concludes that for every $A$-module $M$, the tensor
product $A \hatotimes_A M$ is an essential module.
\subsection{Resolutions and homology}
In this article we consider homology theories in the additive but in
general not abelian category of modules over a bornological algebra $A$.
This implies that we have to use methods from relative homological algebra.
Essentially this means that only so-called allowable chain complexes
and allowable projective resolutions are used to determine homologies and
cohomologies. To define the notion of allowability precisely
recall that a bounded epimorphism of left $A$-modules
$f: M \longrightarrow N$ or in other words a
short exact sequence of left $A$-modules and bounded maps
\begin{displaymath}
  0\longrightarrow K \longrightarrow M \stackrel{f}{\longrightarrow}
  N \longrightarrow 0
\end{displaymath}
is called \emph{linearly split}, if there exists a bounded linear
map $N \rightarrow M$ which is a section of $f$. A left $A$-module
$P$ is now called \emph{projective}, if the functor $\Hom_A (P,
-)$ is exact on linearly split short exact sequences in
$\mathsf{Mod} (A)$. Moreover, a chain complex $(C_\bullet,
\partial)$  of $A$-modules and bounded maps $\partial_k :C_k
\rightarrow C_{k-1}$ is called \emph{allowable}, if for every $k$
the image of $\partial_k $ is in $\mathsf{Mod} (A)$, i.e.~is a
complete bornological subspace of $C_{k-1}$, and if the bounded
epimorphism $\overline{\partial_k} :C_k \rightarrow \im
\partial_k$ induced by $\partial_k$ is linearly split. Likewise
one defines allowable cochain complexes. For homology theories in
categories of modules of bornological algebras the following
result now is crucial.
\begin{proposition}
 Let $A$ be a quasi-unital bornological algebra $A$. Then every free
 left $A$-module is projective. Moreover, the  category
 $\mathsf{Mod} (A)$ has enough projectives, that means for every
 left $A$-module $M$ there exists a projective
 left $A$-module $P$ together with a split epimorphism
 of $A$-modules $P \rightarrow M$. Hereby, $P$ can be chosen to
 be free. Finally, the functor
 \begin{displaymath}
   \mathsf{Mod} (A) \rightarrow \mathsf{Mod_e} (A), \quad M
   \mapsto A \hatotimes M
 \end{displaymath}
 preserves projective modules, and the category $\mathsf{Mod_e} (A)$
 of essential left $A$-modules has enough projectives as well.
\end{proposition}
\begin{proof}
  See  \cite[Sec.~4]{MeyEDCBM}).
\end{proof}
 The proposition implies that for every left $A$-module there exists
 an \emph{allowable projective resolution} of $M$, i.e.~an allowable
 acyclic complex $(P_\bullet,\partial)$ of projective left $A$-modules
 $P_k$, $k\geq 0$ together with a quasi-isomorphism
 $\varepsilon : P_\bullet \rightarrow M_\bullet$ in the category
 of left $A$-modules, where $M_\bullet$
 denotes the complex which is concentrated in degree $0$ and
 coincides there with the $A$-module $A$.
 These conditions are equivalent to the requirement
 that $\varepsilon $ is a split bounded $A$-linear
 surjection $\varepsilon :P_0 \rightarrow M$ which satisfies
 \begin{displaymath}
   \varepsilon \circ \partial_1 =0
 \end{displaymath}
 and that there exists an $A$-linear splitting $h:M \rightarrow P_0$
 and a family $(h_k)_{k\in \N}$ of bounded linear maps
 $h_k :P_k \rightarrow P_{k+1}$ such that
 \begin{displaymath}
   \partial_1 h_0 = \id_{P_0} - h \varepsilon \quad \text{and}  \quad
   \partial_{k+1} h_k - h_{k-1} \partial_k = \id_{P_k}\quad
   \text{for all $k \geq 1$}.
 \end{displaymath}
The proof of the following result is standard in (relative) homological
algebra.
\begin{theorem}[Comparison Theorem] {\rm (}cf.~\cite[Thm.~A.9]{MeyACC}{\rm })
  Assume that $M$ and $N$ are two left $A$-modules over a bornological algebra
  $A$. Let $P_\bullet \rightarrow M_\bullet$ and
  $Q_\bullet \rightarrow N_\bullet$ be allowable resolutions  of $M$
  resp.~$N$. If $P_\bullet$ is projective,
  then there exists for every morphism $f:M \rightarrow N$ of left $A$-modules
  a \emph{lifting} of $f$, i.e.~a chain map
  $F: P_\bullet  \rightarrow Q_\bullet$ in the category
  of left $A$-modules such that the diagram
  \begin{displaymath}
    \xymatrix{
      P_\bullet \ar[r]\ar[d]_{F}& M_\bullet \ar[d]^{f}\\
      Q_\bullet \ar[r]& N_\bullet
    }
  \end{displaymath}
  commutes.
  Any two such liftings of $f$ are homotopic. In particular, any two
  allowable projective resolutions of $M$ are homotopy equivalent.
\end{theorem}
The comparison theorem allows the construction of derived functors in the
category of $A$-modules. In particular, the functors $\Ext$ and $\Tor$
can now be defined for $A$-modules as usual.

\subsection{Hochschild homology and Bar resolution}
\label{Sec:HHBR}
Given a bornological algebra $A$ and an $A$-bimodule $M$, the
Hochschild homology $H_\bullet (A,M)$ and cohomology $H^\bullet (A,M)$
are defined as derived functors in the category
$\Mod \big( \envA\big)$ of $A$-bimodules as follows:
\begin{equation}
  H_\bullet (A,M) := \Tor_\bullet^\envA (A,M) , \quad
  H^\bullet (A,M) := \Ext^\bullet_\envA  (A,M) .
\end{equation}
A particularly useful resolution of the $A$-bimodule $A$ is
given by the \emph{Bar complex} $(\Barcpl_\bullet (A), b')$
together with the multiplication map inducing the quasi-isomorphism
$\Barcpl_\bullet (A) \rightarrow  A$.
Hereby,
\begin{displaymath}
  \Barcpl_k ( A ) = A \hatotimes A^{\hatotimes k}  \hatotimes A ,
\end{displaymath}
  and $b'$ is the standard boundary map on the Bar complex:
\begin{displaymath}
\begin{split}
  b_0' (a_0 \otimes a_1) & = 0, \\
  b_k' (a_0 \otimes a_1 \ldots a_k\otimes a_{k+1}) & =
 \sum_{i=0}^k \, (-1)^i \, a_0 \otimes \ldots \otimes a_i \, a_{i+1}\otimes
 \ldots \otimes a_{k+1} .
\end{split}
\end{displaymath}
Obviously, if $A$ is quasi-unital, then the Bar complex of $A$ provides an
allowable projective resolution of $A$ in the category of $A$-bimodules.
In other words, this means that a quasi-unital bornological
algebra $A$ is \emph{H-unital} in the sense of Wodzicki
(see \cite{WodECHRAKT,LodCH}).
Thus, for quasi-unital $A$, the Hochschild homology and cohomology groups
$H_\bullet (A,M)$, $H^\bullet (A,M)$ are computed as the homology
resp.~cohomology of the Hochschild complexes
\begin{equation}
\label{EqHochCplx}
\begin{split}
 & (C_\bullet (A,M), b_*') \text{ with } C_\bullet (A,M):=
 \Barcpl_\bullet (A)\hatotimes_{\envA}M, \quad \text{and } \\
 & (C^\bullet (A,M),{b'}^*)  \text{ with } C^\bullet (A,M) :=
 \Hom_{\envA}(\Barcpl_\bullet (A),M) .
\end{split}
\end{equation}

For some applications, in particular to define the cup product on
the Hochschild cochain complex of a quasi-unital bornological
algebra which does not have a unit,  the left and right reduced
Bar complexes $\big( \lrBarcpl_\bullet (A),b'\big)$ and $\big(
\rrBarcpl_\bullet (A),b'\big)$ are quite useful. They carry the
same boundary as the Bar complex, and have components
\begin{displaymath}
  \lrBarcpl_k (A) := \unA \hatotimes A^{\hatotimes k} \hatotimes A \quad
  \text{and} \quad
  \rrBarcpl_k (A) := A \hatotimes A^{\hatotimes k} \hatotimes \unA .
\end{displaymath}
Obviously, under the assumption that $A$ is quasi-unital, the left and right
reduced Bar complexes are both allowable projective resolutions of $A$.
Moreover, the canonical embeddings
$\Barcpl_\bullet  (A) \hookrightarrow \lrBarcpl_\bullet  (A)$
and  $\Barcpl_\bullet  (A) \hookrightarrow \rrBarcpl_\bullet  (A)$
have the following quasi-inverses:
\begin{equation}
\label{Eq:QuInredBar}
\begin{split}
  r_k :&\:\lrBarcpl_k  (A)  \rightarrow \Barcpl_k  (A), \quad
  a_0 \otimes \ldots \otimes a_{k+1} \mapsto  a_0 r(a_1) \cdot \ldots \cdot
  r(a_{k+1}), \\
  l_k :&\: \rrBarcpl_k  (A)  \rightarrow \Barcpl_k  (A), \quad
  a_0 \otimes \ldots \otimes a_{k+1} \mapsto  l (a_0) \cdot \ldots \cdot
  l(a_k) a_{k+1} ,
\end{split}
\end{equation}
where $l: A \rightarrow A \hatotimes A$ resp.~$r: A \rightarrow A \hatotimes A$
is an $A$-left resp.~$A$-right linear section of the multiplication map on $A$.

Finally in this section we consider the \textit{reduced Hochschild chain}
and \textit{reduced Hochschild cochain complexes}.
These are defined by
\begin{equation}
\label{Eq:DefRedHoch}
\begin{split}
  C_k^\text{\tiny \rm red} (A,M) & \: :=
  \begin{cases}
    A \hatotimes_A M \hatotimes_A A , & \text{ if $k=0$}, \\
    M \hatotimes \big( \unA / \Bbbk \big)^{\hatotimes k}
    & \text{ if $k\geq 1$},
  \end{cases}
  \\
  C^k_\text{\tiny \rm red} (A, N) & \: :=
  \begin{cases}
    \Hom_\envA \big( A \hatotimes A , N \big), & \text{ if $k=0$}, \\
    \Hom_\Bbbk \big( \big( \unA / \Bbbk \big)^{\hatotimes k} , N \big) ,
    & \text{ if $k\geq 1$},
  \end{cases}
\end{split}
\end{equation}
and carry the same boundary resp.~coboundary maps as the unreduced complexes.
By construction the canonical maps
\[
  C_\bullet (A,M) \rightarrow
  C_\bullet^\text{\tiny \rm red} (A,M)
  \quad \text{and} \quad
  C^\bullet_\text{\tiny \rm red} (A,N) \rightarrow  C^\bullet (A,N)
\]
are then chain maps. If $M$ is an essential $A$-bimodule (resp.~$N$ a rough
$A$-bimodule), then the first (resp.~the second) of these chain maps is a
quasi-isomorphism. Since $A$ is assumed to be quasi-unital, the first chain map
is always a quasi-isomorphism for $M=A$. In many applications, and in particular
those appearing in this article, the second chain map is also a quasi-isomorphism
for $N=A$. In case $A$ is unital, both chain maps are always quasi-isomorphisms.
\subsection{The cup product on Hochschild cohomology}
\label{Sec:CPHC}
Under the general assumption from above that $A$ is a (possibly nonunital)
bornological algebra, we will now explain the construction of the cup
product
\begin{displaymath}
 \cup : H^\bullet (A,A) \times H^\bullet (A,A) \rightarrow H^\bullet (A,A).
\end{displaymath}
One way to define $\cup$ is via the Yoneda product on (bounded) extensions
\begin{displaymath}
  0 \rightarrow A \rightarrow E_1 \rightarrow \cdots \rightarrow E_k
  \rightarrow A \rightarrow 0
\end{displaymath}
and the interpretation of $\Ext^k_\envA (A,A)$ as the space of
equivalence classes of such extensions. Alternatively, and that is
the approach we will follow here, one can use the
quasi-isomorphisms from Eq.~(\ref{Eq:QuInredBar}) to directly
define a cup product $\cup : C^\bullet(A,A) \times C^\bullet (A,A)
\rightarrow C^\bullet (A,A)$ on the Hochschild cochain complex,
which on cohomology coincides with the Yoneda product. More
precisely, we define for $f \in C^k (A,A)$ and $g \in C^l (A,A)$
the product $f \cup g \in C^{k+l} (A,A)$ by
\begin{equation}
\label{Eq:DefCup}
  f \cup g (a_0 \otimes \cdots \otimes a_{k+l+1} ) :=
  f \big( l_k ( a_0 \otimes \cdots \otimes a_k \otimes 1 )\big)
  \, g \big( r_l (1 \otimes a_{k+1} \otimes \cdots \otimes
  a_{k+l+1}) \big),
\end{equation}
where $a_0 , \cdots, a_{k+l+1} \in A$.
It is straightforward to check that the thus defined map $\cup$ is a
chain map and associative up to homotopy.
The cup-product induced on the reduced Hochschild cochain complex
by the embedding
$C^\bullet_\text{\tiny \rm red} (A,A) \rightarrow  C^\bullet (A,A)$
is given by
\begin{equation}
\label{Eq:DefCupred}
  f \cup g (a_1 \otimes \cdots \otimes a_{k+l}) :=
  f (a_1 \otimes \cdots \otimes a_k) \, g (a_{k+1} \otimes \cdots \otimes
  a_{k+l})
\end{equation}
for $f \in C^k_\text{\tiny \rm red} (A,A)$, $g \in C^l_\text{\tiny \rm red} (A,A)$
and $a_1 , \cdots, a_{k+l} \in A$.
\subsection{Bornological structures on convolution algebras and their modules}
\label{Sec:bcalgebras}
Consider a proper \'etale Lie groupoid $\grp$ and let $\calA \rtimes \grp$
denote its convolution algebra (see Sec.~\ref{Sec:outline}).
A subset $S \subset \calA \rtimes \grp$ is said to be \emph{bounded},
if there is a compact subset in $K\subset \grp_1$ such that
$\supp a \subset K$ for every $a\in S$, and if for each differential operator
$D$ on $\grp_1$  one has
\begin{displaymath}
  \sup_{a \in S} \| Da \|_K < \infty .
\end{displaymath}
The bounded subset of $\calA \rtimes \grp$ form a bornology
which coincides both with the von Neumann and the precompact bornology defined
in Example \ref{Ex:BorEx}. In this article, we always assume that
$\calA \rtimes \grp$ carries this bornology. By an immediate argument one
checks that the convolution product is bounded and that the bornology
on $\calA \rtimes \grp$ is complete. Thus $\calA \rtimes \grp$ becomes a
bornological algebra. Let us check that it is quasi-unital.
To this end choose a sequence of smooth maps
$\varphi_k : \grp_0 \rightarrow [0,1]$ such that the support of each
$\varphi_k$
is compact and such that $\big( \varphi_k^2 \big)_{k\in \N}$ is a locally
finite partition of unity on $\grp_0$. Obviously, one can even achieve that
\begin{equation}
\label{Eq:SuppPartUnity}
  \overline{(\supp \varphi_k )^\circ} = \supp \varphi_k
\end{equation}
holds for every $k\in \N$; this is a property we will need later.
Now extend each $\varphi_k$ by zero to a
smooth function on $\grp_1$ and denote the resulting element of
$\calA \rtimes \grp$ again by  $\varphi_k$.
Then put $u_k := \sum_{l\leq k} \varphi_l * \varphi_k$
and check that
$\big( u_k \big)_{k\in \N}$ is an approximate identity. Moreover, the maps
\begin{displaymath}
\begin{split}
  l:&\; \calA \rtimes \grp \rightarrow \calA \rtimes \grp \, \hatotimes \,
  \calA \rtimes \grp, \: a \mapsto
  a * \varphi_k \otimes \varphi_k \quad \text{and} \\
  r:&\; \calA \rtimes \grp \rightarrow \calA \rtimes \grp \,\hatotimes \,
  \calA \rtimes \grp, \: a \mapsto \sum_{k\in \N}
  \varphi_k \otimes \varphi_k * a
\end{split}
\end{displaymath}
are both sections of the convolution product. This proves
\begin{proposition}
\label{Prop:ConBorAlg}
  The convolution algebra $\calA \rtimes \grp$ of a proper \'etale Lie groupoid
  $\grp$ together with the von Neumann bornology is a quasi-unital
  bornological algebra.
\end{proposition}
Next let us consider the case where $\grp_0$ carries a $\grp$-invariant
symplectic form $\omega$ and where a $\grp$-invariant local star product
$\star$ on $\grp_0$ has been chosen.
Under these assumptions consider the crossed product algebra
$\calA^\hbar \rtimes \grp$, where $\calA^\hbar $ denotes the sheaf
$\calC^\infty_{\grp_0}[[\hbar]]$ with product $\star$.
A subset $B \subset \calA^\hbar \rtimes \grp$ is said to be \emph{bounded},
if there is a compact subset in $K\subset \grp_1$ such that
$\supp a \subset K$ for every $a\in B$, and such that
for each differential operator on $\grp_1$ and $k\in \N$ one has
\begin{displaymath}
  \sup_{a \in B} \| Da \|_{k,K} < \infty .
\end{displaymath}
Hereby, $\| \cdot \|_{k,K}$ is the seminorm on $\calA^\hbar \rtimes \grp$
defined by
\begin{displaymath}
  \| a \|_{k,K} := \sup_{g\in K} |a_k (g) |, \quad a \in \calA^\hbar
  \rtimes \grp,
\end{displaymath}
where the $a_l \in \calC^\infty_\text{\tiny\rm cpt} ( \grp_1)$, $l\in \N$
are the unique coefficients in the formal power series expansion
$a = \sum_{l\in \N} a_l \, \hbar^l$.
The bounded subsets of $\calA^\hbar \rtimes \grp$ define a complete bornology
which we call the \emph{canonical bornology} on $\calA^\hbar \rtimes \grp$.
One immediately checks that the convolution product $\star_\text{\tiny\rm c}$
defined by Eq.~(\ref{Eq:DefConProd}) is bounded, hence
$\calA^\hbar \rtimes \grp$
is a bornological algebra. Obviously, the family $\big( u_k\big)_{k\in \N}$
from above forms an approximate unit also for $\calA^\hbar \rtimes \grp$.
By the assumption (\ref{Eq:SuppPartUnity}) it is clear that each of the
functions $\varphi_k$ has only zeros of infinite order.
Now check the following lemma by
using standard arguments from the theory of deformation quantization.
\begin{lemma}
  Let $\varphi : \grp_0 \rightarrow [0,1]$ be a smooth function which has
  only zeros of infinite order, and put $u = \varphi^2$.
  Then $u$ has a star product root, that means there exists
  an element $\Phi = \sum_{l\in \N} \Phi_l \hbar^l \in \calC^\infty[[\hbar]]$
  such that
  \begin{displaymath}
     \Phi \star \Phi = u, \quad \Phi_0 =\varphi, \quad \text{and} \quad
     \supp \Phi \subset   \supp \varphi .
  \end{displaymath}
\end{lemma}
Using this result choose
$\Phi_k \in \calA^\hbar \rtimes \grp$ with support in $\grp_0$ such that
$\Phi_k \star_\text{\tiny\rm c} \Phi_k = \varphi_k^2$ and
$\Phi_k -\varphi_k \in \hbar  \calA^\hbar \rtimes \grp$. The maps
\begin{displaymath}
\begin{split}
  l:&\; \calA^\hbar \rtimes \grp \rightarrow \calA^\hbar \rtimes \grp \,
  \hatotimes \, \calA^\hbar \rtimes \grp, \: a \mapsto \sum_{k\in \N}
  a \star_\text{\tiny\rm c} \Phi_k \otimes \Phi_k \quad \text{and} \\
  r:&\; \calA^\hbar \rtimes \grp \rightarrow \calA^\hbar \rtimes \grp \,
  \hatotimes \, \calA^\hbar \rtimes \grp, \: a \mapsto \sum_{k\in \N}
  \Phi_k \otimes \Phi_k \star_\text{\tiny\rm c} a
\end{split}
\end{displaymath}
then are both sections of the convolution product. Hence we obtain
\begin{proposition}
\label{Prop:DefConBorAlg}
  The crossed product algebra $\calA^\hbar \rtimes \grp$ associated to
  an invariant local deformation quantization on the space of objects
  of a proper \'etale Lie groupoid $\grp$ with an invariant symplectic form
  is a quasi-unital bornological algebra.
\end{proposition}
\subsection{Morita equivalence for bornological algebras}
\label{Sec:MEBA}
Assume that $A$ and $B$ are two bornological algebras. Recall that by an
\emph{$A$-$B$-bimodule} one understands an element of the category
$\Mod \big( \unA \hatotimes (B^\text{\tiny  \rm u})^\op \big)$.
Under the condition that the bornological algebras $A$ and $B$ are
both quasi-unital,
one calls $A$ and $B$ \emph{Morita equivalent}, if there exist bimodules
$P\in \Mod \big( \unA \hatotimes (B^\text{\tiny  \rm u})^\op \big)$
and $Q\in \Mod \big( B^\text{\tiny  \rm u} \hatotimes (\unA)^\op \big)$
such that the following axioms hold true:
\begin{enumerate}[(MOR1)]
\item
  $P$ is essential both as an $A$-left module and as a $B$-right module.
\item
   $Q$ is essential both as a $B$-left module and as an $A$-right module.
\item
   There exist bounded bimodule isomorphisms
   \[
     u: P\hatotimes_B Q \rightarrow A \quad \text{and} \quad
     v: Q \hatotimes_A P \rightarrow B .
   \]
\item
   $P$ is projective as a $B$-right module, and $Q$ is projective as
   an $A$-right module.
\end{enumerate}
We sometimes say in this situation that $(A,B,P,Q,u,v)$
is a \emph{Morita context}. The following result follows easily from the
definition of a Morita context.
\begin{proposition}
  Let $(A,B,P,Q,u,v)$ be a Morita context. Then the functors
  \begin{displaymath}
  \begin{split}
    \Mod (A)&\rightarrow \Mod (B), \quad M \mapsto Q \hatotimes_A M \\
    \Mod (B)&\rightarrow \Mod (A), \quad N \mapsto P \hatotimes_B N
  \end{split}
  \end{displaymath}
  are both exact and quasi-inverse to each other.
  In particular this means that $\Mod (A)$ and $\Mod (B)$ are
  equivalent categories.
\end{proposition}
\begin{example}
\label{Ex:OpWeEq}
  Let $\varphi : \hgrp \rightarrow \grp$ be a weak equivalence of proper
  \'etale Lie groupoids. By \cite[Cor.~3.2]{MrcFBAHSM} it follows that the
  convolution algebras
  $A := \calA \rtimes \grp$ and $B:=\calA \rtimes \hgrp$ are Morita
  equivalent. A Morita context is given by the bimodules
  $P =\calC^\infty_\text{\tiny \rm cpt} (\langle \varphi \rangle )$ and
  $Q= \calC^\infty_\text{\tiny \rm cpt} ({\langle \varphi \rangle}^{\!-})$,
  where $\langle \varphi \rangle := \grp_1 \times_{(s,\varphi)} \hgrp_0$
  and
  ${\langle \varphi \rangle}^{\!-} := \hgrp_0\times_{(\varphi,t)}\grp_1$.
  Let us provide the details for the case, where $\varphi$ is even
  an open embedding. Then,
  $\langle \varphi \rangle $ is the open subset
  $s^{-1} (\varphi (\hgrp_0)) \subset \grp_1$, and
  ${\langle \varphi \rangle}^{\!-} = t^{-1} ( \varphi ( \hgrp_0))
  \subset\grp_1$.
  Moreover, the $A$-$B$-bimodule structure on $P$ is given by
  \begin{displaymath}
     a * p * b (g) = \sum_{g_1 \, g_2  \, \varphi( h) = g \atop
     g_1 \in \grp_1 , g_2 \in \langle\varphi\rangle , h \in \hgrp_1 }
     a (g_1) \, p (g_2) \, b (h) ,
  \end{displaymath}
  where $a \in \calC^\infty_\text{\tiny \rm cpt} (\grp_1 ) $,
  $b \in \calC^\infty_\text{\tiny \rm cpt} (\hgrp_1 ) $
  and $p \in \calC^\infty_\text{\tiny \rm cpt}(s^{-1}(\varphi (\hgrp_1))$.
  The bimodule structure for $Q$ is defined analogously.
  \begin{theorem}
  \label{Thm:WeEqMor}
  Under the assumption that the weak equivalence
  $\varphi : \hgrp \hookrightarrow \grp$ is an open embedding,
  the following holds true:
  \begin{enumerate}[\rm (1)]
  \item
  \label{ItMor12}
    The bimodules
    $P = \calC^\infty_\text{\tiny \rm cpt} (s^{-1}(\varphi (\hgrp_0))$ and
    $Q= \calC^\infty_\text{\tiny \rm cpt} (t^{-1}(\varphi (\hgrp_0))$
    satisfy axioms (MOR1) and (MOR2).
  \item
  \label{ItProj}
    $P$ resp.~$Q$ is projective both as an $A$- as a $B$-module.
  \item
    The maps
  \begin{displaymath}
  \begin{split}
    u : \: & \calC^\infty_\text{\tiny \rm cpt}(s^{-1}(\varphi (\hgrp_0))
    \hatotimes_{\calC^\infty_\text{\tiny \rm cpt} (\hgrp_1)}
    \calC^\infty_\text{\tiny \rm cpt}(t^{-1}(\varphi (\hgrp_0))
    \rightarrow \calC^\infty_\text{\tiny \rm cpt} (\grp_1 ) , \:
    a \otimes \tilde a \mapsto a * \tilde a ,
    \\
    v : \: &\calC^\infty_\text{\tiny \rm cpt}(t^{-1}(\varphi (\hgrp_0))
    \hatotimes_{\calC^\infty_\text{\tiny \rm cpt} (\grp_1)}
    \calC^\infty_\text{\tiny \rm cpt}(s^{-1}(\varphi (\hgrp_0))
    \rightarrow \calC^\infty_\text{\tiny \rm cpt} (\hgrp_1 ) , \:
    \tilde a \otimes  a \mapsto \varphi^* (\tilde a * a ) .
  \end{split}
  \end{displaymath}
  are bounded bimodule isomorphisms.
  \end{enumerate}
  This means in particular that the tuple
  $(\calA\rtimes \grp, \calA\rtimes \hgrp, P,Q, u,v)$ is a Morita context
  between bornological algebras.
  \end{theorem}
  \begin{proof}
    Consider the family $(\varphi_k)_{k\in \N}$ of elements of
    $A = \calA \rtimes \grp$ from above. Then the map
    \begin {displaymath}
    \begin{split}
      l_P : & \: P \rightarrow A \hatotimes P, \: p \mapsto
      \sum_k p* \varphi_k \otimes  \varphi_k \quad \text{resp.~}\\
      r_Q : & \: Q \rightarrow Q \hatotimes A, \: q \mapsto
      \sum_k \varphi_k \otimes  \varphi_k * q
    \end{split}
    \end{displaymath}
    is a section of the left (resp.~right) $A$-action on $P$ (resp.~$Q$).
    Since $A$ is quasi-unital, this implies that $P$ (resp.~$Q$) is essential
    and projective as a left (resp.~right) $A$-module.
    Likewise, one proves the existence of a section $l_Q$ (resp.~$r_P$) of
    the left (resp.~right) $B$-action on $Q$ (resp.~$P$).
    Hence, $Q$ (resp.~$P$) is essential and projective as a left (resp.~right)
    $B$-module. Thus, we have proved (\ref{ItMor12}) and (\ref{ItProj}).

    Next we show that there exists a bounded section of the map
    \begin{displaymath}
      \hat{v} :\calC^\infty_\text{\tiny \rm cpt}(t^{-1}(\varphi (\hgrp_0))
      \hatotimes
      \calC^\infty_\text{\tiny \rm cpt}(s^{-1}(\varphi (\hgrp_0))
      \rightarrow \calC^\infty_\text{\tiny \rm cpt} (\hgrp_1 ) , \:
      \tilde a \otimes  a \mapsto \varphi^* (\tilde a * a ) .
    \end{displaymath}
    To this end choose a family $(\psi_k)_{k\in \N}$ in $B =\calA\rtimes\hgrp$
    which has support in $\hgrp_0$, and such that $(\psi_k^2)_{k\in \N}$
    is a locally finite partition of unity on $\hgrp_0$. For each
    element $b\in B$ define elements $\varphi_* b$ in $P$ and $Q$ by extension
    by zero. Then the map
    \begin{displaymath}
    \begin{split}
      \hat\nu : & \:\calC^\infty_\text{\tiny \rm cpt} (\hgrp_1 ) \rightarrow
      \calC^\infty_\text{\tiny \rm cpt}(t^{-1}(\varphi (\hgrp_0))
      \hatotimes
      \calC^\infty_\text{\tiny \rm cpt}(s^{-1}(\varphi (\hgrp_0)), \:
      b \mapsto \sum_{k\in \N} \varphi_*b * \psi_k \otimes
      \varphi_* \psi_k
    \end{split}
    \end{displaymath}
    is a bounded section of $\hat{v}$ and a morphism of left $B$-modules.
    Hence $\nu := \pi \hat\nu $, where
    $\pi :Q\hatotimes P\rightarrow Q\hatotimes_A P$
    denotes the canonical projection, is a bounded section of $v$.
    Note that by construction the image of $\nu$ lies in the algebraic
    tensor product $Q \otimes_A P$, and that the image of $\nu$ is  a
    complete bornological subspace of $Q \otimes_A P$
    which has to be isomorphic to $B$.
    Since by \cite[Thm.~4]{MrcFBAHSM} the restriction to the algebraic tensor
    product  $v_{|Q \otimes_AP} :Q \otimes_AP \rightarrow B$ is an (algebraic)
    isomorphism of $B$-bimodules, one then concludes that
    $Q \otimes_A P = Q \hatotimes_A P $ and that $v$ is a (bornological)
    isomorphism of $B$-bimodules.

    The proof that $u$ is a (bornological) isomorphism of $A$-bimodules is
    more complicated. We show this claim under the additional assumption that
    $\hgrp_0$ is connected. The general case is slightly more involved,
    but can be proved along the same lines. Denote by $\grp_\alpha$ the
    connected components of $\grp_0$, and by $\grp_{\alpha_0}$ the image
    $\varphi (\hgrp_0)$. Let
    $\grp_{\alpha,\beta}=s^{-1}(\grp_\alpha)\cap t^{-1}(\grp_\beta)$. Then
    $\grp_1$ is the disjoint
    union of the $\grp_{\alpha,\beta}$. Since $\varphi$ is a weak equivalence,
    one derives the following:
    \begin{enumerate}[(i)]
    \item
     The image $\varphi (\hgrp_1)$ coincides with $\grp_{\alpha_0,\alpha_0}$.
    \item
     The bibundle $\langle \varphi \rangle = s^{-1}(\varphi(\hgrp_0))$
     coincides with the disjoint union of the components
     $\grp_{\alpha_0, \alpha}$, and ${\langle \varphi \rangle}^{\! -} =
     t^{-1}(\varphi(\hgrp_0))$ with the disjoint
     union of the components $\grp_{\alpha, \alpha_0}$.
    \item
     There exist unique open embeddings
     $\sigma_{\alpha,\beta}:\grp_{\alpha,\beta}\hookrightarrow
     \grp_{\alpha,\alpha_0}$
     and
     $\tau_{\alpha,\beta}:\grp_{\alpha,\beta}\hookrightarrow
     \grp_{\alpha_0,\beta}$
     such that $ s \circ \sigma_{\alpha,\beta} = s_{|\grp_{\alpha,\beta}}$
     resp.~$ t \circ \tau_{\alpha,\beta} = t_{|\grp_{\alpha,\beta}}$.
    \end{enumerate}
    Next choose for every $\beta$ functions
    $\psi_{\beta,k}\in\calC^\infty_\text{\tiny \rm cpt}
    (\grp_{\beta,\beta}\cap\grp_0 )$
    such that each family $(\psi_{\beta,k}^2)_{k\in \N}$ is a locally finite
    partition of unity. Extend $\psi_{\beta,k}$ by $0$ to a smooth function
    with compact support in $\grp_{\beta,\beta}$. Define
    $\psi_{\beta,k}^{(1)} \in \calC^\infty_\text{\tiny \rm cpt} (\hgrp_1)$ and
    $\psi_{\beta,k}^{(2)} \in \calC^\infty_\text{\tiny \rm cpt} (t^{-1}
    \varphi (\hgrp_0))$ by
    \begin{displaymath}
      \psi_{\beta,k}^{(1)} =
      \begin{cases}
        \varphi^* \big( (\tau_{\alpha_0,\beta}\sigma_{\beta,\beta})_*
        (\psi_{\beta,k})\big)
        & \text{over $\hgrp_0$} \\
        0 & \text{else}.
      \end{cases}
    \end{displaymath}
    and
    \begin{displaymath}
      \psi_{\beta,k}^{(2)} =
      \begin{cases}
        (\tau_{\beta,\beta})_* (\psi_{\beta,k})
        & \text{over $\grp_{\alpha_0,\beta}$} \\
        0 & \text{else}.
      \end{cases}
    \end{displaymath}
    Now we have the means to construct a bounded section of the map
    \begin{displaymath}
      \hat{u} :\calC^\infty_\text{\tiny \rm cpt}(s^{-1}(\varphi (\hgrp_0))
      \hatotimes
      \calC^\infty_\text{\tiny \rm cpt}(t^{-1}(\varphi (\hgrp_0))
      \rightarrow \calC^\infty_\text{\tiny \rm cpt} (\grp_1 ) , \:
      a \otimes  \tilde a \mapsto a * \tilde a  .
    \end{displaymath}
    Put
    \begin{displaymath}
    \begin{split}
      \hat\mu : \: &  \calC^\infty_\text{\tiny \rm cpt} (\grp_1 ) \rightarrow
      \calC^\infty_\text{\tiny \rm cpt}(s^{-1}(\varphi (\hgrp_0))
      \hatotimes
      \calC^\infty_\text{\tiny \rm cpt}(t^{-1}(\varphi (\hgrp_0)), \\
      & \calC^\infty_\text{\tiny \rm cpt} (\grp_{\alpha,\beta}) \ni a \mapsto
      \sum_{k \in \N}
      (\sigma_{\alpha,\beta})_* (a) * \psi_{\beta,k}^{(1)}
      \otimes \psi_{\beta,k}^{(2)} .
    \end{split}
    \end{displaymath}
    and check that $\hat\mu$ is a bounded section of $\hat{u}$ indeed.
    One proceeds now exactly as for $v$ to show that
    $P\otimes_B Q= P\hatotimes_BQ$
    and that $u$ is a bornological isomorphism of $A$-bimodules.
  \end{proof}
\end{example}
\noindent
Hochschild (co)homology of bornological algebras and their bimodules is
invariant under a Morita context as the following result shows.
\begin{theorem}[cf.~{\cite[Thm.~1.2.7]{LodCH}}]
Assume that $A$ and $B$ are quasi-unital
bornological algebras and assume that there is a
Morita  context $(A, B, P,Q, u,v)$ with the additional property that
\begin{equation}
\label{Eq:Homup}
  q u( p\otimes q') = v(q \otimes p) q' \text{ and }
  p v( q\otimes p') = u(p \otimes q) p'
  \text{ for all $p,p' \in P$ and $q,q' \in Q$.}
\end{equation}
Let $M$ be an essential $A$-bimodule. Then  there are natural isomorphisms
\begin{equation}
\label{Eq:MorInvHochschild}
  H_\bullet (A,M) \cong H_\bullet (B, Q\hatotimes_A M\hatotimes_A P)
  \quad \text{and} \quad
  H^\bullet (A,M) \cong H^\bullet (B, Q\hatotimes_A M\hatotimes_A P)
\end{equation}
\end{theorem}
\begin{proof}
We only prove the homology case. The cohomology case is proven similarly.
To prove the claim first choose approximate identities
$(u_{S,k})_{S\in \calB,k\in\N}$ for $A$ and
$(v_{T,l})_{T\in \calD,l\in\N}$ for $B$, where $\calB$ and $\calD$ are the
bornologies of $A$ and $B$ respectively. Since $P$ and $Q$ are essential
modules over $A$ and $B$, there exists a bounded section
$\hat\mu : A \rightarrow P\hatotimes Q $
(resp.~$\hat\nu : B \rightarrow Q\hatotimes P $)
of the composition of $u:P\hatotimes_B Q \rightarrow A$ with the canonical
projection $P\hatotimes Q \rightarrow P\hatotimes_B Q $
(resp.~of $V : Q\hatotimes_A P \rightarrow B$ with
$Q\hatotimes P \rightarrow Q\hatotimes_A P$).
The section $\hat\mu$ (resp.~$\hat\nu$) can even be chosen to be a morphism
of left $A$-modules (resp.~of left $B$-modules).
For every bounded $S\subset A$ and $k\in \N$
(resp.~bounded $T\subset B$ and $l\in \N$) let
\begin{displaymath}
  \sum_{i\in \N} p_{S,k,i} \otimes q_{S,k,i} := \hat\mu (u_{S,k})
  \text{ and }
  \sum_{j\in \N} q_{T,l,j}' \otimes p_{T,l,j}' := \hat\nu (v_{T,l}) .
\end{displaymath}
Then define for every $n \in \N$ a bounded map
$\varrho_n : M \hatotimes A^{\hatotimes n} \rightarrow
(Q \hatotimes_A M \hatotimes_A P) \hatotimes B^{\hatotimes k}$ by
\begin{displaymath}
\begin{split}
 \varrho_n & \, ( m \otimes a_1 \otimes \ldots \otimes a_n ) = \\
 & \lim_{ (S,k) \in \calB \times \N} \sum_{i_0,i_1,\cdots,i_n}
 (q_{S,k,i_0} \otimes m \otimes p_{S,k,i_1}) \otimes
 v (q_{S,k,i_1} \otimes a_1 p_{S,k,i_2}) \otimes\ldots \\
 &\mbox{ } \hspace{50mm} \ldots \otimes
 v (q_{S,k,i_n} \otimes a_n p_{S,k,i_0}) .
\end{split}
\end{displaymath}
Note that this map is well-defined since $(u_{S,k})$ is an approximate
identity, $\hat\mu$ is a bounded morphism of left $A$-modules, and since
$M$ is essential. Analogously we define maps
$\theta_n : (Q \hatotimes_A M \hatotimes_A P) \hatotimes B^{\hatotimes k}
\rightarrow M \hatotimes A^{\hatotimes n} $ by
\begin{displaymath}
\begin{split}
 \theta_n & \, ( q \otimes m \otimes p \otimes b_1 \otimes \ldots \otimes b_n )
 = \\
 & \lim_{(T,l) \in \calD \times \N} \sum_{j_0,j_1,\cdots,j_n}
 u(p_{T,l,j_0}' \otimes q ) \, m \, u( p \otimes q_{T,l,j_1}') \otimes
 u (p_{T,l,j_1}' \otimes b_1 q_{T,l,j_2}') \otimes \ldots\\
 &\mbox{ } \hspace{50mm} \ldots \otimes
 u (p_{T,l,j_n}' \otimes b_n q_{T,l,j_0}') .
\end{split}
\end{displaymath}
Again, convergence is guaranteed by the fact that $A$ and $B$ are quasi-unital
and that $M$ is essential.
By assumption (\ref{Eq:Homup}) on $u$ and $v$ the maps $\varrho$ and $\theta$
are both chain maps.
The composites $\theta\circ \psi$ and $\varrho \circ \theta$ are homotopic to
the identity. A simplicial homotopy between   $\theta \circ \varrho$ and the
identity is given by
\begin{displaymath}
\begin{split}
  h_i & \, (m\otimes a_1 \otimes \ldots \otimes a_n) = \\
  & \lim_{(S,k) \in \calB \times \N \atop (T,l) \in \calD \times \N}
  \sum_{i_0,i_1,\cdots,i_i \atop j_0 , j_1 , \ldots , j_i}
  m \, u (p_{S,k,i_0} \otimes q_{T,l,j_0}')\otimes
  u (p_{T,l,j_0}' \otimes q_{S,k,i_0})\, a_1 \,
  u(p_{S,k,i_1}\otimes q_{T,l,j_1}') \otimes \\
  & \mbox{ } \hspace{3mm} \ldots \otimes
  u (p_{T,l,j_{i-2}}' \otimes
  q_{S,k,i_{i-2}})\, a_{i-1} \, u(p_{S,k,i_{i-1}}\otimes q_{T,l,j_{i-1}}')
  \otimes  \\
  & \mbox{ } \hspace{3mm}
  \otimes  u( p_{T,l,j_{i-1}}' \otimes q_{S,k,i_{i-1}}) \otimes
  a_i \otimes a_{i+1} \otimes \ldots \otimes a_n .
\end{split}
\end{displaymath}
By a straightforward though somewhat lengthy argument (cf.~\cite[Chap.~5]{HanHHMEEG})
one checks that the thus defined $h_i$ are well-defined and form a bounded
simplicial homotopy indeed.
Similarly, one constructs a bounded homotopy between $\psi\circ \theta$
and the identity.

Since $[\varrho] = [\theta]^{-1}$, where $[\varrho]$ and
$[\theta]$ denote the induced maps on the Hochschild homology of
$\calA \rtimes \grp$ resp.~$\calA \rtimes \hgrp$, and since
$\theta$ neither depends on the particular choice of an
approximate identity on  $\calA \rtimes \grp$ nor on the choice of
the elements $p_{S,k,i}$ and $q_{S,k,i}$, one concludes  that
$[\varrho]$ is independent of the particular choice of these data.
Likewise on shows that $[\theta]$ does not depend on the choice of
an approximate identity on $\calA \rtimes \hgrp$ and of the
elements $p_{T,l,j}'$ and $q_{T,l,j}'$. One concludes from this
that $[\theta]$ and $[\varrho]$ are natural isomorphisms. This
finishes the proof of the claim.
\end{proof}

\begin{remark}
If one represents the Hochschild cohomology groups $H^n (A,A)$ as equivalence
classes of bounded extensions
\begin{equation}
\label{Eq:BoundExt}
  0 \rightarrow A \rightarrow E_1 \rightarrow \cdots \rightarrow E_k
  \rightarrow A \rightarrow 0 ,
\end{equation}
the isomorphism $H^\bullet (A,A) \rightarrow H^\bullet (B,B)$ between the
Hochschild
cohomologies of two Morita equivalent bounded algebras is given by tensoring
(\ref{Eq:BoundExt}) with $P$ and $Q$, i.e.~by mapping it to the bounded extension
\begin{displaymath}
  0 \rightarrow B \rightarrow Q \hatotimes_A E_1 \hatotimes_A P \rightarrow \cdots
  \rightarrow Q \hatotimes_A E_k \hatotimes_A  P\rightarrow B \rightarrow 0 ,
\end{displaymath}

\end{remark}

Let us now apply the preceding theorem to the situation of Example
\ref{Ex:OpWeEq}. Then we can prove the following result.
\begin{theorem}
\label{Thm:ConvAlgMorMor}
Assume that $\varphi: \hgrp \hookrightarrow \grp$ is a weak equivalence
of proper \'etale Lie groupoids which also is an open embedding,
and let $(\calA \rtimes \grp,\calA \rtimes \hgrp , P,Q,u,v)$
denote the Morita context from Theorem \ref{Thm:WeEqMor}.
The resulting natural isomorphisms in Hochschild homology and cohomology
are implemented by the following natural chain maps
\begin{displaymath}
\begin{split}
  \varphi_*: \: & C_\bullet (\calA \rtimes \hgrp ,\calA \rtimes \hgrp)
  \rightarrow C_\bullet (\calA \rtimes \grp ,\calA \rtimes \grp), \\
  & a_0\otimes a_1 \otimes \ldots \otimes a_n \mapsto
  \varphi_* (a_0) \otimes\varphi_* (a_1) \otimes \ldots \otimes
  \varphi_* (a_n)
\end{split}
\end{displaymath}
and
\begin{displaymath}
\begin{split}
  \varphi^*: \: & C^\bullet (\calA \rtimes \grp ,\calA \rtimes \grp)
  \rightarrow C^\bullet (\calA \rtimes \hgrp ,\calA \rtimes \hgrp), \\
  & F \mapsto \Big( (\calA \rtimes \hgrp)^{\hatotimes n}
  \ni a_1 \otimes \ldots \otimes a_n \mapsto
  F\big(
  \varphi_* (a_1) \otimes \ldots \otimes \varphi_* (a_n)\big)
  \circ \varphi \Big) .
\end{split}
\end{displaymath}
\end{theorem}
\begin{proof}
We only show the claim in the homology case. The claim for
Hochschild cohomology can be proved by a dual argument.

As above let $A=\calA \rtimes \grp$ and $B=\calA \rtimes \hgrp $.
Choose a locally finite family
$(\psi_k)_{k\in \N}\in \calC^\infty_\text{\tiny \rm cpt}(\hgrp_1)$
of smooth functions $ \psi_k : \hgrp_1 \rightarrow [0,1]$ with compact
and connected support on $\hgrp_0$ such that $(\psi_k^2)$ is a partition
of unity on $\hgrp_0$. Put for $m \in \N$:
\begin{equation}
\label{Eq:shunita}
  \psi_{[m],k}:=
  \begin{cases}
    \sqrt{\sum_{l=0}^m \, \psi_l^2 } & \text{for $k=0$}, \\
    \psi_{k+m}  & \text{for $k\geq 1$} .
  \end{cases}
\end{equation}
and
\begin{equation}
\label{Eq:shunitb}
  u_{[m],k} = \sum_{l=0}^k \big( \psi_{[m],l} \big)^2 .
\end{equation}
Then $(u_{[m],k})_{k\in \N}$ is an approximate unit in $\calA \rtimes \hgrp$
for every $m\in \N$. Denote by $K_{[m]}$ the compact set
$\overline{\big(u_{[m],0}^{-1} (1)\big)^\circ}$, i.e.~the closure of the set
of all points $x\in \hgrp_0$ on a neighborhood of which $u_{[m],0}$
(and thus $\psi_{[m],0}$ also) has value $1$.
Then for every compact $K\subset \hgrp_0$ there exists an $m_K\in\N$ such that
$K \subset K_{[m]}^\circ$ for all $m\geq m_K$. Hence one has
\begin{equation}
   {\psi_{[m],0}}_{|K} =  {u_{[m],0}}_{|K} = 1 \quad
   \text{for all $m\geq m_K$}.
\end{equation}
Let $K_{[m]} = \bigcup_{i=0}^{i_m} K_{[m],i}$ be the decomposition of
$K_{[m]}$ in connected components, i.e.~each $K_{[m],i}$ is compact
and connected and
$K_{[m],i} \neq K_{[m],i'}$ for $i\neq i'$.
By the proof of Theorem \ref{Thm:WeEqMor} one can construct for every
fixed $m\in\N$ a section
\[
  \hat\nu_{[m]} : \calC^\infty_\text{\tiny \rm cpt} (\hgrp_1 )
  \rightarrow \calC^\infty_\text{\tiny \rm cpt}
  (t^{-1} (\varphi (\hgrp_0))
  \hatotimes \calC^\infty_\text{\tiny \rm cpt} (s^{-1}
  (\varphi (\hgrp_0))
\]
and elements
$p_{[m],0,j}' \in \calC^\infty_\text{\tiny \rm cpt} (s^{-1} (\varphi (\hgrp_0))$,
$q_{[m],0,j}' \in \calC^\infty_\text{\tiny \rm cpt} (t^{-1} (\varphi (\hgrp_0))$,
$j= 0,\ldots,j_m$
such that every $p_{[m],0,j}'$ and every $q_{[m],0,j}'$ has connected support
in $\varphi (\hgrp_0)$, such that
\begin{equation}
  \hat\nu_{[m]} (u_{[m],0})  = \sum_{j=0}^{j_m} q_{[m],0,j}' \otimes p_{[m],0,j}'
\end{equation}
and finally such  that
\begin{equation}
\begin{split}
  \big( p_{[m],0,j}' \big)_{| \varphi ( K_{[m],j'} )}  & =
  \begin{cases}
    1 & \text{if $j= j'$}, \\
    0 & \text{else},
  \end{cases}
  \\
  \big( q_{[m],0,i}' \big)_{| \varphi ( K_{[m],j'} )}  & =
  \begin{cases}
    1 & \text{if $j= j'$}, \\
    0 & \text{else}.
  \end{cases}
\end{split}
\end{equation}
Put $ p_{[m],0,j}'=0$ and $ q_{[m],0,j}'=0$ for $j > j_m$ and let
\begin{equation}
  \sum_{j\in\N} q_{[m],k,j}' \otimes p_{[m],k,j}' :=
  \hat\nu_{[m]} (u_{[m],k})  \quad
  \text{for $k\geq 1$}.
\end{equation}
For every fixed $m\in \N$ we now have the data to construct the quasi-isomorphism
$\theta_{[m]} : C_\bullet (\calA \rtimes \hgrp , \calA \rtimes \hgrp ) \rightarrow
 C_\bullet (\calA \rtimes \grp , \calA \rtimes \grp )$
as in the proof of the preceding theorem. Note that the induced
maps  $[\theta_{[m]}]$ between the  Hochschild homology groups all
coincide.

Now let $K\subset \hgrp_0$ be compact and consider
\[
  a_0, a_1 , \ldots , a_n \in \calC^\infty_\text{\tiny \rm cpt}
 (s^{-1} (K ) \cap t^{-1} (K)) \subset \calA \rtimes \hgrp .
\]
Then one  for $m \geq m_k$  by construction:
\begin{displaymath}
\begin{split}
  \theta_{[m]} \,& (a_0\otimes a_1 \otimes \ldots \otimes a_n )  = \\
  & = \sum_{j_0, \ldots ,j_n} \,
  u( p_{[m],0,j_0 }' \otimes a_0 q_{[m],0,j_1}') \otimes
  u (p_{[m],0,j_1}' \otimes a_1 q_{[m],0,j_2}') \otimes \ldots\\
  &\mbox{ } \hspace{50mm} \ldots \otimes
  u (p_{[m],0,j_n}' \otimes a_n q_{[m],0,j_0}')   \\
  & = \varphi_* (a_0) \otimes\varphi_* (a_1) \otimes \ldots \otimes
  \varphi_* (a_n) .
\end{split}
\end{displaymath}
Since $\calA \rtimes \hgrp$ is the inductive limit of the subspaces
$\calC^\infty_\text{\tiny \rm cpt} (s^{-1} (K ) \cap t^{-1} (K))$ the claim follows.
\end{proof}


\bibliographystyle{alpha}

\end{document}